\documentclass{article}

\usepackage{verbatim}

\usepackage{amsfonts}

\usepackage{amsthm}

\usepackage{amssymb}

\usepackage{amsmath}

\usepackage{mathabx}

\usepackage{enumerate}

\usepackage[all]{xy}

\usepackage{graphicx}

\usepackage{centernot}

\usepackage{mathtools}

\usepackage{stmaryrd}

\usepackage[pagebackref,colorlinks,linkcolor=blue,citecolor=blue,urlcolor=blue,hypertexnames=true]{hyperref}

\usepackage[pagebackref,hypertexnames=true]{hyperref}

\newcommand{\N}{\mathbb{N}}

\newcommand{\Z}{\mathbb{Z}}

\newcommand{\R}{\mathbb{R}}

\newcommand{\Q}{\mathbb{Q}}

\newcommand{\C}{\mathbb{C}}

\newcommand{\T}{\mathbb{T}}

\newcommand{\Hawaii}{Hawai\kern.05em`\kern.05em\relax i}

\theoremstyle{plain}
\newtheorem{theorem}{Theorem}[section]
\newtheorem{lemma}[theorem]{Lemma}
\newtheorem{corollary}[theorem]{Corollary}
\newtheorem{proposition}[theorem]{Proposition}

\newtheorem{question}[theorem]{Question}
\newtheorem{definition-theorem}[theorem]{Definition / Theorem}

\newtheorem*{conjecture*}{Conjecture}
\newtheorem*{theorem*}{Theorem}

\theoremstyle{definition}
\newtheorem{definition}[theorem]{Definition}
\newtheorem{example}[theorem]{Example}

\theoremstyle{remark}
\newtheorem{remark}[theorem]{Remark}

\newtheorem*{example*}{Example}  
\newtheorem*{remark*}{Remark}

\begin{document}
\title{Conditional representation stability, classification of $*$-homomorphisms, and relative eta invariants}
\author{Rufus Willett}

\maketitle

\abstract{A quasi-representation is a map from a group into a finite-dimensional unitary group (or similar object) that approximately satisfies the relations needed to be a representation.  This paper studies whether quasi-representations can be approximated in operator norm, and on finite subsets of the group, by honest representations.  Starting with work of Kazhdan and Voiculescu, there are known topological obstructions to the existence of such approximations.

The purpose of this paper is to explore whether approximation is possible if the known obstructions vanish, partially generalizing work of Gong-Lin and Eilers-Loring-Pedersen for the free abelian group of rank two, and the Klein bottle group.  We show that this is possible, at least in a weak sense, for some `low-dimensional' groups, including many one-relator groups, and fundamental groups of manifolds of dimension at most three.

The techniques used in the paper are $K$-theoretic: they have their origin in Baum-Connes-Kasparov type assembly maps, and in the Elliott program to classify $C^*$-algebras; Kasparov's bivariant KK-theory is a crucial tool.  The key new technical ingredients are: a stable uniqueness theorem in the sense of Dadarlat-Eilers and Lin that works for non-exact $C^*$-algebras; and an analysis of maps on $K$-theory with finite coefficients in terms of the relative eta invariants of Atiyah-Patodi-Singer.  Although the proofs go through $K$-theoretic machinery, results of Dadarlat can be used to reexpress the main theorems in elementary topological terms.}

\tableofcontents

\section{Introduction}

The goal of this paper is to study an instance of the following general question: if one has an approximate solution to an equation, must it be close to an actual solution?  

We study this in the context of unitary representations of discrete groups, and where the approximations take place in the operator norm.  Starting with work of Voiculescu \cite{Voiculescu:1983km} and Kazhdan \cite{Kazhdan:1982aa}, it is known that there are topological obstructions to a positive answer.  The purpose of the current paper is to explore what happens when the known obstructions vanish: specifically, is vanishing sufficient for an approximation by an actual solution to exist?

\subsection{Almost commuting unitaries}

The main topic of this paper is finite dimensional unitary representations, and approximate representations, of discrete groups.  However, we will start with a more basic question about matrices that motivates our main questions and results; informally, it asks if any pair of approximately commuting unitary matrices can be approximated by an actually commuting pair.

All norms in this paper are operator norms on spaces of bounded operators on (possibly finite dimensional) Hilbert spaces.

\begin{question}
For any $\epsilon>0$, does there exist $\delta>0$ with the following property? 

For any $n$\footnote{The order of quantifiers is crucial: if $\delta$ is allowed to depend on $n$, the question is simpler, and the answer is different.}, if $u,v\in M_n(\C)$ are unitary matrices satisfying 
$$
\|uv-vu\|<\delta,
$$
then there are unitary matrices $u',v'\in M_n(\C)$ satisfying 
$$
\|u'-u\|<\epsilon,\quad \|v'-v\|<\epsilon, \quad \text{and}\quad u'v'=v'u'.
$$
\end{question}

The question is completely answered by the following result.  To state it, we need some notation.  Let $u,v\in M_n(\C)$ be unitary matrices such that $\|uv-vu\|<1$.  The function 
\begin{equation}\label{basic path}
[0,1]\owns t\mapsto \text{det}\big(t+(1-t)uvu^{-1}v^{-1}\big)
\end{equation}
then defines a path in the complex numbers that starts and ends at $1$, and avoids $0$.  Hence we may define $w(u,v)$ to be the winding number of this path.  

\begin{theorem}[Several contributors - see below]\label{alm com}
\begin{enumerate}[(i)]
\item \label{no go} There is $\epsilon>0$ such that for any $n$, if $u,v\in M_n(\C)$ are unitary matrices such that $\|uv-vu\|<1$ and $w(u,v)\neq 0$, then there do not exist unitary matrices $u',v'\in M_n(\C)$ with $\|u'-u\|<\epsilon$, $\|v'-v\|<\epsilon$ and $u'v'=v'u'$. 
\item \label{go} For any $\epsilon>0$ there exists $\delta>0$ such that for any $n$, if $u,v\in M_n(\C)$ are unitary matrices satisfying $\|uv-vu\|<\delta$ \text{and $w(u,v)=0$}, then there are unitary matrices $u'$ and $v'$ satisfying $\|u'-u\|<\epsilon$, $\|v'-v\|<\epsilon$ and $u'v'=v'u'$.
\end{enumerate}
\noindent Moreover, for any $\delta>0$ and $k\in \Z$ there exist $n$ and unitary matrices $u,v\in M_n(\C)$ with $\|uv-vu\|<\delta$ and $w(u,v)=k$.
\end{theorem}

In brief, Theorem \ref{alm com} says there is a robust integer-valued topological invariant $w(u,v)$ that completely determines whether an almost solution of the equation ``$uv=vu$'' is close to an actual solution.

Part \eqref{no go} has its roots in work of Voiculescu \cite{Voiculescu:1983km}.  The interpretation in terms of winding numbers comes from work of Loring \cite{Loring:1985ud} and of Exel and Loring \cite{Exel:1991aa} in this context, and also of Kazhdan in a slightly different setting \cite{Kazhdan:1982aa}.  See also \cite{Enders:2023aa} for a far-reaching recent approach to this and related questions.

Part \eqref{go} is due independently to Gong and Lin \cite[Corollary M3]{Gong:1998aa}, and to Eilers, Loring, and Pedersen \cite[Corollary 6.15]{Eilers:1999ln}.  It seems less well-known than part \eqref{no go}.  The reader might compare it to \cite[Main Theorem on page 746]{Arzhantseva:2015aa} (respectively, \cite[Theorem 2]{Glebsky:2010aa}): this establishes a similar result for almost commuting elements of permutation groups with respect to the normalized Hamming distance (respectively, almost commuting unitaries with respect to the normalized Hilbert-Schmidt norm) that holds without any analogue of the vanishing condition on the winding number.

Part \eqref{no go} of Theorem \ref{alm com} has been generalized extensively.  Indeed, it can be reframed as asking whether an approximate unitary representation of the group $\Z^2$ is close to an actual representation.  Generalized from here to other discrete groups, it becomes a fundamental statement about `stability' of group representations: see for example \cite{Arzhantseva:2014aa} and \cite{Thom:2018aa} for surveys.  Generalizations of part \eqref{go} of Theorem \ref{alm com} have not received so much attention: it is the aim of this paper to say something in this direction.

\subsection{Representation stability}

We now reformulate Theorem \ref{alm com} in terms of representation theory and in the process bring some topology into play; this needs some terminology.  Throughout, $H$ denotes a Hilbert space, $\mathcal{B}(H)$ the bounded operators on $H$, and $\mathcal{B}(H)_1$ the closed unit ball of $\mathcal{B}(H)$.   We will mainly be interested in the case $H=\C^n$ so that $\mathcal{B}(H)=M_n(\C)$.

\begin{definition}\label{quasi rep}
Let $\Gamma$ be a discrete group, let $S$ be a subset of $\Gamma$, and let $\epsilon>0$.  An \emph{($S$,$\epsilon$)-representation} of $\Gamma$ is a unital map
$$
\pi:\Gamma\to \mathcal{B}(H)_1 \footnote{It is probably more common in the literature to force a quasi-representation $\phi$ to have values in the group of unitary operators on the Hilbert space.  The two definitions are equivalent up to an approximation: see Lemma \ref{uni qr} below.  The extra flexibility allowed by our definition is important to us mainly when we need to consider ucp quasi-representations as in Definition \ref{ucp qr} below: see Proposition \ref{ucp extend}, part \eqref{ucp is rep} below, which says that the theory of unitary-valued ucp quasi-representations is essentially trivial.}
$$
such that 
$$
\|\pi(s)\pi(t)-\pi(st)\|<\epsilon
$$
for all $s,t\in S$.

If we do not want to specify the pair $(S,\epsilon)$, we will just say that $\phi$ is a \emph{quasi-representation}\footnote{A quasi-representation is therefore just a unital map $\Gamma\to\mathcal{B}(H)_1$.  It may seem a bit silly to introduce terminology for this; the point is to emphasize that we are currently thinking of the map as an approximate representation.}.
\end{definition}

\begin{definition}\label{stable}
A group $\Gamma$ is \emph{stable}\footnote{In some other references (for example, \cite{Eilers:2018ab} or \cite{Dadarlat:384aa}), the property defined here is called ``matricial stability''.  It is also sometimes called ``operator norm stability'' if other norms on matrix algebras are under consideration.  We do not use the additional adjectives in this paper, as we will not study any other kinds of stability.} if for any finite subset $S$ of $\Gamma$ and $\epsilon>0$ there exists a finite subset $T$ of $\Gamma$ and $\delta>0$ such that if $\phi:\Gamma\to M_n(\C)_1$ is a $(T,\delta)$-representation in the sense of Definition \ref{quasi rep}, then there exists a unitary representation $\pi:\Gamma\to M_n(\C)$ such that 
$$
\|\phi(s)-\pi(s)\|<\epsilon
$$
for all $s\in S$.

If $P$ is a property of quasi-representations, then we say that $\Gamma$ is \emph{stable, conditional on $P$} if for any finite subset $S$ of $\Gamma$ and $\epsilon>0$ there exists a finite subset $T$ of $\Gamma$ and $\delta>0$ such that if $\phi:\Gamma\to M_n(\C)_1$ is a $(T,\delta)$-representation satisfying $P$, then there exists a unitary representation $\pi:\Gamma\to M_n(\C)$ such that 
$$
\|\phi(s)-\pi(s)\|<\epsilon
$$
for all $s\in S$.
\end{definition}

Group theoretic analogues of Definition \ref{stable} (with more general `targets' than matrix algebras) can be found in \cite[Definition 3.1]{Arzhantseva:2015aa} and \cite[Definition 1.9]{Chiffre:2018ds}.  Purely $C^*$-algebraic analogues (with more general `domains' than groups), can be found in \cite[Definition 2.2.7 and Proposition 2.2.9]{Eilers:1998aa}.  See \cite[Sections 2.1 and 2.2]{Eilers:2018ab} for a more recent survey of these $C^*$-algebraic notions and their translation to group theory; in particular, see \cite[Proposition 2.16]{Eilers:2018ab} for a reformulation in terms of generators and relations for finitely presented groups.  Compare also \cite[Section 5.2]{Eilers:1999aa} for a purely $C^*$-algebraic analogue of conditional stability, which the authors of  \cite{Eilers:1999aa} call \emph{stability with contingencies}.

Switching to representation theory from the matrices of Theorem \ref{alm com}, we state another illustrative result that was a major motivation for us.  Let $\Z^2$ be the fundamental group of the two-torus $\T^2$, and let $\Z\rtimes \Z$ denote the fundamental group of the Klein bottle $K$.  These spaces can be illustrated as CW complexes in a standard way:
\begin{center}
\includegraphics[width=5cm]{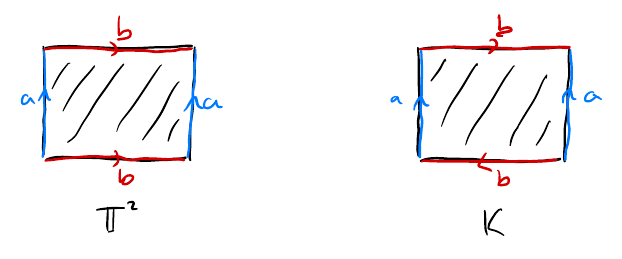}
\end{center} 
This leads to the presentations $\Z^2=\langle a,b\mid aba^{-1}b^{-1}\rangle$ and $\Z\rtimes \Z=\langle a,b\mid aba^{-1}b\rangle$ defined by reading off the edges around the $2$-disk.  The fact that the relation ``$aba^{-1}b^{-1}$'' is a commutator leads to the fact that $H_2(\T^2)=\Z$; as $aba^{-1}b$ is not a commutator, $H_2(K)=0$.  Relatedly, the fact that ``$aba^{-1}b^{-1}$'' is a commutator is important in well-definedness of the winding number of the path in line \eqref{basic path}, while the fact that ``$aba^{-1}b$'' is not a commutator means that the path defined by
$$
[0,1]\owns t\mapsto \text{det}\big(t+(1-t)uvu^{-1}v\big)
$$ 
for unitaries $u,v\in M_n(\C)$ does not usually have a well-defined winding number, even if $uvu^{-1}v$ is very close to $1$, as it does not start and end at the same place.  These connections to (co)homology are already noted in the original papers of Voiculescu \cite[page 431]{Voiculescu:1983km} and Kazhdan \cite[page 321]{Kazhdan:1982aa}.

One has in fact the following result.

\begin{theorem}[Gong-Lin, Eilers-Loring-Pedersen]\label{torus and klein}
\begin{enumerate}[(i)]
\item \label{z2 the} The group $\Z^2$ is stable, conditional on vanishing of the winding number invariant $w(u,v)$. 
\item \label{klein the} The group $\Z\rtimes \Z$ is stable.
\end{enumerate}
\end{theorem}

Part \eqref{z2 the} is just a restatement of Theorem \ref{alm com}, part \eqref{go}.  Part \eqref{klein the} is due to Eilers, Loring, and Pedersen \cite[Theorem 4.3.1 and Corollary 8.2.2; see also table on page 139]{Eilers:1998aa}.

Our original motivation for writing this paper was to establish an analogue of Theorem \ref{torus and klein} for higher genus surfaces.  Here a direct analogue of Theorem \ref{alm com} part \eqref{no go} is known (and essentially due to Kazhdan \cite{Kazhdan:1982aa}), but part \eqref{go} was open.  

Our main results give partial generalizations of Theorem \ref{torus and klein} to several interesting classes of groups which contain the fundamental groups of the torus and Klein bottle: for example many one-relator groups, and fundamental groups of manifolds of dimension at most three.   Moreover, most of the groups in these classes are non-amenable, which requires new arguments.  We set up the necessary terminology in the next section.

\subsection{Basic notions, and surface groups}

We first recall a slight generalization of a notion due to Dadarlat \cite[page 2]{Dadarlat:384aa}, which Dadarlat calls \emph{weak stability}.  We use slightly different terminology both to allow for some variants, and to avoid confusion with other (mutually inconsistent) uses of the terminology ``weak stability'' in the literature on related problems: see for example \cite[Definition 2.10]{Eilers:2018ab} (and also \cite[Section 4.1]{Loring:1997aa}), and \cite[Definition 7.1]{Arzhantseva:2015aa} (and also \cite[Definition 1.5]{Dogon:2023aa}).

\begin{definition}\label{ws}
Let $\Gamma$ be a group, and let $\mathcal{R}$ be a class of finite-dimensional representations of $\Gamma$.  For each finite subset $S$ of $\Gamma$ and $\epsilon>0$, let $\mathcal{Q}(S,\epsilon)$ be a class of quasi-representations, possibly depending on $S$ and $\epsilon$.

The group $\Gamma$ is \emph{$\mathcal{Q}$-$\mathcal{R}$-stable} if for any finite subset $S$ of $\Gamma$ and $\epsilon>0$ there exists a finite subset $T$ of $\Gamma$ and $\delta>0$ such that if $\phi:\Gamma\to M_n(\C)_1$ is a $(T,\delta)$-representation, then there exists a map $\theta:\Gamma\to M_k(\C)$ in $\mathcal{Q}(S,\epsilon)$ and a representation $\pi:\Gamma\to M_{n+k}(\C)$ in $\mathcal{R}$ such that 
$$
\|(\phi(s)\oplus \theta(s))-\pi(s)\|<\epsilon
$$
for all $s\in S$.

If $P$ is a property of quasi-representations, then we say that $\Gamma$ is \emph{$\mathcal{Q}$-$\mathcal{R}$-stable, conditional on $P$} if for any finite subset $S$ of $\Gamma$ and $\epsilon>0$ there exists a finite subset $T$ of $\Gamma$ and $\delta>0$ such that if $\phi:\Gamma\to M_n(\C)_1$ is a $(T,\delta)$-representation satisfying $P$, then there exist then there exists a map $\theta:\Gamma\to M_k(\C)$ in $\mathcal{Q}(S,\epsilon)$ and a representation $\pi:\Gamma\to M_{n+k}(\C)$ in $\mathcal{R}$ such that 
$$
\|(\phi(s)\oplus \theta(s))-\pi(s)\|<\epsilon
$$
for all $s\in S$.

If $\mathcal{Q}=\mathcal{R}$, we just say that $\Gamma$ is \emph{$\mathcal{R}$-stable (conditional on $P$)}.
\end{definition}

\begin{remark}\label{qrstab vs stab}
Stability of $\Gamma$ is the case of $\mathcal{Q}$-$\mathcal{R}$-stability where $\mathcal{Q}$ is the zero quasi-representation on the zero vector space, and $\mathcal{R}$ consists of all finite-dimensional representations.
\end{remark}

\begin{remark}
We will mainly be interested in the case when $\mathcal{Q}$ (as well as $\mathcal{R}$) is a class of honest representations.  Specifically, the main case we will use is the class $\mathcal{R}_q$ of finite-dimensional representations that factor through a finite quotient of $\Gamma$, and $\mathcal{Q}=\mathcal{R}=\mathcal{R}_q$.
\end{remark}

\begin{remark}\label{vfs rem}
The  notion of $\mathcal{Q}$-$\mathcal{R}$-stability contains that of \emph{very flexible stability}\footnote{Thanks to Tatiana Shulman and Francesco Fournier-Facio for explaining this to me.  Being able to make this comparison is the main reason for introducing $\mathcal{Q}$-$\mathcal{R}$ stability rather than just $\mathcal{R}$-stability.}, introduced by Becker and Lubotzky in \cite[Section 4.4]{Becker:2020aa}.  Very flexible stability (for the operator norm) is equivalent to the case of $\mathcal{Q}$-$\mathcal{R}$-stability where $\mathcal{Q}(S,\epsilon)$ is the class of $(S,\epsilon)$-representations, and $\mathcal{R}$ is the class of all finite-dimensional representations.  See \cite[Section 7]{Fournier-Facio:2026aa} for related discussion.
\end{remark}


Let us now state some of our main results.  The following theorem was our original motivation for developing the ideas in this paper.

\begin{theorem}\label{intro surf the}
\begin{enumerate}[(i)]
\item \label{surf or} Let 
$$
\Gamma=\Bigg\langle a_1,...,a_g,b_1,...,b_g~\Big|~ \prod_{i=1}^g [a_i,b_i]\Bigg\rangle
$$
be the fundamental group of a closed orientable surface of genus $g>0$.  For any quasi-representation $\phi:\Gamma\to M_n(\C)$ such that 
$$
\Bigg\|\prod_{i=1}^g[\phi(a_i),\phi(b_i)]-1\Bigg\|<1,
$$
define $w(\phi)\in \Z$ to be the winding number of the path 
\begin{equation}\label{surf wind}
[0,1]\owns t\mapsto \text{det}\Bigg( t+(1-t) \prod_{i=1}^g[\phi(a_i),\phi(b_i)]\Bigg).
\end{equation}
Then $\Gamma$ is $\mathcal{R}_q$-stable, conditional on $w(\phi)=0$.  
\item \label{surf nor} Let 
$$
\Gamma=\Bigg\langle a_1,...,a_{g}~\Big|~ \prod_{i=1}^{g} a_i^2\Bigg\rangle
$$
be the fundamental group a closed non-orientable surface of (non-orientable) genus $g>0$.  Then $\Gamma$ is $\mathcal{R}_q$-stable.
\end{enumerate}
\end{theorem}

\begin{remark}\label{surf ns}
That the vanishing condition in Theorem \ref{intro surf the} part \eqref{surf or} is \emph{necessary} for some sort of stability result is well-known.  This was observed by multiple authors starting (essentially) with the work of Kazhdan \cite{Kazhdan:1982aa}, and more recently made explicit by Dadarlat \cite[Section 4]{Dadarlat:2011kx} and Eilers, Shulman, and S\o{}rensen \cite[Section 4.3]{Eilers:2018ab}.   There have been many other `no go' results along these lines: see for example \cite{Eilers:2018ab,Dadarlat:384aa,Glebe:2024aa,Bader:2023aa}.

Thus part \eqref{surf or} of Theorem \ref{intro surf the} should be viewed as giving a sufficient condition for a quasi-representation of a surface group to be close (in some weak sense) to an actual representation that matches the known necessary condition.
\end{remark}

\begin{remark}\label{r stab rem}
We do not if ``$\mathcal{R}_q$-stable'' can be replaced by ``stable'' in Theorem \ref{intro surf the} other than in the case of the torus or Klein bottle covered by Theorem \ref{torus and klein}, and the elementary case of the projective plane. 
\end{remark}

\begin{remark}\label{llm rem}
Theorem \ref{intro surf the} can be usefully compared to the result of Lazarovich, Levit, and Minsky \cite{Lazarovich:2019aa} on quasi-representations of surface groups with values in permutation groups.  In the language of \cite[Section 1]{Lazarovich:2019aa}, the main result of that paper shows that surface groups are \emph{flexibly stable} for quasi-representations in permutation groups.  Apart from the focus on permutation representations, the result of \cite{Lazarovich:2019aa} differs from ours in two important ways: first it is `absolute', i.e.\ works for all permutation quasi-representations without topological conditions; second the notion of ``flexible stability'' from \cite[Section 1]{Lazarovich:2019aa} is different from $\mathcal{R}_q$-stability as in Definition \ref{ws} in that flexible stability only allows block sum with an auxiliary trivial representation which is `small' relative to the dimension of the quasi-representation one starts with (Theorem \ref{intro surf the} gives no control on the size of the auxiliary representation $\theta$ appearing in Definition \ref{ws}).  
\end{remark}

\subsection{Free-by-cyclic groups, one relator groups, and three manifold groups}

The winding number invariant appearing in Theorem \ref{intro surf the} is due to Kazhdan \cite{Kazhdan:1982aa}.  Dadarlat gives alternative interpretations in terms of index theory and Chern classes in \cite[Section 4]{Dadarlat:2011kx}; these results, and Dadarlat's more general results in \cite{Dadarlat:2022aa}, are important for the proof of Theorem \ref{intro surf the}, and of the related results we give below.  For these theorems, we recall that if $c\in H_2(\Gamma)$ is a group homology class, then Dadarlat \cite{Dadarlat:2022aa} shows how to define a winding number invariant $w(c,\cdot)$ of (suitable) quasi-representations: see Definition \ref{dad wn} and Theorem \ref{dad 2d} below for details.

Recall that a \emph{free-by-cyclic group} is a group of the form $F\rtimes_{\alpha} \Z$, where $F$ is a finite rank free group, and $\Z$ acts on $F$ by an automorphism $\alpha$; note that this class of groups includes the torus and Klein bottle groups of Theorem \ref{torus and klein} as the two special cases where $F=\Z$ and $\alpha$ is either the trivial or non-trivial automorphism.

\begin{theorem}\label{fbc intro}
Let $\Gamma=F\rtimes_\alpha \Z$ be a free-by-cyclic group.  

Let $n$ be the rank of $F$, and let $\phi_*:\R^n\to \R^n$ be the map induced by $\phi$ on the first real homology group.  Then $H_2(\Gamma)$ is free abelian, with rank equal to the multiplicity $m$ of the eigenvalue $1$ of $\alpha_*$; choose a basis $c_1,...,c_m$ for $H_2(\Gamma)$.  

The group $\Gamma$ is $\mathcal{R}_q$-stable, conditional on vanishing of the winding number invariants $w(c_i,\cdot)$ for all $i\in \{1,...,m\}$.  In particular, if $\alpha_*$ does not have 1 as an eigenvalue, then $\Gamma$ is $\mathcal{R}_q$-stable.
\end{theorem}

The next theorem looks at a large class of one-relator groups.   See Remark \ref{o r summary} below.

\begin{theorem}\label{o r intro}
Let $\Gamma=\langle s_1,...,s_m\mid r\rangle$ be a torsion-free one-relator group, and assume that any two non-trivial elements of $\Gamma$ generate a free group.  
\begin{enumerate}[(i)]
\item If $r$ is in the commutator subgroup of the free group on $s_1,...,s_n$, then $H_2(\Gamma)\cong \Z$.  Let $c\in H_2(\Gamma)$ be a generator, and let $w(c,\cdot)$ be the associated winding number invariant.  Then $\Gamma$ is $\mathcal{R}_q$-stable, conditional on vanishing of $w(c,\cdot)$.
\item  If $r$ is not in the commutator subgroup of the free group on $s_1,...,s_n$, then $\Gamma$ is weakly stable.
\end{enumerate}
\end{theorem}

Let us conclude this section with a theorem for three-manifold groups.  This is a bit less general than our main result on three-manifold groups: see Theorem \ref{three man} for that.

\begin{theorem}\label{intro 3man the}
Let $\Gamma$ be the fundamental group of a closed aspherical three manifold $M$.  Choose a decomposition $H_2(M)=F\oplus T$ into the sum of a free subgroup and a torsion subgroup\footnote{If $M$ is orientable, $T$ is trivial, but not in general.}, choose a basis $c_1,...,c_m$ for the free summand $F$, and let $w(c_i,\cdot)$ be the associated winding number invariants.  Then $\Gamma$ is $\mathcal{R}_q$-stable, conditional on vanishing of $w(c_i,\cdot)$ for $i\in \{1,...,m\}$.  In particular, if $H_2(M)$ is torsion (or trivial) then $\Gamma$ is $\mathcal{R}_q$-stable.
\end{theorem}

\begin{remark}\label{gen ns}
Analogously to Remark \ref{surf ns} above, the fact that the vanishing conditions in Theorems \ref{fbc intro}, \ref{o r intro}, and \ref{intro 3man the} are necessary for $\mathcal{R}_q$-stability to hold is well-known.  This follows from the main result of \cite{Dadarlat:2022aa}, and also from \cite[Theorem 4.10]{Eilers:2018ab}.
\end{remark}

\begin{remark}
The conclusion of Theorem \ref{intro 3man the} cannot be strengthened to (conditional) stability (compare Remark \ref{qrstab vs stab} above).  Indeed, Theorem \ref{intro 3man the} applies to $\Z^3$, but this group is known not to be conditionally stable with respect to the winding number vanishing condition in Theorem \ref{intro 3man the} by \cite[Theorem 4.2]{Gong:1998aa}; compare also \cite[Theorem 3.13]{Eilers:2018ab} on crystallographic groups in this regard.  

We do not know if the conclusions of Theorems \ref{intro surf the}, \ref{fbc intro}, or \ref{o r intro} can be strengthened to (conditional) stability.
\end{remark}

\begin{remark}\label{vfx rem 2}
Analogously to Remark \ref{vfs rem} above, all the groups appearing in Theorems \ref{intro surf the}, \ref{fbc intro}, \ref{o r intro}, and \ref{intro 3man the} are known to be very flexibly stable, regardless of any topological conditions: see \cite[Section 7]{Fournier-Facio:2026aa}.
\end{remark}

\subsection{$C^*$-algebra $K$-theory and stable uniqueness}\label{k stab sec}

We stated our results in the previous section in terms of winding numbers as this seemed more elementary and explicit.  Underlying these winding numbers are $K$-theory classes, however, and it is probably fair to say that $K$-theory is more directly relevant to the problem\footnote{Our opinions (arguably, biases) on this are influenced by Atiyah: compare \cite[page 245]{Atiyah:1967aa}.}: this is evidenced by Loring's analysis \cite{Loring:1985ud} of Voiculescu's example \cite{Voiculescu:1983km} in terms of Bott periodicity, and computations showing that one can deduce Bott periodicity from properties of almost commuting matrices \cite{Willett:2020ab}.

We will thus work with $K$-theory of $C^*$-algebras, as well as the dual $K$-homology theory.  See \cite{Rordam:2000mz} for background on $C^*$-algebra $K$-theory, and \cite{Higson:2000bs} or \cite{Blackadar:1998yq} for background on $C^*$-algebra ($K$-theory and) $K$-homology.  

Recall that a \emph{$C^*$-algebra} is a norm-closed and $*$-closed subalgebra of the $*$-algebra of bounded operator on a Hilbert space.  Let $C^*(\Gamma)$ be the maximal, or full, group $C^*$-algebra of $\Gamma$, i.e.\ the completion of the complex group algebra $\C[\Gamma]$ in the norm given by the supremum of the norms in all unitary representations of $\C[\Gamma]$.  Then $C^*(\Gamma)$ has the universal property that any unitary representation of $\Gamma$ extends uniquely by linearity and continuity to a $*$-representation of $C^*(\Gamma)$.   Let $K_0(C^*(\Gamma))$ be the (algebraic) $K_0$ group of $C^*(\Gamma)$, and let $K_0(\C)$ be the $K_0$ group of the complex numbers (which is isomorphic to $\Z$); by Morita invariance, $K_0(\C)$ agrees with $K_0(M_n(\C))$ for all $n$.  

Now, given a quasi-representation $\phi$, we aim to:
\begin{enumerate}[(i)]
\item \label{strati} show that it in some sense defines an element 
$$\phi_*\in\text{Hom}(K_0(C^*(\Gamma)),K_0(\C))~;$$
\item \label{stratii} show that this element is equivalent to an honest representation $\pi$ (in an appropriate weak sense) as long as 
$$
\phi_*=\pi_*\quad  \text{in}\quad \text{Hom}(K_0(C^*(\Gamma),K_0(\C)).
$$
\end{enumerate}
Such theorems exist in the $C^*$-algebra literature as part of the classification program for simple nuclear $C^*$-algebras: see for example \cite{Elliott:1995dq} for the set-up of this program, and \cite{White:2022aa} for a survey of recent progress.  They are called \emph{stable\footnote{``Stable'' is here used in the $K$-theoretic sense of `up to block sum with something'; it is not directly connected to a group being stable as in Definition \ref{stable}.} uniqueness theorems} after work of Lin \cite{Lin:2002aa} and Dadarlat-Eilers \cite[Section 4]{Dadarlat:2001aa}.  There are numerous technicalities involved here: indeed, to establish \eqref{strati} one seems to need additional positivity properties of $\phi$; and to deal with \eqref{stratii}, $K$-theory with integral coefficients is not enough, as one instead needs to work with integer coefficients and with all finite coefficients at once.

Now, if we `only' wanted to work with amenable groups, one could adapt the stable uniqueness theorems of Dadarlat-Eilers \cite[Section 4]{Dadarlat:2001aa} to a theorem of the necessary form.  Indeed, this is essentially what is done in \cite[Theorem 1.5]{Dadarlat:384aa}. 

However, this will not suffice for us as our key examples (such as those in Theorems \ref{intro surf the}, \ref{fbc intro}, \ref{o r intro}, and \ref{intro 3man the}) involve non-amenable groups.  The existing techniques do not seem to work in this case as the maximal group $C^*$-algebra $C^*(\Gamma)$ will not even be \emph{exact}\footnote{Exactness of $C^*(\Gamma)$ is a stronger property than the more widely studied exactness of the `reduced' $C^*$-algebra $C^*_r(\Gamma)$ and should not be confused with this.  Indeed, exactness of $C^*_r(\Gamma)$ is equivalent to Yu's \emph{property A} (see \cite[Definition 2.1]{Yu:200ve}) for $\Gamma$ by the main result of \cite{Ozawa:2000th}, while for residually finite $\Gamma$, exactness of $C^*(\Gamma)$ is equivalent to amenability for $\Gamma$ by \cite[Theorem 7.5]{Kirchberg:1993aa} (see also \cite[Proposition 3.7.11]{Brown:2008qy}).}, a technical assumption needed to establish the stable uniqueness theorems: compare for example \cite[Theorem 4.15]{Dadarlat:2001aa} and \cite[Theorem 5.4]{Dadarlat:384aa}.   We thus give a different approach to stable uniqueness theorems based on a new model of $K$-homology \cite{Willett:2020aa} established by the author and Yu; this has the advantage that it allows one to work directly with approximate representations (under additional assumptions).  This new stable uniqueness theorem is probably the most technically novel aspect of the current paper: see Theorem \ref{stab main} below

We now state the main stable uniqueness result in the context of group $C^*$-algebras.  Unfortunately, this needs a bit of an alphabet soup of technical conditions on $\Gamma$: FD, UCT, and LLP from \cite{Lubotzky:2004xw}, \cite{Rosenberg:1987bh}, and \cite{Kirchberg:1993aa} respectively.  For now we just note that all the groups in Theorems \ref{intro surf the}, \ref{fbc intro}, \ref{o r intro}, and \ref{intro 3man the} above satisfy all of these conditions and all are satisfied by amenable residually finite groups.  We refer forward in the paper for details: see Section \ref{main sec}.

Here is a simplified version of the main theorem where we make an additional assumption on the odd-dimensional $K$-theory group $K_1(C^*(\Gamma))$ to obviate the need for $K$-theory with finite coefficients: see Theorem \ref{gp main} below for the full version and Theorem \ref{stab main} for the purely $C^*$-algebraic theorem that gives rise to it.




\begin{theorem}\label{intro gp main 2}
Let $\Gamma$ be a countable discrete group.   Assume moreover that $\Gamma$ is FD, UCT, and LLP, and that $K_1(C^*(\Gamma))$ is torsion-free.

Then for any finite subset $S$ of $\Gamma$ and any $\epsilon>0$ there exist a finite subset $T$ of $\Gamma$, $\delta>0$, and a finite subset $P$ of $K_0(C^*(\Gamma))$ with the following property.

Let $\phi:C^*(\Gamma)\to M_n(\C)$ be a $(T,\delta)$-representation in the sense of Definition \ref{quasi rep}.  Then $\phi$ induces a well-defined `partial homomorphism' from $P$ to $K_0(\C)$.  Moreover, if $\pi:\Gamma\to M_n(\C)$ is a finite-dimensional unitary representation such that the induced map $\pi_*:K_0(C^*(\Gamma))\to K_0(\C)$ agrees with $\phi_*$ on $P$, then there is a unitary representation $\theta:\Gamma\to M_k(\C)$ that factors through a finite quotient of $\Gamma$, and a unitary $u\in M_{n+k}(\C)$ such that 
$$
\|u(\phi(s)\oplus \theta(s))u^*-(\pi(s)\oplus \theta(s))\|<\epsilon
$$
for all $s\in S$. 
\end{theorem}


\subsection{The Baum-Connes conjecture and index theory}

In order to deduce Theorems \ref{intro surf the}, \ref{fbc intro}, \ref{o r intro}, and \ref{intro 3man the} from Theorem \ref{intro gp main 2}, we need to show that for any quasi-representation $\phi:\Gamma\to M_n(\C)$ and finite subset $P$ of $K_0(C^*(\Gamma))$ satisfying appropriate conditions, there exists an honest representation $\pi:\Gamma\to M_n(\C)$ such that the map $\pi_*:K_0(C^*(\Gamma))\to K_0(\C)$ induced by $\pi$ agrees with the partial homomorphism $\phi_*:P\to K_0(\C)$ induced by $\phi$.

Assuming that $\Gamma$ is torsion free, the key ingredient for this is the \emph{Baum-Connes-Kasparov assembly map} \cite{Kasparov:1988dw,Baum:1994pr}
\begin{equation}\label{bc}
\mu:RK_*(B\Gamma)\to K_*(C^*(\Gamma)),
\end{equation}
which relates the topologically defined $K$-homology group $RK_*(B\Gamma)$ with the $K$-theory of group $C^*$-algebra; see \cite{Aparicio:2020aa} for a survey on this.  Following Baum-Douglas \cite{Baum:1980pt}, the left hand side of line \eqref{bc} above consists very roughly of continuous maps from closed manifolds (plus some orientation and bundle data) to the classifying space $B\Gamma$ modulo appropriate equivalence relations.  If $\mu$ is an isomorphism, this in principle allows one to compute the map $\pi_*:K_0(C^*(\Gamma))\to K_0(\C)$ in terms of index theory pairings between Dirac operators on the manifolds making up cycles for $RK_*(B\Gamma)$ with the flat bundle associated to the representation.  

Much of this analysis has been done in a beautiful series of papers by Dadarlat (see for example \cite{Dadarlat:2011kx,Dadarlat:384aa,Dadarlat:2022aa}), although what we need to understand the maps $\pi_*:K_0(C^*(\Gamma))\to K_0(\C)$ is relatively simple and essentially part of the folklore of the subject.

The main new index-theoretic observations we make concern finite coefficients.  Indeed, we have mentioned above that one should be working not just with the $K$-theory group $K_0(C^*(\Gamma))$ with integer coefficients, but also with $K$-theory with finite coefficients $K_0(C^*(\Gamma);\Z/n)$, as introduced by Schochet \cite{Schochet:1984ab}.  The assumption in Theorem \ref{intro gp main 2} that $K_1(C^*(\Gamma))$ is torsion free is made precisely as it allows us to avoid this issue.

In general, however, given a quasi-representation $\phi:\Gamma\to M_n(\C)$ it also induces partially defined maps $\phi_*$ on $K$-theory with finite coefficients, and we need to find a representation $\pi:\Gamma\to M_n(\C)$ such that $\pi_*:K_0(C^*(\Gamma);\Z/n)\to K_0(\C;\Z/n)$ agrees with $\phi_*$ (where the latter is defined) for all $n\geq 2$.  We do this computation using the \emph{relative eta invariants} of Atiyah, Patodi, and Singer \cite{Atiyah:1975aa}.  Very roughly, relative eta invariants measure how the spectrum of a differential operator shifts when it is `twisted' by a finite-dimensional representation.  The key technical results are Theorem \ref{eta inv} which relates the map on $K$-theory with finite coefficients induced by a representation to relative eta invariants; and Theorem \ref{match the} which uses this to show that the map induced on $K$-theory, including $K$-theory with finite coefficients, can be `matched' by the map induced by an honest representation, under appropriate assumptions.  It is important here that all the groups in Theorems \ref{intro surf the}, \ref{fbc intro}, \ref{o r intro}, and \ref{intro 3man the} have classifying spaces of dimension at most three: our methods using relative eta invariants are currently limited to low dimensions.

After our new stable uniqueness theorem, this application of relative eta invariants to representation theory is the second main technical novelty in the current paper.  We suspect, however, that some of the computations we need here may be know to experts: in particular Theorem \ref{eta inv} is in some sense implicit in the original three papers of Atiyah, Patodi, and Singer \cite{Atiyah:1975ab,Atiyah:1975aa,Atiyah:1976aa}.

\subsection{Further questions}

Here we summarize (a small selection of the!) issues that we do not even start to address in this paper; we mention these mainly to highlight problems that we consider interesting for future work.

One of the most interesting applications of representation stability is in obstructing approximation results for groups.  The paradigm here, following \cite{Chiffre:2018ds}, proceeds roughly as follows.  Assume one wants to show that a particular group $\Gamma$ does not admit a separating family of suitably good quasi-representations.  Roughly, this follows if $\Gamma$ is stable, and also if it does not admit a separating family of honest representations.  Our methods \emph{require} that $\Gamma$ has a separating family of honest representations, and so are not useful\footnote{Probably not useful: it is conceivable there could be useful `relative versions' or similar.} for results along these lines.

A second limitation to our methods is the reliance on low-dimensionality assumptions (on classifying spaces) for concrete examples.  There seems to be a lot of potential for higher-dimensional computations, but this would require new ideas, particularly around the computation of relative eta invariants, or other approaches to $K$-theory with finite coefficients.  We do at least get `rational' results that work regardless of dimension: see Theorem \ref{high main} below.

Another issue we have nothing to say about is what happens in the presence of finite subgroups.  There is probably quite a lot that can be said here, perhaps in terms of Lott's \emph{delocalized invariants} \cite{Lott:1999aa} or related machinery.  We did not attempt to address this here, but this is only due to our limited knowledge (and energy), not because we are aware of serious obstructions.  There is a partial obstruction in that our methods require finite-dimensional classifying spaces for concrete examples, and the classifying space of a group with torsion is always infinite-dimensional.  However, it is plausibly possible to get around this by considering the classifying space for proper actions or something similar.  This again seems an interesting avenue for further work: indeed, Weinberger, Wu, and Yu have announced results in this direction, using fairly different methods from ours.


\subsection{Outline of the paper}

Section \ref{ucp sec} discusses quasi-representations in more detail: in particular, we introduce ucp quasi-representations and discuss the role of Kirchberg's LLP in allowing the approximation of general quasi-representations by ucp quasi-representations.  

Section \ref{ckh sec} recalls the controlled $K$-homology groups from \cite{Willett:2020aa}, and relates them to homomorphisms between $K$-theory groups using the universal multicoefficient theorem of Dadarlat and Loring \cite{Dadarlat:1996aa} and Dadarlat's work on the topology of $KK$-theory \cite{Dadarlat:2005aa}.  This is then used to establish our $C^*$-algebraic stable uniqueness theorem in Section \ref{su sec}.  Sections \ref{ckh sec} and \ref{su sec} are written in the language of abstract $C^*$-algebras; in Section \ref{main sec} we specialize to group $C^*$-algebras and state our main $K$-theoretic result for quasi-representations of groups.  Section \ref{main sec} also surveys when the various assumptions (FD, UCT, LLP) going into the main theorem are known for some concrete classes of groups: all three assumptions are known for amenable groups (for FD one also needs residually finite) and somewhat beyond this class, but the state of knowledge is only partial, and there is strong evidence all are obstructed by appropriate forms of property (T).

Section \ref{bc sec} brings the Baum-Connes conjecture and index-theoretic methods into play.  The main technical result is Theorem \ref{match the} which lets us match the $K$-theoretic data associated to a quasi-representation with $K$-theoretic data from an honest representation.  Finally, Section \ref{example sec} first establishes a general `rational' stability result, and then specializes to low-dimensional (meaning the classifying space is low-dimensional) examples where explicit computations are possible: in particular, it establishes Theorems \ref{intro surf the}, \ref{fbc intro}, \ref{o r intro}, and \ref{intro 3man the} above.

\subsection{Notation and conventions}

We write $K_*(A):=K_0(A)\oplus K_1(A)$ for the $\Z/2$-graded $K$-theory group of a $C^*$-algebra $A$, and similarly $K^*(A):=K^0(A)\oplus K^1(A)$ denotes the $\Z/2$-graded $K$-homology group.  For a compact metrizable space $X$, we write $K^*(X)$ and $K_*(X)$ for $K$-theory and $K$-homology; thus $K^*(X)=K_*(C(X))$ and $K_*(X)=K^*(C(X))$ where $C(X)$ is the $C^*$-algebra of continuous complex-valued functions on $X$ (the position of the stars records whether the corresponding functor on spaces or $C^*$-algebras is covariant or contravariant).  A homomorphism between $\Z/2$-graded groups is \emph{graded} if it splits as a direct sum of homomorphisms that preserve the components.

If $A$ is a unital $C^*$-algebra, we write $1_n$ and $0_n$ for the unit and zero element of $M_n(A)$.  For a $C^*$-algebra $A$, $a\in M_n(A)$ and $b\in M_m(A)$, define $a\oplus b:=\begin{psmallmatrix} a & 0 \\ 0 & b\end{psmallmatrix}\in M_{n+m}(A)$ for the usual block sum matrix.

Tensor products between $C^*$-algebras will always be the minimal or spatial tensor product: see \cite[Chapter 3]{Brown:2008qy} for background.  At least one of the $C^*$-algebras in any tensor product we take will typically be nuclear, however, so the choice of tensor product does not matter.

We always assume the groups $\Gamma$ we work with are discrete, and generally also assume they are countable.  Countability is probably not strictly necessary -- we would guess that most arguments could be extended to the general case through taking appropriate limits -- but the countable case covers all the examples we are interested in, and doing everything in general would require a lot more pedantry with $KK$-theory and the UCT (amongst other things) than seemed likely to be helpful to the reader.  

A subset $S$ of a group is called \emph{symmetric} if it is closed under taking inverses.  A subset $X$ of a $C^*$-algebra is \emph{symmetric} if it is closed under taking adjoints.

Throughout the paper, $\tau:\Gamma\to \C$ is the trivial representation of a group $\Gamma$, and $\tau^{(n)}:\Gamma\to M_n(\C)$ is the $n$-dimensional trivial representation; we will occasionally also write $\tau$ for $\tau^{(n)}$ if there does not seem to be a risk of confusion.  We also use the same notation for the maps induced on $C^*(\Gamma)$ by these representations.

We will sometimes need to use Kasparov's $KK$-theory: see \cite[Chapter VIII]{Blackadar:1998yq} or \cite[Section 2]{Kasparov:1988dw} for background.  For any pair of separable $C^*$-algebras $A$, $B$, $KK$-theory assigns an abelian group $KK(A,B)$ in such a way that $KK(\C,A)$ and $KK(A,\C)$ canonically identify with the $K$-theory $K_0(A)$ and $K$-homology $K^0(A)$ respectvely.  If $C$, $D$, $E$ are separable $C^*$-algebras and ``$\otimes$'' denotes the spatial tensor product of $C^*$-algebras or tensor product of abelian groups as appropriate, then there is a pairing 
\begin{equation}\label{kk pair}
KK(A,B\otimes C)\otimes KK(C\otimes D,E)\to KK(A\otimes D,B\otimes E)
\end{equation}
with strong formal properties including appropriate versions of associativity: see \cite[Definition 2.12 and Theorem 2.14]{Kasparov:1988dw} or \cite[Section 18.9]{Blackadar:1998yq}.  The most important of these pairings for us will be the pairing
$$
KK(\C,A)\otimes KK(A,\C)\to KK(\C,\C)=\Z,
$$
between $K$-theory and $K$-homology.  Another important special case for us are the pairings
$$
KK(\C,A\otimes O_{n+1})\otimes KK(A,\C)\to KK(\C,O_{n+1})= \Z/n
$$
where $O_{n+1}$ is the Cuntz algebra of \cite{Cuntz:1977aa}.  These will be useful for us when we discuss $K$-theory with $\Z/n$ coefficients, as for us $KK(\C,A\otimes O_{n+1})$ equals the $K$-theory group of $A$ with $\Z/n$ coefficients, denoted $K_0(A;\Z/n)$.  Finally, we note that the group $KK(A,B)$ is often denoted $KK_0(A,B)$, and thought of as the \emph{even} $KK$-theory group; there is also an \emph{odd} group $KK_1(A,B)$, and a pairing 
$$
KK_i(A,B\otimes C)\otimes KK_j(C\otimes D,E)\to KK_{i+j \text{ mod } 2}(A\otimes D,B\otimes E)
$$
defined for any separable $C^*$-algebras.

\subsection{Acknowledgments}

I would like to thank Goulnara Arzhantseva, Jos\'{e} Carri\'{o}n, Marius Dadarlat, Robin Deeley, Alon Dogon, S\o{}ren Eilers, Francesco Fournier-Facio, James Gabe, Guihua Gong, Asaf Hadari, Huaxin Lin, Christopher Schafhauser, Dimitri Shlyakhtenko, Tatiana Shulman, Andreas Thom, Jianchao Wu, and Guoliang Yu for helpful comments, explanations, and / or corrections during the writing of this paper. I am also grateful to the anonymous referee for several useful comments. 

I would like to thank the Isaac Newton Institute for Mathematical Sciences, Cambridge, for support and hospitality during the program Operators, Graphs, Groups: this led to a collaboration \cite{Fournier-Facio:2026aa} with Francesco Fournier-Facio that (among other things) greatly improved the range of validity of the main results of this paper as compared to earlier drafts.  

Support from the US NSF (DMS 2247968) and Simons Foundation (MP-TSM-00002363) is also gratefully acknowledged.

\section{Ucp quasi-representations and the local lifting property}\label{ucp sec}

In this section we discuss some general considerations about the relationship between different types of quasi-representation.  We believe the results here are essentially folklore; we give proofs where we could not find a result in the literature.

The following lemma will not be used directly in the paper: the point of including it is to relate our notion of unit ball-valued quasi-representations from Definition \ref{quasi rep} with more the more usual unitary-valued version.

\begin{lemma}\label{uni qr}
Let $\Gamma$ be a discrete group, let $S\subseteq \Gamma$ be symmetric, and let $\epsilon\in (0,1)$.  Let $\phi:\Gamma\to\mathcal{B}(H)_1$ be an $(S,\epsilon)$-representation as in Definition \ref{quasi rep}.  Then there exists an $(S,6\epsilon)$-representation $\sigma:\Gamma\to\mathcal{B}(H)_1$ such that $\|\sigma(s)-\phi(s)\|<\epsilon$ for all $s\in S$ and such that $\sigma$ takes values in the unitary operators.
\end{lemma}

\begin{proof}
For each $s\in S$, we have that $\|\phi(s)\phi(s^{-1})-1\|<\epsilon$ and $\|\phi(s^{-1})\phi(s)-1\|<\epsilon$.  Let $b=\phi(s^{-1})\phi(s)$.  Then $b$ is invertible with
$$
b^{-1}=\sum_{n=0}^\infty (1-b)^n
$$
whence $\|b^{-1}\|< (1-\epsilon)^{-1}$.  Hence $\phi(s)$ is invertible with $\phi(s)^{-1}=\phi(s)b^{-1}$, whence also 
$$
\|\phi(s)^{-1}\|\leq \|\phi(s)\|\|b^{-1}\|< (1-\epsilon)^{-1}.
$$
Hence 
$$
1=\|\phi(s)\phi(s)^{-1}\|< \|\phi(s)\|(1-\epsilon)^{-1},
$$
and so $\|\phi(s)\|> 1-\epsilon$.  

Let now $\phi(s)=u_sa_s$ be the polar decomposition of $\phi(s)$ (see for example \cite[pages 21-23]{Blackadar:2006eq}) where $u_s$ is unitary and $a_s$ is positive and invertible.  We have $\|a_s\|=\|\phi(s)\|\leq 1$, and that $\phi(s)^{-1}=a_s^{-1}u_s^*$ whence $\|a_s^{-1}\|=\|\phi(s)^{-1}\|< (1-\epsilon)^{-1}$.  As $a_s$ and $a_s^{-1}$ are positive, their norms are the maximal elements of their spectrum so these inequalities imply that the spectrum of $a_s$ is contained in $(1-\epsilon,1]$.  The functional calculus then implies that $\|a_s-1\|<\epsilon$.  Hence 
\begin{equation}\label{uni close}
\|\phi(s)-u_s\|<\epsilon
\end{equation}
for each $s\in S$.

Now, for each $t\in S^2\setminus (S\cup \{e\})$, choose $l(t)$ and $r(t)$ in $S$ with $t=l(t)r(t)$.  Define $\sigma:\Gamma\to \mathcal{U}(H)$ 
$$
\sigma(g)=\left\{\begin{array}{ll} u_s & g=s\in S \\ u_{l(t)}u_{r(t)} & g=t\in S^2\setminus (S\cup \{e\}) \\ 1 & \text{otherwise}\end{array}\right.
$$
From line \eqref{uni close}, we have $\|\phi(s)-\sigma(s)\|<\epsilon$ for all $s\in S$.  To check that $\sigma$ is an $(S,6\epsilon)$-representation, we need to check that $\|\sigma(s_1)\sigma(s_2)-\sigma(s_1s_2)\|<6\epsilon$ for all $s_1,s_2\in S$.  Computing:
\begin{align}\label{inq}
\|\sigma(s_1)\sigma(s_2)-\sigma(s_1s_2)\| & \leq \|\phi(s_1)\phi(s_2)-\sigma(s_1s_2)\|+\|(\phi(s_1)-u_{s_1})\phi(s_2)\| \nonumber \\ & \quad +\|u_{s_1}(\phi(s_2)-u_{s_2})\| \nonumber \\
& <2\epsilon+\|\phi(s_1)\phi(s_2)-\sigma(s_1s_2)\|.
\end{align}
There are three cases: if $s_1=s_2^{-1}$, then $\sigma(s_1s_2)=1$ and $\|\phi(s_1)\phi(s_2)-1\|=\|\phi(s_1s_2)-\phi(s_1)\phi(s_2)\|<\epsilon$, and we are done by line \eqref{inq}; if $s_1s_2=s\in S$, then we have 
$$
\|\phi(s_1)\phi(s_2)-\sigma(s_1s_2)\|\leq \|\phi(s)-\sigma(s)-(\phi(s)-\phi(s_1s_2))\|<2\epsilon
$$
and again are done by line \eqref{inq}; finally, if $s_1s_2=t=l(t)r(t)\in S^2\setminus (S\cup \{e\})$, then 
\begin{align*}
\|\phi(s_1)\phi(s_2)-\sigma(s_1s_2)\| & \leq \|\phi(s_1)\pi(s_2)-\phi(t)\|+\|\phi(t)-\phi(l(t))\phi(r(t))\| \\ & \quad +\|(\phi(l(t))-u_{l(t)})\phi(r(t))\|+\|u_{l(t)}(\phi(r(t))-u_{r(t)})\| \\
& < 4\epsilon
\end{align*}
and again we are done by line \eqref{inq}.
\end{proof}

We need a special class of quasi-representation (see Definition \ref{quasi rep} for the general case).  This is classical, going back (at least) to Naimark \cite{Naimark:1943aa}.

\begin{definition}\label{ucp qr}
An $(S,\epsilon)$-representation $\phi:\Gamma\to \mathcal{B}(H)_1$ is \emph{unital\footnote{The word ``unital'' is redundant -- all quasi-representations preserve identities for us -- but we keep as it is the standard convention in the $C^*$-algebra literature.} completely positive} (ucp) if for any finite subset $F$ of $\Gamma$, the operator matrix
$$
(\phi(g^{-1}h))_{g,h\in F}\in M_{|F|}(\mathcal{B}(H))
$$
is positive.  
\end{definition}

One of the main reasons that ucp quasi-representations are useful is the next result, which is classical; the reader might also compare \cite[Proposition 2.3]{Kirchberg:1994ez}, which gives a more detailed result for  tracial approximations (and when the group has property (T)).  To state it, we recall that a linear map $\phi:A\to \mathcal{B}(H)$ from a complex $*$-algebra $A$ to $\mathcal{B}(H)$ is \emph{unital completely positive} (ucp) if it is unital, and if for all $n$ and any $a\in M_n(A)$, if $\phi_n:M_n(A)\to M_n(\mathcal{B}(H))$ is the map defined by applying $\phi$ entrywise, then the element $\phi(a^*a)\in M_n(\mathcal{B}(H))$ is a positive operator.

For more background on ucp maps between $C^*$-algebras and related concepts we use below like operator systems and the LLP, we recommend for example \cite{Brown:2008qy},  \cite{Pisier:2020aa}, or \cite{Paulsen:2003ib}.

\begin{proposition}\label{ucp extend}
Let $\Gamma$ be a discrete group, and let $\phi:\Gamma\to \mathcal{B}(H)_1$ be a ucp quasi-representation in the sense of Definition \ref{ucp qr}.
\begin{enumerate}[(i)]
\item \label{phi comp} There exists another Hilbert space $H'$, an isometry $v:H\to H'$, and a unitary representation $\pi:\Gamma\to \mathcal{B}(H')$ such that $\phi(g)=v^*\pi(g)v$ for all $g\in \Gamma$.
\item \label{phi ext} There is a unique ucp map $\phi:C^*(\Gamma)\to \mathcal{B}(H)$ extending $\phi$ by linearity and continuity.
\item \label{approx mult dom} If $S$ is a symmetric subset of $\Gamma$ and if $\epsilon>0$ is such that 
$$
\|\phi(s)^*\phi(s)-1\|\leq \epsilon\quad \text{and}\quad \|\phi(s)\phi(s)^*-1\|\leq \epsilon
$$
for all $s\in S$, then
$$
\|\phi(gs)-\phi(g)\phi(s)\|\leq \sqrt{\epsilon}\quad \text{and} \quad \|\phi(sg)-\phi(s)\phi(g)\|\leq \sqrt{\epsilon}
$$
for all $g\in \Gamma$ and $s\in S$.  In particular, this holds if $\phi$ is an $(S,\epsilon)$-representation as in Definition \ref{quasi rep}.
\item \label{ucp is rep} If $\phi$ is unitary-valued, then it is a representation.
\end{enumerate}
\end{proposition}

\begin{proof}
Part \eqref{phi comp} is essentially due to Naimark \cite{Naimark:1943aa}; see for example \cite[Theorem 4.8]{Paulsen:2003ib} for a modern proof of the precise statement we give.  Part \eqref{phi ext} is a direct consequence of part \eqref{phi comp}: indeed, $\pi:\Gamma\to \mathcal{B}(H')$ extends uniquely to a $*$-homomorphism $\pi:C^*(\Gamma)\to \mathcal{B}(H)$, and one checks directly that the formula $\phi(a):=v^*\pi(a)v$ defines a ucp extension of $\phi$ to $C^*(\Gamma)$.  

For part \eqref{approx mult dom}, let $v$ and $\pi$ be as in part \eqref{phi comp} and define $p:=vv^*$, which is a projection on $H'$.  Then for any $s\in S$, 
$$
\|p\pi(s)^*(1-p)\pi(s)p\|=\|v(v^*v-v\pi(s)^*vv^*\pi(s)v)v^*\|=\|1-\phi(s)^*\phi(s)\|\leq \epsilon.
$$
Hence by the $C^*$-identity, $\|(1-p)\pi(s)p\|\leq \sqrt{\epsilon}$.  For any $g\in \Gamma$ and $s\in S$, we therefore have 
\begin{align*}
\|\phi(gs)-\phi(g)\phi(s)\| & =\|v^*\pi(g)\pi(s)v-v^*\pi(g)vv^*\pi(s)v\| \\ & =\|v^*\pi(g)(1-p)\pi(s)v\|.
\end{align*}
As $pv=v$, the right hand side above is at most $\sqrt{\epsilon}$.  The inequality $\|\phi(gs)-\phi(g)\phi(s)\|\leq \sqrt{\epsilon}$ follows similarly on reversing the roles of $\phi(s)$ and $\phi(s)^*$.

Finally, part \eqref{ucp is rep} follows from part \eqref{approx mult dom} on taking $S=\Gamma$.
\end{proof}

\begin{remark}\label{ucp alm mult}
Analogously to part \eqref{approx mult dom} of Proposition \ref{ucp extend}, if $\phi:A\to B$ is a general ucp map between unital $C^*$-algebras and $X$ is a subset of $A$ such that $\|\phi(aa^*)-\phi(a)\phi(a)^*\|\leq \epsilon$ and $\|\phi(a^*a)-\phi(a)^*\phi(a)\|\leq \epsilon$ for all $a\in X$, then 
$$
\|\phi(ab)-\phi(a)\phi(b)\|\leq \sqrt{\epsilon}\|b\|\quad \text{and} \quad \|\phi(ba)-\phi(b)\phi(a)\|\leq \sqrt{\epsilon}\|b\|
$$
for all $b\in A$ and $a\in X$.  The proof is essentially the same.
\end{remark}

From Proposition \ref{ucp extend}, we see that the ucp condition is a strong one to put on quasi-representations.  Remarkably, however, under an assumption on $C^*(\Gamma)$ that holds quite widely, any suitably multiplicative quasi-representation can be approximated by a ucp quasi-representation.  The goal of the rest of this section is to establish this.  

We next give the key definition, which is due to Kirchberg \cite[Section 2]{Kirchberg:1993aa}.  For the statement, recall that an \emph{operator system} is a norm-closed and adjoint-closed subspace of a $C^*$-algebra that contains the unit.  Recall also if $(n_m)_{m=1}^\infty$ is a sequence of natural numbers that $\prod_{m=1}^\infty M_{n_m}(\C)$ denotes the  $C^*$-algebra of bounded sequences $(a_m)$ with $a_m\in M_{n_m}(\C)$, and $\bigoplus_{m=1}^\infty M_{n_m}(\C)$ denotes the ideal in this consisting of sequences $(a_m)$ such that $\|a_m\|\to 0$ as $m\to\infty$.

\begin{definition}\label{llp}
A unital $C^*$-algebra $A$ has the \emph{local lifting property} (LLP) if for any ucp map $\phi:A\to B/J$ into a quotient $C^*$-algebra and any finite-dimensional operator system $E\subseteq A$, there exists a ucp map $\psi:E\to B$ that lifts $\phi$, i.e.\ so that the following diagram commutes
$$
\xymatrix{ & & B \ar[d] \\
E \ar@{-->}[urr]^-\psi \ar[r] & A \ar[r]^-\phi & B/J}
$$
(here the unlabeled maps are the canonical inclusion and quotient).

The $C^*$-algebra has the \emph{weak matricial LLP} if the above conclusion holds in the special case that $B=\prod_{m=1}^\infty M_{n_m}(\C)$ for some sequence $(n_m)_{m=1}^\infty$ of natural numbers, $J$ is the ideal $\bigoplus_{m=1}^\infty M_{n_m}(\C)$, and $\phi$ is a unital $*$-homomorphism.

A group $\Gamma$ has the \emph{(weak matricial) LLP} if its maximal group $C^*$-algebra $C^*(\Gamma)$ has the (weak matricial) LLP.
\end{definition}

See Remark \ref{llp rem} below for a brief survey of the class of groups that are known to have the LLP. 

We next need to recall the notation of an injective $C^*$-algebra: see for example \cite[Section IV.2.1]{Blackadar:2006eq} for more background.

\begin{definition}\label{inj}
A unital $C^*$-algebra $I$ is \emph{injective} if for any operator system $E$ in a unital $C^*$-algebra $A$ and any ucp map $\phi:E\to I$, there exists a ucp extension $\widetilde{\phi}:A\to I$, i.e.\ so that the diagram
$$
\xymatrix{ A \ar@{-->}[dr]^-{\widetilde{\phi}} & \\ E\ar[u] \ar[r]_-\phi & I }
$$
commutes.  
\end{definition}

The only example of injective $C^*$-algebras we will need are products of the form $\prod_{m=1}^\infty M_{n_m}(\C)$: injectivity of such $C^*$-algebras is a direct consequence of the finite-dimensional case of Arveson's extension theorem as in \cite[Theorem 1.2.3]{Arveson:1969aa} (see for example \cite[Theorem 6.2]{Paulsen:2003ib} or \cite[Corollary 1.5.16]{Brown:2008qy} for textbook treatments).  

Variants of the next result are well-known: we got the idea from \cite[Lemma 2.1]{Lin:1995aa}; \cite[Corollary 1.7]{Ioana:2020aa} is also related.  We could not find exactly what we need in the literature, so provide a proof for the reader's convenience.

\begin{proposition}\label{llp approx}
Let $\Gamma$ be a countable group with the weak matricial LLP.   Then for any finite symmetric subset $S$ of $\Gamma$ and any $\epsilon>0$ there exists a finite subset $T$ of $\Gamma$ and $\delta>0$ such that if $\phi:\Gamma\to M_n(\C)$ is a $(T,\delta)$-representation in the sense of Definition \ref{quasi rep} then there exists a ucp $(S,\epsilon)$-representation $\psi:\Gamma\to M_n(\C)$ such that $\|\psi(s)-\phi(s)\|<\epsilon$ for all $s\in S$.
\end{proposition}

\begin{proof}
Assume for contradiction that the statement fails.  Let $T_1\subseteq T_2\subseteq T_3\subseteq \cdots $ be a nested collection of finite symmetric subsets of $\Gamma$ with union all of $\Gamma$.  Then there exists a finite symmetric subset $S$ of $\Gamma$ and $\epsilon>0$ such that for any $m\geq 1$ there are $n_m\in \N$ and a $(T_m,1/m)$-representation $\phi_m:\Gamma\to M_{n_m}(\C)$ such that for any ucp $(S,\epsilon)$-representation $\psi:\Gamma\to M_{n_m}(\C)$, we have $\|\psi(s)-\phi_m(s)\|\geq \epsilon$ for all $s\in S$ 

Define $M$ to be the $C^*$-algebra $\prod_{m=1}^\infty M_{n_m}$, and let $M_0$ be the ideal $\bigoplus_{m=1}^\infty M_{n_m}(\C)$.  Define $M_\infty:=M/M_0$ and define 
$$
\Phi:\Gamma\to M_\infty,\quad g\mapsto [\phi_1(g),\phi_2(g),\cdots].
$$
This is a homomorphism into the unitary group of $M_\infty$, so extends uniquely by linearity and continuity to a unital $*$-homomorphism $\Phi:C^*(\Gamma)\to M_\infty$.  Let $E\subseteq C^*(\Gamma)$ be the (finite-dimensional) operator system spanned by $\{1\}\cup S\cup S^2$.  Then by the weak matricial LLP for $C^*(\Gamma)$ there is a ucp map $\Psi:E\to M$ such that $\Psi$ lifts $\Phi|_E$.  As $M$ is an injective $C^*$-algebra, we may extend $\Psi$ to a ucp map $\Psi:C^*(\Gamma)\to M$.  Write $\psi_m:C^*(\Gamma)\to M_{n_m}$ for the composition of $\Psi$ with the canonical quotient map $M\to M_{n_m}$ defined by evaluating at the $m^\text{th}$ coordinate; note that $\psi_m$ is also ucp.  

Note now that the sequence $(\phi_m(s))_{m=1}^\infty\in M$ is bounded by definition of a quasi-representation and in particular, it is an element of $M$ that lifts $\Phi(s)$.  As $\Psi(s)$ also lifts $\Phi(s)$, we have that $\Psi(s)-(\phi_m(s))\in M_0$, i.e.\ that $\psi_m(s)-\phi_m(s)\to 0$ as $m\to\infty$.  It follows that for suitably large $m$, $\psi_m$ is a ucp $(S,\epsilon)$-representation, and satisfies $\|\psi_m(s)-\phi_m(s)\|<\epsilon$ for all $s\in S$, contradicting the assumption.
\end{proof}

Let us make a few comments on the relation of the LLP and weak-matricial LLP.

\begin{remark}\label{weak llp}
The version of the LLP where one stipulates that \emph{ucp} maps from $A$ into $B/J$, where $B=\prod_{m=1}^\infty M_{n_m}(\C)$ and $J=\oplus_{m=1}^\infty M_{n_m}(\C)$, lift to ucp maps is equivalent to the general case of the LLP: this follows from arguments of Ozawa in \cite{Ozawa:2001aa}; see particularly \cite[Remark 2.11]{Ozawa:2001aa}.  

On the other hand, the weak matricial LLP is genuinely weaker than the usual LLP.  Indeed, there are $C^*$-algebras without the LLP that do not admit any non-zero $*$-homomorphisms into $B/J$ as above.  For example, any $C^*$-algebra admitting such a $*$-homomorphism has a non-zero stably finite quotient.  In particular, this is not the case for $\mathcal{B}(\ell^2(\N))$.  On the other hand, $\mathcal{B}(\ell^2(\N))$ is known not to have the LLP as a consequence of \cite[Corollary 3.1]{Junge:1995aa} (see also \cite[Corollary 13.2.5 and Theorem 13.5.1]{Brown:2008qy} for a textbook exposition).

We do not know if the weak matricial LLP is genuinely weaker than the LLP for group $C^*$-algebras.   It was pointed out to us by Tatiana Shulman that the weak matricial LLP for $A$ is equivalent to the Brown-Douglas-Fillmore semigroup $Ext(A)$ (see for example \cite[Chapter 2]{Higson:2000bs}) being a group whenever $A$ is an RFD\footnote{See Example \ref{good ex} below for the definition of RFD.} $C^*$-algebra, and that it is open whether $Ext(A)$ being a group is equivalent to the LLP in this level of generality.
\end{remark}

\section{Controlled $K$-homology, $KL$-theory, and total $K$-theory}\label{ckh sec}

Our goal in this section is to recall the definition of controlled $K$-homology and its relationship with $KL$-homology from \cite{Willett:2020aa}.  We then relate that to the action of $KL$-theory on total $K$-theory from \cite{Dadarlat:1996aa,Dadarlat:2005aa}.  Throughout this section, anything called $A$ is a separable unital $C^*$-algebra; we are only interested in the case that $A=C^*(\Gamma)$, but the extra generality makes no difference.

This section is perhaps the most technical of the paper, and we allow ourselves to skip some standard $K$-theoretic details where this seemed unlikely to cause confusion.  We hope a background in $C^*$-algebra $K$-theory at the level of \cite{Rordam:2000mz} should be enough to understand this material: although we use some more advanced material like $K$-homology, $KL$-homology, and $KK$-theory, we only really need their formal properties.

\subsection{Controlled $K$-homology}

We start with the definition of controlled $K$-homology from \cite{Willett:2020aa}.   Let us say that a representation of a $C^*$-algebra is \emph{infinitely amplified} if it is a countably infinite direct sum of some other representation.

\begin{definition}\label{con k hom}
Let $A$ be a separable unital $C^*$-algebra.  Fix a faithful, unital, infinitely amplified representation $A\subseteq \mathcal{B}(H)$ of $A$ on a separable Hilbert space $H$.

For $\epsilon>0$ and a finite subset $X$ of the unit ball of $A$, define  
$$
\mathcal{P}_\epsilon(X):=\left\{\begin{array}{l|l} p\in M_2(\mathcal{B}(H)) &  p=p^*=p^2,~p-\begin{psmallmatrix} 1 & 0 \\ 0 & 0 \end{psmallmatrix}\in M_2(\mathcal{K}(H)), \\
& \text{and }\|[p,x]\|<\epsilon \text{ for all } x\in X\end{array}\right\}.
$$
Define the \emph{controlled (even) $K$-homology group} 
$$
K^0_\epsilon(X):=\pi_0(\mathcal{P}_\epsilon(X))
$$ 
to be the set of path components of $\mathcal{P}_\epsilon(X)$.  The group operation on $K^0_\epsilon(X)$ is defined by 
$$
[p]+[q]:=[sps^*+tqt^*]
$$
where $s$ and $t$ are isometries in the commutant of $A$ satisfying the \emph{Cuntz $O_2$-relation} $ss^*+tt^*=1$ (the choice of such isometries does not matter: we refer to \cite[Proposition 6.5]{Willett:2020aa} for details, but will not need to use the specifics).
\end{definition}

With this structure, each $K^0_\epsilon(X)$ is a countable abelian group.  It does not depend on the choice of representation by (a slight variant of) \cite[Lemma 6.6]{Willett:2020aa}.

\begin{definition}\label{x eps hom}
Let $\phi:A\to B$ be a linear map from a $*$-algebra to a $C^*$-algebra, let $X$ be a finite subset of $A$, and let $\epsilon>0$.   Then $\phi$ is called an \emph{$(X,\epsilon)$-$*$-homomorphism} if
$$
\|\phi(x^*)-\phi(x)^*\|<\epsilon\footnote{We will typically apply this definition to ucp maps, in which case ``$\phi(x^*)=\phi(x)^*$'' holds automatically.  For $*$-preserving linear maps, what we call ``being an $(X,\epsilon)$-$*$-homomorphism'' is often referred to as being ``$(X,\epsilon)$-multiplicative'' in the literature.} \quad \text{and}\quad \|\phi(xy)-\phi(x)\phi(y)\|<\epsilon.
$$
for all $x,y\in X$.
\end{definition}

\begin{example}\label{xem ex}
The most important examples of $(X,\epsilon)$-$*$-homomorphisms for us come from almost commuting projections.  Indeed, let $X$ be a finite subset of the unit ball of a $C^*$-algebra $A$, and let $\pi:A\to B$ be a $*$-homomorphism.  Let $\epsilon>0$ and let $p\in B$ be a projection such that $\|[p,x]\|<\epsilon$ for all $x\in X$.  Then one checks directly that the map
$$
\phi:A\to B,\quad a\mapsto p\pi(a)p
$$
is an $(X,\epsilon)$-$*$-homomorphism.  A map of the form above is moreover ucp, and in fact any ucp $(X,\epsilon)$-$*$-homomorphism arises like this (at the price of making $\epsilon$ a little bigger): compare the proof of Lemma \ref{ucp mult com} below.
\end{example}

\begin{definition}\label{dir set}
Let $A$ be a separable, unital $C^*$-algebra.  Let $\mathcal{I}_A$ be the set of ordered pairs $(X,\epsilon)$ where $X$ is a finite subset of the unit ball of $A$, and $\epsilon>0$.  We order $\mathcal{I}_A$ by stipulating that $(X,\epsilon)\leq (Y,\delta)$ if whenever $\phi:A\to B$ is a ucp  $(Y,\delta)$-$*$-homomorphism in the sense of Definition \ref{x eps hom}, then $\phi$ is also an $(X,\epsilon)$-$*$-homomorphism.
\end{definition}

To establish conventions, it will be convenient to recall a standard algebraic construction.

\begin{definition}\label{inv lim}
Let $I$ be a directed set, and let $(G_i)_{i\in I}$ be an inverse system of abelian groups, i.e.\ for each $j\geq i$, there are homomorphisms $\phi_{ij}:G_j\to G_i$ such that $\phi_{ij}\circ \phi_{jk}=\phi_{ik}$ for all $k\geq j \geq i$, and such that $\phi_{ii}$ is the identity on $G_i$ for all $i\in I$.

The \emph{inverse limit} of an inverse system $(G_i)_{i\in I}$ with connecting maps $(\phi_{ij})_{i,j\in \mathcal{I},j\geq i}$ is defined to be 
$$
\lim_{\leftarrow} G_i:=\Bigg\{(g_i)\in \prod_{i\in I} G_i \mid \phi_{ij}(g_j)=g_i \text{ for all } j\geq i\Bigg\}.
$$
If each $G_i$ is a topological\footnote{Typically discrete in our applications.} abelian group then ${\displaystyle \lim_{\leftarrow} G_i}$ is equipped with the subspace topology it inherits from the product topology on $\prod G_i$.
\end{definition}

\begin{remark}\label{dir set rem}
Let $\mathcal{I}_A$ be as in Definition \ref{dir set}. 
\begin{enumerate}[(i)]
\item The set $\mathcal{I}_A$ is directed: an upper bound for $(X,\epsilon)$ and $(Y,\delta)$ is $(X\cup Y,\min\{\epsilon,\delta\})$.
\item \label{inc part} If $(X,\epsilon)\leq (Y,\delta)$ then we have an inclusion 
$$
\mathcal{P}_{\delta}(Y)\subseteq \mathcal{P}_\epsilon(X)
$$
(consider the ucp map defined by compression by an element of $\mathcal{P}_\delta(Y)$ as in Example \ref{xem ex}).  Thus there is a canonical `forgetful map'
$$
\theta:K^0_\delta(Y)\to K^0_\epsilon(X)
$$
induced by the identity inclusion $\mathcal{P}_\delta(Y) \to \mathcal{P}_\epsilon(X)$.  These forgetful maps make the family of groups $(K^0_\epsilon(X))_{(X,\epsilon)\in \mathcal{I}_A}$ into an inverse system.
\item Recall that a subset $S$ of a directed set $I$ is \emph{cofinal} if for all $i\in I$ there is $s\in S$ with $s\geq i$.  The directed set $\mathcal{I}_A$ has cofinal sequences: for example, if $X_1\subseteq X_2\subseteq \cdots$ is a nested sequence of finite subsets of the unit ball of $A$ with dense union and $\epsilon_n\searrow 0$, then $((X_n,\epsilon_n))_{n=1}^\infty$ is cofinal.   Moreover, if $A$ is finitely generated by some subset $X$ of its unit ball, then any sequence $((X,\epsilon_n))_{n=1}^\infty$ with $\epsilon_n\searrow 0$ is cofinal.  

Passing to a cofinal subset does not affect an inverse limit.  As such, for our applications, we could replace $\mathcal{I}_A$ by a sequence; this is occasionally technically useful, but we will not need it in this paper.
\end{enumerate}
\end{remark}

\begin{definition}\label{con k hom lim}
We write ${\displaystyle \lim_{\leftarrow}K^0_\epsilon(X)}$ for the inverse limit of the inverse system $(K^0_\epsilon(X))_{(X,\epsilon)\in \mathcal{I}_A}$ from Remark \ref{dir set rem} part \eqref{inc part}, equipped with the topology it inherits by equipping each $K^0_\epsilon(X)$ with the discrete topology. 
\end{definition}

We now recall the relation of the group from Definition \ref{con k hom lim} to analytic $K$-homology.  Analytic $K$-homology was introduced by Kasparov \cite{Kasparov:1975ht} and Brown-Douglas-Fillmore \cite{Brown:1977qa}; see \cite{Higson:2000bs} for a textbook treatment.  The (even) $K$-homology group $K^0(A)$ of a separable $C^*$-algebra is an abelian group equipped with a canonical (possibly non-Hausdorff) topology: see \cite[8.5]{Brown:1977qa} for the original version of this topology, and \cite{Dadarlat:2005aa} for a definitive modern treatment (in a more general context).  Let $\overline{\{0\}}$ denote the closure of $0$ in $K^0(A)$, and following Dadarlat \cite[Section 5]{Dadarlat:2005aa}\footnote{The $KL$ groups were originally introduced by R\o{}rdam under more stringent conditions on $A$ using algebraic ideas: see \cite[Section 5]{Rordam:1995aa}.  The equivalence of the definition in line \eqref{kl def} with R\o{}rdam's (under R\o{}rdam's conditions on $A$) follows from \cite[Theorem 3.3]{Schochet:2002aa} or \cite[Corollary 4.6]{Dadarlat:1999ab}.} define 
\begin{equation}\label{kl def}
KL^0(A):=K^0(A)/\overline{\{0\}}
\end{equation}
to be the associated maximal Hausdorff quotient group. We call $KL^0(A)$ the \emph{$KL$-homology group} of $A$.

The following result is (a special case of part of) \cite[Proposition A.18]{Willett:2020aa}.

\begin{theorem}\label{ck kl}
For any unital separable $C^*$-algebra $A$, there is a canonical homeomorphism 
$$
\Theta:KL^0(A)\to \displaystyle \lim_{\leftarrow}K^0_\epsilon(X). \eqno\qed
$$
\end{theorem}

\subsection{Controlled $K$-theory with finite coefficients}

In this subsection, we recall the definition of $K$-theory with finite coefficients and introduce a controlled variant of this.  We then recall the coefficient change operations and use these to bundle all the $K$-theory groups with various coefficients into a single object and discuss homomorphisms between these objects.  Finally, we relate the group ${\displaystyle \lim_{\leftarrow} K^0_\epsilon(X)}$ of Definition \ref{con k hom lim} above to this object.

$K$-theory with finite coefficients was introduced into $C^*$-algebra theory by Schochet \cite{Schochet:1984ab}, based on a construction of Dold for generalized homology theories \cite[Section 1]{Dold:1962aa}.  The original definition (see \cite[Definition 1.2]{Schochet:1984ab}) is $K_*(A;\Z/n):=K_*(A\otimes C_n)$ for a particular commutative, non-unital $C^*$-algebra $C_n$  that satisfies 
\begin{equation}\label{right k}
K_0(C_n)=\Z/n \quad \text{and} \quad K_1(C_n)=0.
\end{equation}
One can replace $C_n$ with any `reasonable' $C^*$-algebra with the $K$-theory groups as in line \eqref{right k} above: see \cite[Theorem 6.4]{Schochet:1984ab}.  For us, it will be convenient to use the Cuntz algebras $O_{n+1}$ for $2\leq n<\infty$, primarily as they are unital\footnote{There is no unital commutative $C^*$-algebra satisfying the conditions in line \eqref{right k}, so we cannot use something commutative.}, and also as the Kirchberg-Phillips classification theorem \cite{Phillips-documenta,Kirchberg-ICM} guarantees that any homomorphism $\Z/n\to \Z/m$ is the map induced on $K$-theory by a unital homomorphism between matrix algebras over Cuntz algebras.  See \cite{Cuntz:1977aa} for the original definition of Cuntz algebras, and \cite[Section 4.2]{Rordam:2002cs} for further background.

\begin{definition}\label{kthy coeff}
Let $A$ be a $C^*$-algebra and let $n\geq 2$.  The \emph{$K$-theory of $A$ with coefficients in $\Z/n$} is defined to be 
$$
K_*(A;\Z/n):=K_*(A\otimes O_{n+1}).
$$
\end{definition}

We will next introduce a `controlled' variant of $K$-theory, possibly with finite coefficients.  The original version of controlled $K$-theory is from \cite{Yu:1998wj}, and is adapted to geometric motivations; the version we give below is adapted to abstract $C^*$-algebras and ucp $(X,\epsilon)$-$*$-homomorphisms in the sense of Definition \ref{x eps hom}.  There are also antecedents for this from the classification theory of $C^*$-algebras: see for example \cite[Subsection 3.3]{Dadarlat:2002aa}.

\begin{definition}\label{x eps com}
Let $A$ and $B$ be $C^*$-algebras with $A$ unital, let $X$ be a finite subset of $A$, and let $\epsilon>0$.  A projection $p\in A\otimes B$ is \emph{$(X,\epsilon)$-compatible} if whenever $\phi:A\to C$ is a ucp $(X,\epsilon)$-$*$-homomorphism to another $C^*$-algebra $C$ in the sense of Definition \ref{x eps hom}, we have that\footnote{As $\phi$ is contractive and positive, the spectrum is automatically contained in $[0,1]$; the condition is thus saying that the spectrum avoids $[1/4,3/4]$.  The choice of $1/4$ and $3/4$ is not important: we just need to ensure that there is some gap around $1/2$.} 
$$
\text{spectrum}\big((\phi\otimes\text{id}_B)(p)\big)\subseteq [0,1/4)\cup (3/4,1].
$$
Write $\text{Proj}^{X,\epsilon}(A\otimes B)$ for the set of $(X,\epsilon)$-compatible projections in $A\otimes B$.  
\end{definition}

\begin{definition}\label{x eps com 2}
Let $A$ be a unital $C^*$-algebra, let $X$ be a finite subset of the unit ball of $A$ such that $X=X^*$, and let $\epsilon>0$.  For each $n\geq 2$ define 
$$
\mathcal{Q}^\epsilon(X;\Z/n):=\bigsqcup_{k=1}^\infty \text{Proj}^{X,\epsilon}(A\otimes M_k(\C)\otimes O_{n+1}).
$$
Let $\sim$ be the equivalence relation on $\mathcal{Q}^\epsilon(X;\Z/n)$ generated by the conditions below: 
\begin{enumerate}[(i)]
\item $p\sim q$ if both are in the same path component of some $\text{Proj}^{X,\epsilon}(A\otimes M_k(\C)\otimes O_{n+1})$;
\item $p\sim q$ if $p\in \text{Proj}^{X,\epsilon}(A\otimes M_k(\C)\otimes O_{n+1})$ and there is $m\geq 1$ such that $q\in \text{Proj}^{X,\epsilon}(A\otimes M_{k+m}(\C)\otimes O_{n+1})$ and $q=p\oplus 0_m$.
\end{enumerate}
Define $\mathcal{V}^\epsilon(X;\Z/n)$ to be $\mathcal{Q}^\epsilon(X;\Z/n)/\sim$, and define $K^\epsilon_0(X;\Z/n)$ to be the Grothendieck group of $\mathcal{V}^\epsilon(X;\Z/n)$.

Define $\mathcal{Q}^\epsilon(X)$, $\mathcal{V}^\epsilon(X)$, and $K_0^\epsilon(X)$ analogously, but omitting the $O_{n+1}$.

We call the groups $K^\epsilon_0(X;\,\cdot\,)$ \emph{controlled $K$-theory groups}.
\end{definition}

The definition of $K^\epsilon_0(X;\,\cdot\,)$ is motivated by the following basic lemma, whose proof we leave to the reader.

\begin{lemma}\label{func lem}
Let $A$ be a unital $C^*$-algebra, let $X$ be a finite subset of the unit ball of $A$ such that $X=X^*$, and let $\epsilon>0$.  Let $\phi:A\to B$ be a ucp $(X,\epsilon)$-$*$-homomorphism as in Definition \ref{x eps hom}.  Let $\chi$ be the characteristic function of $(1/2,\infty)$.  

For any $n\geq 2$ and any $(X,\epsilon)$-compatible projection $p\in \text{Proj}^{X,\epsilon}(A\otimes M_k(\C)\otimes O_{n+1})$, define 
$$
\phi_*(p):=\chi((\phi\otimes \text{id}_{M_k(\C)\otimes O_{n+1}})(p))\in B\otimes M_k(\C)\otimes O_{n+1}.
$$
Then the assignment $[p]\mapsto [\phi_*(p)]$ uniquely determines a homomorphism 
$$
\phi_*:K_0^\epsilon(X;\Z/n)\to K_0(B;\Z/n).
$$
The same construction with $O_{n+1}$ omitted similarly uniquely determines a homomorphism
$$
\phi_*:K_0^\epsilon(X)\to K_0(B). \eqno\qed
$$
\end{lemma}

Note that if we drop the condition that the projections and homotopies used to define $K_0^\epsilon(X;\,\cdot\,)$ above are $(X,\epsilon)$-compatible, then we just get $K_0(A;\,\cdot\,)$.  Thus Lemma \ref{func lem} says that the groups $K^\epsilon_0(X;\,\cdot\,)$ are in some sense `the part\footnote{\label{part fn} The word ``part'' may be misleading: there is a canonical map $K^\epsilon_0(X)\to K_0(A)$, but it is probably not injective in general.} of $K_0(A;\,\cdot\,)$ on which ucp $(X,\epsilon)$-$*$-homomorphisms act'.

Having introduced these groups, we introduce natural transformations between them.  There are several equivalent ways to do this: see for example \cite[Section 1.2]{Dadarlat:1996aa} or \cite[Section 6]{Dadarlat:2012aa}. The discussion in \cite[Appendix A]{Carrion:2020aa} gives a recent overview, and in particular \cite[Lemma A.4]{Carrion:2020aa} gives a proof that all `reasonable' approaches give the same outcome.

\begin{definition}\label{lambda0}
For notational convenience, define $B_0=\C$, and $B_n=O_{n+1}$ for $n\geq 2$.  Let $\mathcal{N}:=\{0\}\cup \{n\in \Z\mid n\geq 2\}$.  

Let $\Lambda_0$ be the (small) category with object set $\mathcal{N}$, and with morphisms given by $\text{Hom}(n,m):=KK(B_n,B_m)$.  Composition of morphisms is given by the Kasparov product.  

A \emph{$\Lambda_0$-module} is a functor from the category $\Lambda_0$ to the category of abelian groups: more concretely, a $\Lambda_0$-module is a sequence $\underline{G}=(G_n)_{n\in \mathcal{N}}$ of abelian groups such that elements of $KK(B_n,B_m)$ define morphisms from $G_n$ to $G_m$, compatibly with the relations between these morphisms in the $KK$-category.  A \emph{$\Lambda_0$-submodule} $\underline{H}$ of $\underline{G}$ consists of a collection of subgroups $H_n$ of $G_n$ that are preserved by the morphisms from $\Lambda_0$.

Given $\Lambda_0$-modules $\underline{G}$ and $\underline{H}$, we define $\text{Hom}_{\Lambda_0}(\underline{G},\underline{H})$ to be the set of natural transformations from $\underline{G}$ to $\underline{H}$: more concretely, an element of this Hom set is a sequence of group homomorphisms $(\alpha_n:G_n\to H_n)_{n\in \mathcal{N}}$ that intertwine the morphisms coming from $KK(B_n,B_m)$.  We equip $\text{Hom}_{\Lambda_0}(\underline{G},\underline{H})$ with the abelian group structure defined by 
$$
(\alpha_n)+(\beta_n):=(\alpha_n+\beta_n)
$$
and with the topology of pointwise convergence, i.e.\ a net $(\alpha^{(i)})_{i\in I}$ converges to $\alpha$ if and only if for every $n$, and every $g\in G_n$ there exists $i_{n,x}\in I$ such that for all $i\geq i_{n,x}$, $\alpha_n^{(i)}(g)=\alpha_n(g)$.
\end{definition}

\begin{example}\label{basic lambda 0}
Let $A$ be a $C^*$-algebra, and let $\underline{K}_0(A)$ denote the sequence of abelian groups $(G_n)_{n\in \mathcal{N}}$ with $G_0=K_0(A)$ and $G_n=K_0(A;\Z/n)$ for $n\geq 2$.  Then $\underline{K}_0(A)$ is a $\Lambda_0$-module, with the action defined by Kasparov product.  We call $\underline{K}_0(A)$ the \emph{total $K$-theory} of $A$.

Moreover, a homomorphism $\phi:A\to B$ (or more generally, an element of $KK(A,B)$) induces an element $\phi_*$ of $\text{Hom}_{\Lambda_0}(\underline{K}_0(A),\underline{K}_0(B))$ by associativity of the Kasparov product.
\end{example}

\begin{remark}\label{only zero}
We warn the reader that what we call ``$\text{Hom}_{\Lambda_0}(\underline{K}_0(A),\underline{K}_0(B))$'' is usually not used in favor of a richer object denoted $\text{Hom}_{\Lambda}(\underline{K}_*(A),\underline{K}_*(B))$.  The latter also involves information from the $K_1$-groups, and additional (`Bockstein') natural transformations that switch degrees, i.e.\ that map between between $K_0(A;\,\cdot\,)$ and $K_1(A;\,\cdot\,)$ and vice versa: see \cite[Section 1.4]{Dadarlat:1996aa}.  In our language, $\Lambda$ can be described as the category with objects $\{(i,n)\mid i\in \Z/2, n\in \mathcal{N}\}$, and where $\text{Hom}((i,n),(j,m)):=KK_{i+j}(B_n,B_m)$; the group $\text{Hom}_{\Lambda}(\underline{K}_*(A),\underline{K}_*(B))$ is then defined analogously to $\text{Hom}_{\Lambda_0}(\underline{K}_0(A),\underline{K}_0(B))$.  

However, in this paper the only example we actually use will be $B=\C$, in which case the natural forgetful map 
$$
\text{Hom}_{\Lambda}(\underline{K}_*(A),\underline{K}_*(\C))\to \text{Hom}_{\Lambda_0}(\underline{K}_0(A),\underline{K}_0(\C))
$$
is an isomorphism; thus we hope our approach should cause no confusion.
\end{remark}

\begin{remark}\label{l0n}
It will also occasionally be useful for us to work with the following `truncations' of the category $\Lambda_0$.  For $n=0$, or $n\geq 2$, let $\Lambda_0^{(n)}$ denote the full subcategory of $\Lambda_0$ on the subset of $\mathcal{N}$ consisting of $0$ and all $k\geq 2$ such that $k$ divides $n$.

We abuse notation by also writing $\underline{K}_0(A)$ for the image of the $\Lambda_0$-module of Remark \ref{basic lambda 0} under the natural forgetful map from $\Lambda_0$-modules to $\Lambda_0^{(n)}$-modules.
\end{remark}



\begin{remark}\label{lambda0 homs}
Let $B_n$ be the $C^*$-algebra from Definition \ref{lambda0}.  It follows from the Kirchberg-Phillips classification theorem \cite{Phillips-documenta,Kirchberg-ICM}\footnote{We only need this for Cuntz algebras, in which case the essential results were known earlier \cite{Rordam:1995aa,Elliott:1995aa}; however, it is probably simpler to just use the full Kirchberg-Phillips theorem.  See \cite{Gabe:2019ws} for a recent approach that is somewhat simpler than the original one of Kirchberg and Phillips.} that for $n,m\geq 2$, any non-zero element of $KK(B_n,B_m)$ can be realized by a unital $*$-homomorphism $\phi:B_n\to M_k(B_m)$ for an appropriate $k\in \N$\footnote{Including matrix algebras is necessary to ensure that the units go to the correct place.} (see for example \cite[Theorem A]{Gabe:2019ws}), and moreover that any such homomorphism is unique up to asymptotic unitary equivalence (see for example \cite[Theorem B]{Gabe:2019ws}).   

On the other hand, any element of $KK(B_1,B_n)$ can be realized by the unique unital homomorphism $\phi:\C\to M_k(O_{m+1})$ for some $k$, and $KK(B_n,B_1)=0$ for $n\geq 2$.
\end{remark}

For later purposes, let us record a basic algebraic property of the $\Lambda_0$-modules $\underline{K}_0(A)$.  Analogous facts are probable known to experts, but we were unable to find a reference in the literature, so sketch a proof.

\begin{lemma}\label{l0 fg}
Let $A$ be a unital $C^*$-algebra, and let $B_n$ and $\mathcal{N}$ be as in Definition \ref{lambda0}.  Let $P$ be a finite collection of projections, each in $A\otimes M_k(\C)\otimes B_{n(p)}$ for some $k\geq 1$ and $n(p)\in \mathcal{N}$, and where we assume that $P$ contains the unit of $A=A\otimes M_1(\C)\otimes B_0$.  

Let $N$ be the least common multiple of the numbers $n(p)$ as $p$ ranges over $P$ assuming at least one $n(p)$ is positive, and let $N=0$ if $n(p)=0$ for all $p\in P$.  With notation as in Remark \ref{l0n}, let $\langle P\rangle$ be the $\Lambda_0^{(N)}$-submodule of $\underline{K}_0(A)$ generated by $P$, so $\langle P\rangle$ consists of a finite collection $\{P_n\mid n=0, \text{ or } n\geq 2\text{ and } n|N\}$ with each $P_n$ a subgroup of $K_0(A\otimes B_n)$.

Then for each $n$, each element of $P_n$ can be represented by a formal difference of projections of the form 
\begin{equation}\label{sm rep}
\bigoplus _{i=1}^l (\text{id}_{M_{k_i}(A)}\otimes \lambda_i)(p_i)
\end{equation}
where $l\in \N$, $p_i\in P$ and $\lambda_i$ is a $*$-homomorphism representing a morphism in $\Lambda_0^{(N)}$ as in Remark \ref{lambda0 homs}.  

Moreover, $\langle P\rangle$ is `finitely presented' in the following sense: there is a finite set $H$ of homotopies between elements of the form in line \eqref{sm rep} such that if for some $l,n\in \N$ and $[\lambda_i]$ a morphism for $\Lambda_0^{(N)}$ we have 
\begin{equation}\label{sm rel}
\sum_{i=1}^l [\lambda_i][p_i]=0 \quad \text{in}\quad K_0(A\otimes B_n)
\end{equation}
then there is $m\geq 0$, $S\subseteq \{1,...,l\}$ and a homotopy between 
$$
\bigoplus_{i\in S} (\text{id}_{M_{k_i}(A)}\otimes \lambda_i)(p_i)\oplus 1_m,\quad  \bigoplus_{i\not\in S} (\text{id}_{M_{k_i}(A)}\otimes \lambda_i)(p_i)\oplus 1_m
$$
that can be approximated arbitrarily well by block sums of images of the homotopies from the finite set $H$ under $*$-homomorphisms representing the elements of $\Lambda_0^{(N)}$ as in Remark \ref{lambda0 homs}.
\end{lemma}

\begin{proof}
The existence of representatives as in line \eqref{sm rep} is straightforward.

For the statement on homotopies, we first note that $B_n$ is in the bootstrap class of Rosenberg-Schochet \cite{Rosenberg:1987bh}.  It thus satisfies the UCT exact sequence of \cite[1.17, page 439]{Rosenberg:1987bh}; as $K_1(B_n)=0$, the UCT exact sequence just reduces to an isomorphism $KK(B_n,B_m)\cong \text{Hom}(K_0(B_n),K_0(B_m))$.  Hence $KK(B_n,B_m)\cong \text{Hom}(\Z/n,\Z/m)$ (here ``$\Z/0$'' should be interpreted to mean $\Z$).  These identifications are moreover compatible with composition of $KK$ elements.  

It is then not too difficult to see from this that the category of $\Lambda_0^{(N)}$-modules is equivalent to the category of (left) modules over the ring $R:=\text{End}(\Z\oplus \bigoplus_{n|N}\Z/n)$ of endomorphisms of the abelian group $\Z\oplus \bigoplus_{n|N}\Z/n$.  Note that as a (left) module over $\Z$, $R$ is finitely generated; it is thus Noetherian as a module over $\Z$, and hence (left) Noetherian as a ring.  As $R$ is left Noetherian, finitely generated left modules over $R$ are finitely presented; the statement on only needing finitely many homotopies follows from this and a little algebra. 
\end{proof}

We now show that the controlled $K$-theory groups of Definition \ref{x eps com 2} are $\Lambda_0$-modules.

\begin{lemma}\label{x eps lambda module}
Let $A$ be a unital $C^*$-algebra, let $X$ be a finite subset of the unit ball of $A$ such that $X=X^*$, and let $\epsilon>0$.  Let $\underline{K}_0^\epsilon(X)$ denote the sequence of abelian groups $(G_n)_{n=1}^\infty$ with $G_1=K_0^\epsilon(X)$ and $G_n=K_0^\epsilon(X;\Z/n)$ for $n\geq 2$.  

Then $\underline{K}_0^\epsilon(X)$ can be made into a $\Lambda_0$-module in a canonical way.  Moreover, if $\phi:A\to B$ is a ucp $(X,\epsilon)$-$*$-homomorphism as in Definition \ref{x eps hom}, then the induced maps $\phi_*$ of Lemma \ref{func lem} define an element 
$$
\phi_*\in \text{Hom}_{\Lambda_0}(\underline{K}_0^\epsilon(X),\underline{K}_0(B)).
$$
\end{lemma}

\begin{proof}
As in Remark \ref{lambda0 homs} every element of $\Lambda_0$ is induced by a $*$-homomorphism that is unique up to approximate unitary equivalence.  One checks directly that unital $*$-homomorphisms induce maps between controlled $K$-theory groups.  Moreover, the uniqueness up to approximate unitary equivalence guarantees that the induced maps on $K$-theory do not depend on the choice of inducing $*$-homomoprhism.

Compatibility of this structure with the maps induced by an $(X,\epsilon)$-$*$-homomorphism follows as if $\phi:A\to B$ and $\psi:C\to D$ are ucp maps between $C^*$-algebras, then 
$$
(\text{id}_B\otimes \psi)\circ (\phi\otimes\text{id}_C) =(\phi\otimes \text{id}_D)\circ (\text{id}_A\otimes \psi). \qedhere
$$
\end{proof}

Our next task is to relate the $\Lambda_0$-modules $\underline{K}_0(A)$ of Example \ref{basic lambda 0} and $\underline{K}_0^\epsilon(X)$ of Lemma \ref{x eps lambda module}.   

Let $\mathcal{I}_A$ be the directed set of Definition \ref{dir set}.  Then for $(X,\epsilon)\leq (Y,\delta)$ in $\mathcal{I}_A$ there are canonical maps
$$
K_0^\epsilon(X;\Z/n)\to K_0^\delta(Y;\Z/n)
$$
induced by the identity inclusions $\mathcal{Q}^\epsilon(X;\Z/n)\to \mathcal{Q}^\delta(Y;\Z/n)$, and similarly for the case with integer coefficients.  These maps fit together into maps of $\Lambda_0$-modules, so we get an inverse system 
$$
(\underline{K}_0^\epsilon(X))_{(X,\epsilon)\in \mathcal{I}_A}
$$
in the category of $\Lambda_0$-modules.  For each $(X,\epsilon)$ there are moreover canonical homomorphisms
\begin{equation}\label{dir lim maps 0}
K_0^\epsilon(X;\Z/n)\to K_0(A;\Z/n)
\end{equation}
induced by the identity inclusions of $(X,\epsilon)$-compatible projections into all projections, and similarly for the case with integer coefficients.  These again fit together to give maps of $\Lambda_0$-modules
\begin{equation}\label{dir lim maps}
\underline{K}_0^\epsilon(X)\to \underline{K}_0(A).
\end{equation}

The following lemma can be compared to the analogous results in other versions of controlled $K$-theory: see for example \cite[Proposition 4.9]{Guentner:2014bh} or \cite[Discussion around Remark 1.18]{Oyono-Oyono:2011fk}. 

\begin{lemma}\label{dir lim k thy}
Let $A$ be a unital $C^*$-algebra.  The maps of line \eqref{dir lim maps} induce canonical isomorphisms
\begin{equation}\label{dir lim zn}
\lim_{\rightarrow} K_0^\epsilon(X;\Z/n)\cong  K_0(A;\Z/n)
\end{equation}
and similarly 
\begin{equation}\label{dir lim int}
\lim_{\rightarrow} K_0^\epsilon(X)\cong  K_0(A).
\end{equation}
These maps fit together to define an isomorphism 
\begin{equation}\label{dir lim ex}
\lim_{\rightarrow} \underline{K}_0^\epsilon(X)\cong  \underline{K}_0(A)
\end{equation}
in the category of $\Lambda_0$-modules.  Finally, this isomorphism induces a canonical algebraic isomorphism 
\begin{equation}\label{kappa def}
\kappa : \lim_{\leftarrow}\text{Hom}_{\Lambda_0}(\underline{K}^\epsilon_0(X),\underline{K}_0(B))\cong \text{Hom}_{\Lambda_0}(\underline{K}_0(A),\underline{K}_0(B)).
\end{equation}
that is moreover a homeomorphism if each of the $\text{Hom}_{\Lambda_0}$ groups involved is equipped with the topology of pointwise convergence, and the left hand side is equipped with the inverse limit topology. 
\end{lemma}

\begin{proof}
For the isomorphisms of lines \eqref{dir lim zn} and \eqref{dir lim int}, the essential point is that any finite set of projections in $A\otimes M_k(\C)\otimes O_{n+1}$ is $(X,\epsilon)$-compatible for a large enough element of the set $\mathcal{I}_A$ of Definition \ref{dir set}: to see this, approximate the projections by finite sums of elementary tensors, take $X$ to contain all of the elements of $A$ appearing in such tensors (we may assume all of these to be in the unit ball of $A$), and take $\epsilon$ suitably small; we leave the elementary-but-messy details to the reader.  Using compactness of the unit interval $[0,1]$, this also implies that any homotopy of projections is $(X,\epsilon)$-compatible for suitable $(X,\epsilon)$.  The direct limit isomorphisms of lines \eqref{dir lim zn} and \eqref{dir lim int} follow readily from these observations: we leave the remaining details to the reader.

We also leave the direct checks that these isomorphisms induce an isomorphism in the category of $\Lambda_0$-modules to the reader.  The isomorphism of line \eqref{kappa def} is essentially just a restatement of the universal property of the direct limit in line \eqref{dir lim ex}, and the statement that it is a homeomorphism also follows directly from the definitions.
\end{proof}

We now come to the action of the controlled $K$-homology groups of Definition \ref{con k hom} on the controlled $K$-theory groups of Definition \ref{x eps com}.

\begin{lemma}\label{big Phi map 0}
Let $A$ be a unital $C^*$-algebra, let $X$ be a finite subset of the unit ball of $A$ such that $X=X^*$, and let $\epsilon\in (0,1/24)$.  Let $p$ be an element of the set $\mathcal{P}_\epsilon(X)$ of Definition \ref{con k hom}, and let $q$ be an element of the set $\mathcal{Q}^\epsilon(X;\Z/n\Z)$ of Definition \ref{x eps com}.  In particular, recall that $q\in A\otimes M_k(\C)\otimes O_{n+1}$ and $p\in M_2(\C)\otimes \mathcal{B}(H)$, where $A\subseteq \mathcal{B}(H)$ is a fixed infinitely amplified, faithful, unital representation, and $p$ is of the form $p=\begin{psmallmatrix} 1 & 0 \\ 0 & 0 \end{psmallmatrix} +b$ for some $b\in M_2(\mathcal{K}(H))$.  Finally, let $\chi$ be the characteristic function of $(1/2,\infty)$.

Then the formula 
\begin{equation}\label{pair form}
\langle p,q\rangle:=\chi\Big(\begin{psmallmatrix} 1 & 0 \\ 0 & 0 \end{psmallmatrix} + (1_{M_2(\C)}\otimes q)(b\otimes 1_{M_k(\C)\otimes O_{n+1}})(1_{M_2(\C)}\otimes q)\Big).
\end{equation}
gives a well-defined projection in $M_2(\mathcal{B}(H)\otimes M_k(\C)\otimes O_{n+1})$ whose difference with $\begin{psmallmatrix} 1 & 0 \\ 0 & 0 \end{psmallmatrix}$ is in $M_2(\mathcal{K}(H)\otimes M_k(\C)\otimes O_{n+1})$.

This works completely analogously on omitting $O_{n+1}$.
\end{lemma}

\begin{proof}
For simplicity of notation, let us write $q$ for $1_{M_2(\C)}\otimes q$, $b$ for $b\otimes 1_{M_k(\C)\otimes O_{n+1}}$, $p$ for $p\otimes 1_{M_k(\C)\otimes O_{n+1}}$, and $e$ for $\begin{psmallmatrix} 1 & 0 \\ 0 & 0 \end{psmallmatrix}$.  Note first that as $q$ is $(X,\epsilon)$-compatible (see Definition \ref{x eps com}), we have by Example \ref{xem ex} that 
$$
\text{spectrum}(pqp)\subseteq [0,1/4)\cup (3/4,1].
$$
Writing $p=e+b$ and using that $e$ commutes with $q$, we see that 
$$
pqp=eq+bqe+eqb+bqb;
$$
as $\|[q,p]\|=\|[q,b]\|<\epsilon$, this is within $6\epsilon$ of 
\begin{equation}\label{ebq}
eq+q(be+eb+b^2)q.
\end{equation}
As $6\epsilon<1/4$, we thus see that $eq+q(be+eb+b^2)q$ has spectrum contained in $[0,1/2)\cup (1/2,1]$.  On the other hand, as $(e+b)^2=p^2=p=(e+b)$, we have that $be+eb+b^2=b$, and so the element in line \eqref{ebq} equals $eq+qbq$.   As $e(1-q)$ is orthogonal to this projection, we see that the spectrum of $e+qbq$ is contained in $[0,1/2)\cup (1/2,1]$; however, up to our notational simplifications, this is the element 
$$
\begin{psmallmatrix} 1 & 0 \\ 0 & 0 \end{psmallmatrix} + (1_{M_2(\C)}\otimes q)(b\otimes 1_{M_k(\C)\otimes O_{n+1}})(1_{M_2(\C)}\otimes q)
$$
appearing in the statement of the lemma.  Hence $\chi$ is continuous on the spectrum of this element, and the functional calculus gives that the element $\chi(e+qbq)$ appearing in the statement is a well-defined projection in $M_2(\mathcal{B}(H)\otimes M_k(\C)\otimes O_{n+1})$.  On the other hand, on the spectrum of $e+qbq$, $\chi$ is a uniform limit of a sequence $(f_k)_{k=1}^\infty$ of polynomials with $f_k(0)=0$ and $f_k(1)=1$.  For any such polynomial, $f_k(e)=e$, and one computes directly that $f_k(e)-f_k(e+qbq)$ lies in the ideal $M_2(\mathcal{K}(H)\otimes M_k(\C)\otimes O_{n+1})$ for each $k$.  Taking the limit in $k$ completes the proof.
\end{proof}

The next result is the main goal of this subsection.  At this point, the proof consists of direct checks that we leave to the reader.

\begin{proposition}\label{phi lem}
Let $A$ be a unital $C^*$-algebra, let $X$ be a finite subset of the unit ball of $A$ such that $X=X^*$, and let $\epsilon\in (0,1/100)$.  Let $p$ be an element of the set $\mathcal{P}_\epsilon(X)$ of Definition \ref{con k hom}.
Then with notation as in Lemma \ref{big Phi map 0}, the formula  
$$
\Phi^{(X,\epsilon)}(p):[q]\to [\langle p,q\rangle]-\begin{bsmallmatrix} 1 & 0 \\ 0 & 0 \end{bsmallmatrix}
$$
induces a well-defined homomorphism 
$$
\Phi^{(X,\epsilon)}(p):K_0^\epsilon(X;\Z/n)\to K_0(\mathcal{K}(H)\otimes O_{n+1})\cong K_0(\C;\Z/n)
$$
and similarly
$$
\Phi^{(X,\epsilon)}(p):K_0^\epsilon(X)\to K_0(\mathcal{K}(H))\cong K_0(\C).
$$
These homomorphisms fit together to give a well-defined homomorphism 
$$
\Phi^{(X,\epsilon)}:K^0_\epsilon(X)\to \text{Hom}_{\Lambda_0}(\underline{K}_0^\epsilon(X),\underline{K}_0(\C))
$$
that is continuous when the left hand side is equipped with the discrete topology, and the right hand side has the topology of pointwise convergence.  

Finally, taking inverse limits over the set $\mathcal{I}_A$ of Definition \ref{dir set} and applying Lemma \ref{dir lim k thy} gives a homomorphism
\begin{equation}\label{ck hom}
\Phi:{\displaystyle \lim_{\leftarrow}K^0_\epsilon(X)}\to \text{Hom}_{\Lambda_0}(\underline{K}_0(A),\underline{K}_0(\C)),
\end{equation}
which is moreover continuous when $\text{Hom}_{\Lambda_0}(\underline{K}_0(A),\underline{K}_0(\C))$ is equipped with the topology of pointwise convergence. \qed
\end{proposition}

We will need one more piece of terminology and observation for later sections.

\begin{definition}\label{fin set tot hom}
Let $A$ be a unital $C^*$-algebra, and let $B_n$ be as in Definition \ref{lambda0}.  We define a \emph{$K$-datum} $P$ for $A$ to be a finite collection of projections $p$, each in some $A\otimes M_k(\C)\otimes B_{n(p)}$.  We let $N=N(P)$ for the least common multiple of $\{n(p)\mid p\in P\}$ as long as this set contains a non-zero integer, and zero if all $n(p)$ are zero.

With notation as in Remark \ref{l0n}, given a $K$-datum $P$, we let $\langle P\rangle$ be the $\Lambda_0^{(N)}$-submodule of $\underline{K}_0(A)$ that it generates as in Lemma \ref{l0 fg}, and we write $\text{Hom}_{\Lambda_0^{(N)}}(P,\underline{K}_0(B))$ for what should more properly be called $\text{Hom}_{\Lambda_0^{(N)}}(\langle P\rangle,\underline{K}_0(B))$.
\end{definition}

\begin{lemma}\label{partial kappa}
Let $A$ be a unital $C^*$-algebra, and let $B$ be another $C^*$-algebra.  For any $K$-datum $P$ as in Definition \ref{fin set tot hom} there exists an element $(X,\epsilon)$ of the directed set $\mathcal{I}_A$ of Definition \ref{dir set} with the following properties:
\begin{enumerate}[(i)]
\item \label{pk1} The canonical restriction map
$$
\text{Hom}_{\Lambda_0}(\underline{K}_0(A),\underline{K}_0(B)) \to \text{Hom}_{\Lambda_0^{(N)}}(P,\underline{K}_0(B))
$$
factors as the composition of the canonical homomorphism 
$$
\text{Hom}_{\Lambda_0}(\underline{K}_0(A),\underline{K}_0(B))\to \text{Hom}_{\Lambda_0}(\underline{K}_0^\epsilon(X),\underline{K}_0(B))
$$
arising from the isomorphism $\kappa$ of line \eqref{kappa def}, and a map 
$$
\kappa^{(X,\epsilon)}_P: \text{Hom}_{\Lambda_0}(\underline{K}_0^\epsilon(X),\underline{K}_0(\C))\to \text{Hom}_{\Lambda_0^{(N)}}(P,\underline{K}_0(\C)).
$$
\item \label{pk2} There exists $\gamma>0$ such that if $\phi,\psi:A\to M_n(\C)$ are ucp $(X,\epsilon)$-$*$-homomorphisms such that 
$$
\|\phi(x)-\psi(x)\|<\gamma
$$
for all $x\in X$, then with notation as in Lemma \ref{x eps lambda module} 
$$
\kappa^{(X,\epsilon)}_P(\phi_*)=\kappa^{(X,\epsilon)}_P(\psi_*).
$$
\end{enumerate}
\end{lemma}

\begin{proof}
Let $P$ be as given.  The existence of $(X,\epsilon)$ as in part \eqref{pk1} follows directly from Lemma \ref{l0 fg}.  

For \eqref{pk2}, enlarging $X$, we may assume that for each $p\in P$ there are $k,n\in \N$ such that $p$ is in the $C^*$-subalgebra of $A\otimes M_k(\C)\otimes O_{n+1}$ generated by elementary tensors of the form $a\otimes m\otimes o$ with $a\in X$, $m\in M_k(\C)$ and $o\in O_{n+1}$ (or with the $O_{n+1}$ omitted as appropriate).  Hence for any such $p$ we can force 
$$
\|(\psi\otimes \text{id})(p)-(\phi\otimes \text{id})(p)\|<1/10
$$
by choosing $\epsilon$ and $\gamma$ small enough (here we use Remark \ref{ucp alm mult}).  Part \eqref{pk2} follows from this.
\end{proof}

\subsection{The approximate $K$-homology UCT}

The Kasparov product defines a pairing 
$$
K^0(A)\to \text{Hom}_{\Lambda_0}(\underline{K}_0(A),\underline{K}_0(\C))
$$
of the usual $K$-homology group of a $C^*$-algebra $A$ with its total $K$-theory as in Examples \ref{basic lambda 0}.  This is continuous when $\text{Hom}_{\Lambda_0}(\underline{K}_0(A),\underline{K}_0(\C))$ is equipped with the topology of pointwise convergence: indeed, this a special case of continuity of the Kasparov product \cite[Theorem 3.5]{Dadarlat:2005aa}.  As the topology of pointwise convergence on $\text{Hom}_{\Lambda_0}(\underline{K}_0(A),\underline{K}_0(\C))$ is Hausdorff, this pairing necessarily sends the closure $\overline{\{0\}}$ of the trivial subgroup to zero, and so descends to a pairing 
\begin{equation}\label{kl hom}
\Psi: KL^0(A) \to \text{Hom}_{\Lambda_0}(\underline{K}_0(A),\underline{K}_0(\C))
\end{equation}
of the $KL$-homology group of line \eqref{kl def} with total $K$-theory.

We now need to appeal to the universal coefficient theorem of Rosenberg and Schochet \cite{Rosenberg:1987bh}.  We will not need the UCT at its full strength, `just' the property in the next definition.

\begin{definition}\label{kh auct}
A separable unital $C^*$-algebra $A$ satisfies the \emph{approximate $K$-homology UCT} if the continuous homomorphism $\Psi$ from line \eqref{kl hom} is a homeomorphism.
\end{definition}

The following theorem is a special case of a result of Dadarlat \cite[Theorem 4.1]{Dadarlat:2005aa}. 

\begin{theorem}[Dadarlat]\label{uct the}
Let $A$ be a separable, unital $C^*$-algebra that satisfies the UCT.  Then $A$ satisfies the approximate $K$-homology UCT of Definition \ref{kh auct}. \qed
\end{theorem}

\begin{remark}\label{auct rem}
The reader might compare the approximate $K$-homology UCT to the \emph{Approximate UCT} (AUCT) introduced by Lin in \cite{Lin:2005aa} as a bivariant version of Definition \ref{kh auct} above.  Dadarlat showed the AUCT to be equivalent to the full UCT for a nuclear $C^*$-algebra $A$ in \cite[Theorem 5.4]{Dadarlat:2005aa}.  Definition \ref{kh auct} can be thought of as the `AUCT for $K$-homology' (as opposed to for full bivariant $KK$-theory).  

Essentially the same argument as for Theorem \ref{uct the} shows that the approximate $K$-homology UCT is implied by Brown's $K$-homology UCT \cite{Brown:1984rx}, i.e.\ there are general implications
$$
\text{UCT} ~\Rightarrow~ \text{$K$-homology UCT} ~ \Rightarrow~ \text{approximate $K$-homology UCT}.
$$
We do not know if either implication is reversible (although it seems unlikely).  There are examples of $C^*$-algebras that do not satisfy the UCT: the first such examples come from \cite[Sections 4 and 5]{Skandalis:1988rr}, and these are still probably the most tractable known examples.  We do not know if there are $C^*$-algebras that do not satisfy the $K$-homology UCT.
\end{remark}

Having got through these preliminaries, our last aim in this subsection is to record the following result.  The proof is essentially bookkeeping, and more-or-less follows from results already in the literature; unfortunately, however, the isomorphism $\Theta$ of Theorem \ref{ck kl} is not defined very directly, so there is quite a lot of bookkeeping to do.

\begin{theorem}\label{kl ck hom}
Let $A$ be a separable unital $C^*$-algebra.  With $\Theta$, $\Phi$ and $\Psi$ as in Theorem \ref{ck kl}, Proposition \ref{phi lem}, and line \eqref{kl hom} respectively, the diagram below commutes
$$
\xymatrix{ KL^0(A) \ar[r]^-\Theta \ar[dr]_-{\Psi} & {\displaystyle \lim_{\leftarrow} K^0_\epsilon(X)} \ar[d]^-\Phi \\ & \text{Hom}_{\Lambda_0}(\underline{K}_0(A),\underline{K}_0(\C))}.
$$
Moreover, all the maps are continuous, and $\Theta$ is a homeomorphism.   If also $A$ satisfies the approximate $K$-homology UCT of Definition \ref{kh auct} then $\Phi$ and $\Psi$ are also homeomorphisms.
\end{theorem}

We need one more lemma before we get to the proof of this.

\begin{lemma}\label{asy alg}
Let $B$ be a stable\footnote{We warn the reader that something like the stability assumption is needed.  For example, the conclusion fails if $B=\C$.} $C^*$-algebra.  Let 
$$
\mathfrak{A}(B):=\frac{C_b([1,\infty),B)}{C_0([1,\infty),B)}
$$
be the asymptotic algebra in the sense of Connes-Higson \cite{Connes:1990if}.  Then the inclusion of $B$ in $\mathfrak{A}(B)$ as constant functions induces an isomorphism on $K$-theory.

Moreover, an inverse to this isomorphism can be defined as follows.  First, one shows that any element of $K_0(\mathfrak{A}(B))$ can be represented in the form $[p]-\begin{bsmallmatrix} 1 & 0 \\ 0 & 0 \end{bsmallmatrix}$, where $p$ is a projection in the $2\times 2$ matrices over the unitization of $\mathfrak{A}(B)$ whose difference with $\begin{psmallmatrix} 1 & 0 \\ 0 & 0 \end{psmallmatrix}$ is in $M_2(\mathfrak{A}(B))$.  Let $a$ be a self-adjoint element of the $2\times 2$ matrices over the unitization of $C_b([1,\infty),B)$ that lifts $p$ and is such that $a-\begin{psmallmatrix} 1 & 0 \\ 0 & 0 \end{psmallmatrix}$ is in $M_2(C_0([1,\infty),B)$.  Let $\chi$ be the characteristic function of $(1/2,\infty)$.  Then the inverse we want can be given by
$$
[p]-\begin{bsmallmatrix} 1 & 0 \\ 0 & 0 \end{bsmallmatrix}\mapsto [\chi (a(t))] - \begin{bsmallmatrix} 1 & 0 \\ 0 & 0 \end{bsmallmatrix} 
$$
for any sufficiently large $t\in [1,\infty)$.

If we instead define 
$$
\mathfrak{A}_u(B):=\frac{C_{ub}([1,\infty),B)}{C_0([1,\infty),B)}
$$
where $C_{ub}([1,\infty),B)$ denotes uniformly continuous and bounded functions from $[1,\infty)$ to $B$, then the same statements hold with $\mathfrak{A}_u(B)$ in place of $\mathfrak{A}(B)$.
\end{lemma}

\begin{proof}
The $C^*$-algebra $C_0([1,\infty),B)$ is contractible, so for the isomorphism result it suffices to show that the constant inclusion $B\to C_b([1,\infty),B)$ induces an isomorphism on $K$-theory.  For this, it suffices to show that the evaluation-at-one map $C_b([1,\infty),B)\to B$ induces an isomorphism on $K$-theory, and for this it suffices to show that $K_*(I)=0$, where $I:=\{f\in C_b([1,\infty),B)\mid f(1)=0\}$ is the kernel of that map.

Let then 
$$
I_O:=C_0((1,2],B)\prod_{n=1}^\infty C([2n+1,2n+2],B), \quad I_E:=\prod_{n=1}^\infty C([2n,2n+1],B)$$
and 
$$
I_Z:=\prod_{n=2}^\infty B.
$$
Then there is a pullback square of $C^*$-algebras and canonical restriction morphisms
\begin{equation}\label{pback}
\xymatrix{ I \ar[r] \ar[d] & I_O \ar[d] \\ I_E\ar[r] & I_Z }
\end{equation}
giving rise to a Mayer-Vietoris sequence 
\begin{equation}\label{mv}
\xymatrix{ K_0(I) \ar[r] &  K_0(I_O)\oplus K_0(I_E) \ar[r] & K_0(I_Z) \ar[d] \\
K_1(I_Z) \ar[u] & K_1(I_O)\oplus K_1(I_E) \ar[l] & K_1(I) \ar[l] }
\end{equation}
where the horizontal maps are induced by the morphisms in line \eqref{pback} (see for example \cite[Definition 2.7.14 and Proposition 2.7.15]{Willett:2010ay}).  On the other hand, using stability of $B$, the $K$-theory of the direct products splits 
$$
K_i(I_O)\cong K_i(C_0((1,2],B))\prod_{n=1}^\infty K_i(C([2n+1,2n+2],B))
$$
and similarly for $I_E$ and $I_Z$ (see for example \cite[Proposition 2.7.12]{Willett:2010ay}).  The required vanishing result follows readily from this and a computation based on the Mayer-Vietoris sequence in line \eqref{mv}.

For the construction of the inverse, it suffices to show that the given process induces a well-defined map on $K$-theory that is a one-sided inverse to the constant inclusion.  This is straightforward: we leave the details to the reader.

Finally, the uniformly continuous case follows from exactly the same argument (or an argument based on an Eilenberg swindle: compare \cite[page 229]{Willett:2010ay}).
\end{proof}

\begin{proof}[Proof of Theorem \ref{kl ck hom}]
Continuity of $\Psi$ is a special case of \cite[Theorem 3.5]{Dadarlat:2005aa} as already noted above. Continuity of $\Phi$ is contained in Proposition \ref{phi lem}.  Moreover, that $\Theta$ is a homeomorphism is covered in Theorem \ref{ck kl} above.  On the other hand, if $A$ satisfies the approximate $K$-homology UCT of Definition \ref{kh auct}, then $\Psi$ is a homeomorphism by definition.  Hence $\Phi$ will also be a homeomorphism in that case once we have established commutativity; this is the rest of the proof.

For commutativity, we use the following diagram
{\tiny \begin{equation}\label{2 tri}
\xymatrix{ KL^0(A) \ar@/^20pt/[rrrr]^-\Theta \ar@/_6pt/[drr]_-\Psi & K^0(A) \ar[l]^-{\alpha_1} \ar[r]_-{\alpha_2} \ar[dr]^-{\beta_1}  & K_0(C_{L,uc}^*(A)) \ar[r]_-{\alpha_3} \ar[d]^-{\beta_2} &  K_0(C^*_{L,c}(A)) \ar[r]_-{\alpha_4} \ar[dl]_-{\beta_3} &  {\displaystyle \lim_{\leftarrow} K^0_\epsilon(X)} \ar@/^6pt/[dll]^-\Phi \\
& & \text{Hom}_{\Lambda_0}(\underline{K}_0(A),\underline{K}_0(\C)) & &  }
\end{equation}}
that we explain now.  

We start with the top part of the diagram.  The map labeled $\alpha_1$ is the canonical quotient map from $K$-homology to $KL$-homology.  The $C^*$-algebra $C^*_{L,c}(A)$ (respectively, $C^*_{L,uc}(A)$) consists of bounded and continuous (respectively, uniformly continuous) functions $b:[1,\infty)\to \mathcal{K}(H)$ such that $[b(t),a]\to 0$ as $t\to\infty$ for all $a\in A$.  The map labeled $\alpha_2$ is defined and shown to be an isomorphism in \cite[Theorem 4.5]{Dadarlat:2016qc}\footnote{This is slightly off: the paper \cite{Dadarlat:2016qc} uses non-unital absorbing representations; it is not difficult to show that for unital $C^*$-algebras one may instead use a unital absorbing representation, and that this makes no difference on the level of $K$-theory, however.  This slight elision occurs at a few other points in the proof below, and is not a serious issue anywhere it appears.}.  The map labeled $\alpha_3$ is the canonical inclusion, and is shown to be an isomorphism in \cite[Theorem 3.4]{Willett:2020aa}.  The map labeled $\alpha_4$ is defined by first showing that every element of $C^*_{L,c}(A)$ can be represented in the form $[p]-\begin{bsmallmatrix} 1 & 0 \\ 0 & 0 \end{bsmallmatrix}$ for some projection $p$ in $M_2(C^*_{L,c}(A)^+)$ (here $C^*_{L,c}(A)^+$ is the unitization of $C^*_{L,c}(A)$) such that $p-\begin{psmallmatrix} 1 & 0 \\ 0 & 0 \end{psmallmatrix}\in M_2(C_{L,c}(A))$; then choosing $t_{X,\epsilon}\in [1,\infty)$ for each $(X,\epsilon)$ such that $\|[p_t,x]\|<\epsilon$ for all $t\geq t_{X,\epsilon}$ and all $x\in X$ and sending $[p]-\begin{bsmallmatrix} 1 & 0 \\ 0 & 0 \end{bsmallmatrix}$ to the element $([p_{t_{X,\epsilon}}])_{(X,\epsilon)\in \mathcal{I}_A}$ of the inverse limit.  

Now, the proofs of \cite[Theorem 4.14 and Proposition A.18]{Willett:2020aa} give that the map $\alpha_4\circ \alpha_3\circ \alpha_2:K^0(A)\to {\displaystyle \lim_{\leftarrow} K^0_\epsilon(X)}$ descends to an isomorphism on $KL$-homology; this is exactly the map we are calling $\Theta$, and therefore the top part of the diagram commutes. 

Let us now look at the lower triangles in diagram \eqref{2 tri}.  The map labeled $\beta_1$ is the standard pairing between $K$-homology and total $K$-theory; as the map $\Psi$ is by definition the map this induces on $KL$-homology, the left-most triangle in diagram \eqref{2 tri} commutes.  

The maps labeled $\beta_2$ and $\beta_3$ are special cases of the maps defined by Wang, Zhang, and Zhou in \cite[Definition 3.6]{Wang:2023ab} that we now spell out.  For notational simplicity, however, we just do the pairing with $K_0(A)$: the same argument works with ``$\,\cdot\,\otimes O_{n+1}$'' dragged through the proof for the groups with finite coefficients.  It is straightforward to see that \cite[Definition 3.6]{Wang:2023ab} is equivalent to the following description of the map $\beta_2$.  Any class in $K_0(C^*_{L,c}(A))$ can be represented in the form $[p]-\begin{bsmallmatrix} 1 & 0 \\ 0 & 0 \end{bsmallmatrix}$ where $p=\begin{psmallmatrix} 1 & 0 \\ 0 & 0 \end{psmallmatrix} +b$ for some $b\in M_2(C^*_{L,c}(A))$.  Let $q\in M_k(A)$ be a projection representing a class in $K_0(A)$.  Then the element $\begin{psmallmatrix} 1 & 0 \\ 0 & 0 \end{psmallmatrix}+1_{M2}\otimes qb\otimes 1_{M_k}$ defines a projection in $2\times 2$ matrices over the unitization of the asymptotic algebra $\mathfrak{A}_u(M_{k}(\mathcal{K}(H)))$ of Lemma \ref{asy alg} whose difference with $\begin{psmallmatrix} 1 & 0 \\ 0 & 0 \end{psmallmatrix}$ is in $\mathfrak{A}_u(M_{k}(\mathcal{K}(H)))$.  The map $\beta_2$ then takes $[p]-\begin{bsmallmatrix} 1 & 0 \\ 0 & 0 \end{bsmallmatrix}$ to the image of the class 
$$
\Big[\begin{psmallmatrix} 1 & 0 \\ 0 & 0 \end{psmallmatrix}+1_{M2}\otimes qb\otimes 1_{M_k}\Big]-\begin{bsmallmatrix} 1 & 0 \\ 0 & 0 \end{bsmallmatrix} \in K_0(\mathfrak{A}(\mathcal{K}(H))
$$
under the isomorphism $K_0(\mathfrak{A}_u(\mathcal{K}(H))\cong K_0(\mathcal{K}(H))\cong K_0(\C)$ from Lemma \ref{asy alg}.  The map $\beta_3$ is defined entirely analogously, but with $\mathfrak{A}_u(\mathcal{K}(H))$ replaced by $\mathfrak{A}(\mathcal{K}(H))$.  

Now the triangle in diagram \eqref{2 tri} involving $\alpha_3$, $\beta_2$ and $\beta_3$ commutes.  On the other hand, the triangle involving $\beta_1$, $\alpha_2$ and $\beta_2$ commutes by \cite[Theorem 3.7]{Wang:2023ab}.  It remains at this point to show that the diagram 
$$
\xymatrix{ K_0(C_{L,c}^*(A)) \ar[r]^-{\alpha_4} \ar[d]_-{\beta_3} & {\displaystyle \lim_{\leftarrow} K^0_\epsilon(X)} \ar[dl]^-\Phi \\
\text{Hom}(K_0(A),K_0(\C)) & }
$$
commutes (and similarly for the case with finite coefficients).   This follows from the description of the isomorphism of $K_0(\mathfrak{A}(\mathcal{K}(H)))\to K_0(\mathcal{K}(H))$ given in Lemma \ref{asy alg}, the description of $\beta_3$ above, and the definition of $\Phi$ from Lemma \ref{phi lem}, so we are done.
\end{proof}

\begin{remark}\label{homeo meaning}
As we will need it later, let us spell out a little more concretely what it means for the map
$$
\Phi:\lim_{\leftarrow} K^0_\epsilon(X) \to  \text{Hom}_{\Lambda_0}(\underline{K}_0(A),\underline{K}_0(\C))
$$
to be a homeomorphism; this turns out to be quite a powerful statement.  A neighborhood basis of the zero element in ${\displaystyle \lim_{\leftarrow} K^0_\epsilon(X)}$ is given by the subgroups $U_{Y,\delta}$ consisting of all elements in ${\displaystyle\lim_{\leftarrow} K^0_\epsilon(X)}$ that map to zero in $K^0_\delta(Y)$ as $(Y,\delta)$ ranges over the set $\mathcal{I}_A$ of Definition \ref{dir set}.  A neighborhood basis of the zero map in $\text{Hom}_{\Lambda_0}(\underline{K}_0(A),\underline{K}_0(\C))$ consists of the subsets $V_{P}$ of elements that restrict to the zero map in $\text{Hom}_{\Lambda_0^{(N)}}(P,\underline{K}_0(\C))$ as $P$ ranges over all $K$-data as in Definition \ref{fin set tot hom}.

Now, $\Phi$ being continuous means that for any $(Y,\delta)\in \mathcal{I}_A$ there exists a $K$-datum $P$ as above such that $\Phi$ descends to a well-defined map
$$
K^0_\delta(Y)\to \text{Hom}_{\Lambda_0^{(N)}}(P,\underline{K}_0(\C)).
$$
Indeed, the map above is the composition
$$
K^0_\delta(Y)\xrightarrow{\Phi^{(Y,\delta)}} \text{Hom}_{\Lambda_0}(\underline{K}_0^\delta(Y),\underline{K}_0(\C))\xrightarrow{\kappa^{(Y,\delta)}_P}\text{Hom}_{\Lambda_0^{(N)}}(P,\underline{K}_0(\C))
$$
of the maps of Proposition \ref{phi lem} and Lemma \ref{partial kappa}.

On the other hand, the fact that $\Phi^{-1}$ is continuous means that for any $K$-datum $P$ there exists $(Y,\delta)\in \mathcal{I}_A$ such that $\Phi^{-1}$ descends to a well-defined map 
$$
(\Phi^{-1})^P_{(Y,\delta)}:\text{Hom}_{\Lambda_0^{(N)}}(P,\underline{K}_0(\C))\to K^0_\delta(Y).
$$
Finally these maps are consistent with the inverse limit structures when they are defined, i.e.\ we have commutative diagrams of homomorphisms
$$
\xymatrix{ & K^0_\delta(Y) \ar[dr]^-{\kappa^{(Y,\delta)}_P\circ \Phi^{(Y,\delta)}} & \\
 \text{Hom}_{\Lambda_0^{(M)}}(Q,\underline{K}_0(\C))\ar[rr] \ar[ur]^-{(\Phi^{-1})^Q_{(Y,\delta)}} & & \text{Hom}_{\Lambda_0^{(N)}}(P,\underline{K}_0(\C))} 
$$
and
$$
 \xymatrix{ K^0_\delta(Y) \ar[dr]_-{\kappa^{(Y,\delta)}_P\circ \Phi^{(Y,\delta)}} \ar[rr] & & K^0_\epsilon(X) \\
&  \text{Hom}_{\Lambda_0^{(N)}}(P,\underline{K}_0(\C)) \ar[ur]_-{(\Phi^{-1})^P_{(X,\epsilon)}} & } 
$$
where defined (here the horizontal maps are forgetful maps).  

In summary, the fact that $\Phi$ is a homeomorphism on the inverse limits means that it comes from a coherent system of maps between the inverse systems.  For $\Phi$ this is essentially true by definition; the more difficult (and mysterious) fact is that it also holds for $\Phi^{-1}$. 
\end{remark}

\section{Stable uniqueness for representations of $C^*$-algebras}\label{su sec}

Our goal in this section is to establish a stable\footnote{``Stable'' is used in the sense of $K$-theory: i.e.\ `up to taking block sum with something'.  It is not directly connected to representation stability as in Definition \ref{stable}.} uniqueness theorem, following Dadarlat-Eilers \cite[Section 4]{Dadarlat:2002aa} and Lin \cite{Lin:2002aa}.  Very roughly, such a theorem says that given two ucp approximately multiplicative maps between $C^*$-algebras, if they agree on a large enough subset of $K$-theory, then they are approximately unitarily equivalent after taking block sum with a homomorphism.  

We have not attempted to state our stable uniqueness result in maximal generality.  For example, a version for the bivariant theory $KL(A,B)$ should be possible under appropriate assumptions\footnote{For our methods to work, these probably have to include nuclearity of $B$.}, but we will not pursue that here.

The following definition is a slight variant on \cite[Definition 3.7]{Dadarlat:2002aa}.

\begin{definition}\label{k0 trip}
Let $A$ be a unital $C^*$-algebra.  A \emph{$\underline{K}_0$-triple} for $A$ is a triple $(P,X,\epsilon)$ where:
\begin{enumerate}[(i)]
\item $X$ is a finite self-adjoint subset of the unit ball of $A$;
\item $\epsilon$ is a positive real number;
\item $P$ is a $K$-datum as in Definition \ref{fin set tot hom} such that the map 
$$
\kappa^{(X,\epsilon)}_P: \text{Hom}_{\Lambda_0}(\underline{K}_0^\epsilon(X),\underline{K}_0(\C))\to \text{Hom}_{\Lambda_0^{(N)}}(P,\underline{K}_0(\C))
$$
of Lemma \ref{partial kappa} part \eqref{pk1} above is defined.
\end{enumerate}
\end{definition}

The following definition will be technically convenient.

\begin{definition}\label{good rep}
Let $A$ be a $C^*$-algebra, and let $\mathcal{R}$ be a class of unitary representations of $A$.  

The collection is said to be \emph{good} if it consists only of finite-dimensional unital representations of $A$, is closed under finite direct sums and unitary equivalence, and if for any $a\in A$ there is $\pi$ in $\mathcal{R}$ with $\pi(a)\neq 0$\footnote{The last condition (i.e.\ that representations from $\mathcal{R}$ separate points in $A$) is the same as $\mathcal{R}$ being dense in the \emph{Fell topology}, i.e.\ the topology introduced by Fell and called the \emph{inner hull-kernel} topology in \cite[Section 2]{Fell:1962aa}}.
\end{definition}

\begin{example}\label{good ex}
We will apply Definition \ref{good rep} to the following two examples: the class $\mathcal{R}_d$ of all finite-dimensional representations of $A$; and if $A=C^*(\Gamma)$, the class $\mathcal{R}_q$ of all finite-dimensional representations of $A$ that (when considered as representations of $\Gamma$) factor through a finite quotient.  If $\mathcal{R}_d$ is good for $A$ is said to be \emph{residually finite dimensional} (RFD), and if $\mathcal{R}_q$ is good for $C^*(\Gamma)$, $\Gamma$ is said to have \emph{property FD} of Lubotzky-Shalom \cite{Lubotzky:2004xw}.
\end{example}

Here is the main theorem of this section.  

\begin{theorem}\label{stab main}
Let $A$ be a separable, unital $C^*$-algebra that satisfies the approximate $K$-homology UCT of Definition \ref{kh auct}.  Let $\mathcal{R}$ be a good class of representations of $A$ as in Definition \ref{good rep}\footnote{So in particular, $A$ is RFD as in Example \ref{good ex}.}.

Let $\mathcal{X}$ be a family of finite, symmetric subsets of the unit ball of $A$ that is closed under finite unions, and such that $\bigcup_{X\in \mathcal{X}}X$ generates $A$ as a $C^*$-algebra.  Then for any $X\in \mathcal{X}$ and $\epsilon>0$ there exists a $\underline{K}_0$-triple $(P,Y,\delta)$ as in Definition \ref{k0 trip} with $Y\in \mathcal{X}$ and with the following property.

Let $\phi^0,\phi^1:A\to M_n(\C)$ be ucp $(Y,\delta)$-$*$-homomorphisms as in Definition \ref{x eps hom}, and let 
$$
\kappa^{(Y,\delta)}_P(\phi^i_*)\in \text{Hom}_{\Lambda_0^{(N)}}(P,\underline{K}_0(\C)),\quad i\in \{0,1\}
$$
be the homomorphisms given by Lemma \ref{x eps lambda module}, Lemma \ref{partial kappa} and the definition of a $\underline{K}_0$-triple.  Assume moreover that $\kappa^{(Y,\delta)}_P(\phi^0_*)=\kappa^{(Y,\delta)}_P(\phi^1_*)$.

Then there exist a $*$-representation $\theta:A\to M_k(\C)$ in $\mathcal{R}$ and a unitary $u\in M_{n+k}(\C)$ such that 
$$
\|u(\phi^0(x)\oplus \theta(x))u^*-(\phi^1(x)\oplus \theta(x))\|<\epsilon
$$
for all $x\in X$.
\end{theorem}

\begin{remark}
In general, one could take the collection $\mathcal{X}$ in Theorem \ref{stab main} to be the set of all finite symmetric subsets of the unit ball (or unit sphere) of $A$, and we could have just stated the theorem in this way.  The purpose of including the collection $\mathcal{X}$ in the statement is to strengthen the conclusion in special cases where a good choice of $\mathcal{X}$ is possible.  For example, if $A$ is finitely generated by some set symmetric subset $X$ of its unit ball, one can take $\mathcal{X}=\{X\}$.  Importantly for us, if $A=C^*(\Gamma)$ for a group $\Gamma$, we can take $\mathcal{X}$ to consist of all finite symmetric subsets of $\Gamma$.
\end{remark}

As already mentioned, Theorem \ref{stab main} is inspired by work of Lin \cite[Theorem 4.8]{Lin:2002aa} and Dadarlat-Eilers \cite[Theorem 4.15]{Dadarlat:2002aa}.  Theorem \ref{stab main} above has the advantage over these earlier results that it requires no nuclearity or exactness assumptions on $A$.  This is important for our applications as group $C^*$-algebras of non-amenable groups are rarely (and possibly never) exact: see \cite[Proposition 7.1]{Kirchberg:1993aa} or \cite[Proposition 3.7.11]{Brown:2008qy}.   On the other hand, Theorem \ref{stab main} has the disadvantage over the work of Lin and Dadarlat-Eilers that there is no control over the the `auxiliary representation' $\theta$: in their setting Lin and Dadarlat-Eilers are able to show that the choices of $\theta$ can be constrained to lie in a finite set that depends only on $X$ and $\epsilon$.

We need some technical lemmas before we get to the proof of Theorem \ref{stab main}.  Variants of the first are well-known; we could not find exactly what we need in the literature so give a proof.  Some of the argument is quite similar to that from Proposition \ref{ucp extend} part \eqref{approx mult dom}

\begin{lemma}\label{ucp mult com}
Let $A$ be a separable\footnote{One could remove the separability assumptions using Hadwin's generalization \cite{Hadwin:1981aa} of Voiculescu's theorem, and on replacing the condition ``$A\cap \mathcal{K}(H)=\{0\}$'' by an appropriate analogue.  However, we need the separability conditions for other ingredients of our main results, and the extra technicality did not seem worth it.}, unital $C^*$-algebra.  Let $H$ be a separable Hilbert space, and let $A$ be included in $\mathcal{B}(H)$ via a faithful unital representation such that $A\cap \mathcal{K}(H)=\{0\}$.   Let $\phi:A\to M_n(\C)$ be a ucp map.  The following hold:
\begin{enumerate}[(i)]
\item \label{isom cut} for any $\epsilon>0$ and finite subset $Y$ of $A$, there is an isometry $w:\C^n\to H$ such that 
$$
\|w^*aw-\phi(a)\|<\epsilon \quad \text{for all} \quad a\in Y;
$$
\item \label{cut com} if $\phi$ is in addition an $(X,\delta)$-$*$-homomorphism in the sense of Definition \ref{x eps hom} for some finite symmetric subset $X$ of the unit ball of $A$, $w$ is as in part \eqref{isom cut} for some $\epsilon>0$ and $Y\supseteq X\cup X^2$, and $p:=ww^*$, then 
$$
\|[p,a]\|<\sqrt{3\epsilon+\delta}\quad \text{for all} \quad a\in X.
$$
\end{enumerate}
\end{lemma}

\begin{proof}
Part \eqref{isom cut} is immediate from the finite-dimensional version of Voiculescu's theorem: see for example \cite[Proposition 1.7.1]{Brown:2008qy} or \cite[Proposition 3.6.7]{Higson:2000bs}.

For part \eqref{cut com}, define $p:=ww^*$ with $w$ satisfying the given conditions.  Then for any $a\in X$,
\begin{align}\label{3ed}
\|pa(1-p)a^*p\| & = \|w^*a(1-ww^*)a^*w\| \nonumber \\
& = \|w^*aa^*w-w^*aww^*a^*w\| \nonumber \\
& < \|\phi(aa^*)-\phi(a)\phi(a^*)\| +3\epsilon \nonumber \\
& < \delta+3\epsilon
\end{align}
where the inequalities use respectively the inequality from part \eqref{isom cut} and the fact that $\phi$ is $(X,\delta)$-multiplicative.  Using the identity $\|bb^*\|=\|b^*b\|$ with $b=pa^*(1-p)$ and interchanging the roles of $a$ and $a^*$ in line \eqref{3ed}, we see also that
\begin{equation}\label{3ed2}
\|(1-p)apa^*(1-p)\|<\delta+3\epsilon
\end{equation}
for all $a\in X$.  On the other hand, the $C^*$-identity and orthogonality imply that for any $a\in A$
\begin{align*}
\|[p,a]\|^2 & = \|pa(1-p)+(1-p)ap\|^2 \\
& = \|pa(1-p)a^*p+(1-p)apa^*(1-p)\| \\
& = \max\{\|pa(1-p)a^*p\|,\|(1-p)apa^*(1-p)\|\},
\end{align*}
so we are done by lines \eqref{3ed} and \eqref{3ed2}.
\end{proof}

The next result is again more-or-less folklore: a proof can be found in \cite[Lemma 2.20]{Willett:2021te}.

\begin{lemma}\label{alm proj com}
Let $\delta\in (0,1/2)$, let $a$ be a self-adjoint element in a $C^*$-algebra $A$ whose spectrum does not intersect $(\delta,1-\delta)$, and let $\chi$ be the characteristic function of $(1/2,\infty)$.  Then for any $b\in A$,
$$
\|[\chi(a),b]\|\leq \frac{1}{1-2\delta}\|[a,b]\|. \eqno\qed
$$
\end{lemma}

The next lemma is a version of the basic fact from operator $K$-theory that nearby projections are unitarily equivalent (see for example \cite[Proposition 2.2.6]{Rordam:2000mz}), that also gives control on commutators.  The proof is the same as the standard one with a little extra work to keep track of commutator estimates throughout the construction.

\begin{lemma}\label{path to uni}
Let $A$ be a unital $C^*$-algebra, let $X$ be a finite, symmetric subset of the unit ball of $A$, and let $\epsilon>0$.  Let $p,q$ be projections in $A$ such that $\|p-q\|<1/4$, and such that $\|[p,x]\|<\epsilon$ and $\|[q,x]\|<\epsilon$ for all $x\in X$.  Then there is a unitary $u\in A$ such that $upu^*=q$ and $\|[u,x]\|< 28\epsilon$ for all $x\in X$.
\end{lemma}

\begin{proof}
Define $v:=pq+(1-p)(1-q)$ and note that $vq=pv$ and that 
\begin{equation}\label{v x 4}
\|[v,x]\|<4\epsilon \quad\text{for all $x\in X$}.  
\end{equation}
Moreover 
$$
\|1-v\|=\|(2q-1)(p-q)\|=\|p-q\|<1/4,
$$
whence $v$ is invertible, 
\begin{equation}\label{v 54}
\|v\|< 5/4
\end{equation}
and 
$$
\|1-v^*v\|\leq \|(1-v^*)v\|+\|1-v\|<9/16.
$$
Hence also
\begin{equation}\label{v*v est}
v^*v\geq 7/16
\end{equation}
and so by the functional calculus 
\begin{equation}\label{v*v 1/2 est}
\|(v^*v)^{-1/2}\|\leq 4/\sqrt{7}.
\end{equation}
Define now $u:=v(v^*v)^{-1/2}$, which is unitary.  As $p$ and $q$ are self-adjoint, we see that $qv^*=v^*p$ and so $v^*v$ commutes with $p$ and $q$.  Hence $(v^*v)^{-1/2}$ also commutes with $p$ and $q$, and so $upu^*=q$.  It remains to estimate $\|[u,x]\|$ for each $x\in X$.  For this, we have 
\begin{align}\label{com1}
\|[u,x]\| & \leq \|[v,x]\|\|(v^*v)^{-1/2}\|+\|v\|\|[x,(v^*v)^{-1/2}]\| \nonumber  \\
& < 4\epsilon\cdot (4/\sqrt{7})+(5/4)\cdot \|[x,(v^*v)^{-1/2}]\|
\end{align}
by lines \eqref{v*v 1/2 est} and \eqref{v 54}.  To estimate $\|[x,(v^*v)^{-1/2}]\|$ we use the identity 
$$
t^{-1/2}=\frac{2}{\pi} \int_0^\infty (\lambda^2+t)^{-1}d\lambda,
$$
valid for any $t>0$.  Using the identity $[x,a^{-1}]=a^{-1}[a,x]a^{-1}$ we see that
\begin{equation}\label{x v*v}
\|[x,(v^*v)^{-1/2}]\|\leq \frac{2}{\pi} \int_0^\infty \|(\lambda^2+v^*v)^{-1}\|\|[v^*v,x]\|\|(\lambda^2+v^*v)^{-1}\|d\lambda.
\end{equation}
Using line \eqref{v*v est}, we see that $\|(\lambda^2+v^*v)^{-1}\|\leq (\lambda^2+7/16)^{-1}$, and so line \eqref{x v*v} implies that 
$$
\|[x,(v^*v)^{-1/2}]\|\leq \frac{2\|[x,v^*v]\|}{\pi}\int_0^\infty (\lambda^2+7/16)^{-2}d\lambda.
$$
Computing, the integral in the line above equals $16\pi / 7^{3/2}$.  Moreover, lines \eqref{v x 4} and \eqref{v 54} imply that $\|[x,v^*v]\|< 10\epsilon$, so we get
$$
\|[x,(v^*v)^{-1/2}]\|< 320\epsilon / 7^{3/2}.
$$
Substituting this back into line \eqref{com1} gives $\|[u,x]\|< 28\epsilon$ as claimed.
\end{proof}

The next lemma is essentially contained in \cite[Section 4.4]{Willett:2021te}: see in particular \cite[Proposition 4.17]{Willett:2021te}.  As that reference does not contain exactly what we need, we sketch a proof.  

\begin{lemma}\label{conj lem}
Let $A$ be a unital $C^*$-algebra, let $X$ be a finite, symmetric subset of the unit ball of $A$, and let $\epsilon>0$.  Let $(p_t)_{t\in [0,1]}$ be a path of projections in $A$ such that $\|[p_t,x]\|<\epsilon$ for all $t\in [0,1]$.  Then there exists $n$ and a unitary $u\in M_{2n+1}(A)=A\otimes M_{2n+1}(\C)$ such that 
\begin{equation}\label{big conj}
u(p_0\oplus 1_n\oplus 0_n)u^*=p_1\oplus 1_n\oplus 0_n
\end{equation}
and $\|[u,x\otimes 1_{2n+1}]\|<4\cdot 10^4\epsilon$ for all $x\in X$.
\end{lemma}

\begin{proof}
Following \cite[Lemma 4.15]{Willett:2021te}, there is $n$ and a $16$-Lipschitz path $(q_t)_{t\in [0,1]}$ of projections between $p_0\oplus 1_n\oplus 0_n$ and $p_1\oplus 1_n\oplus 0_n$ in $M_{2n+1}(A)$ with the property that $\|[x,q_t]\|< 21\epsilon$ for all $x\in X$ (here and throughout we abuse notation and write ``$x$'' in place of ``$x\otimes 1_{2n+1}$'').  Hence we may find $0=t_0<t_1<\cdots <t_{65}=1$ such that $\|q_{t_i}-q_{t_{i+1}}\|<1/4$ for all $i\in \{0,...,64\}$.  Lemma \ref{path to uni} then gives unitaries $u_i$ such that $u_iq_{t_i}u_i^*=q_{t_{i+1}}$ and so that 
\begin{equation}\label{ui x}
\|[u_i,x]\|<588 \epsilon
\end{equation}
for all $x\in X$ and all $i\in \{0,...,64\}$.  Define $u=u_{64}u_{63}\cdots u_1u_0$.  Then $u$ is a unitary that satisfies the equation in line \eqref{big conj}, and we have that for any $x\in X$,
$$
\|[u,x]\|\leq \sum_{i=0}^{64} \|[u_i,x]\|
$$ 
by repeated use of the Leibniz rule.  Combined with line \eqref{ui x}, this implies the desired estimate.
\end{proof}

\begin{proof}[Proof of Theorem \ref{stab main}]
Let $X$ and $\epsilon$ be as in the statement; we may assume $\epsilon<1$.  The discussion in Remark \ref{homeo meaning} implies that there exists a $\underline{K}_0$-triple $(P,Y,\delta)$ such that if $[p],[q]\in K^0_{2\sqrt{\delta}}(Y)$ are such that (with notation as in Proposition \ref{phi lem} and Lemma \ref{partial kappa})
$$
\kappa^{(Y,2\sqrt{\delta})}_P(\Phi^{(Y,2\sqrt{\delta})}(p))=\kappa^{(Y,2\sqrt{\delta})}_P(\Phi^{(Y,2\sqrt{\delta})}(q))
$$
then the images of $[p]$ and $[q]$ under the forgetful map
$$
K^0_{2\sqrt{\delta}}(Y)\to K^0_{\epsilon/(8\cdot 10^5)}(X)
$$
of Remark \ref{dir set rem}, part \eqref{inc part} are the same.   Using that the collection $\{(Z,\gamma)\mid Z\in \mathcal{X},0<\gamma<\epsilon/4\}$ is cofinal in the directed set $\mathcal{I}_A$ of Definition \ref{dir set}, we may moreover assume that $Y$ is in $\mathcal{X}$ and $\delta<\epsilon/4$.  We claim this $Y$ and $\delta$ have the property in the theorem.  

Let $\phi^0,\phi^1:A\to M_n(\C)$ be ucp $(Y,\delta)$-$*$-homomorphisms as in the statement.  In particular, if we use that $2\sqrt{\delta}\geq \delta$ to think of each $\phi^i$ as a $(Y,2\sqrt{\delta})$-$*$-homomorphism then
\begin{equation}\label{pi and phi}
\kappa^{(Y,2\sqrt{\delta})}_P(\phi^0_*)=\kappa^{(Y,2\sqrt{\delta})}_P(\phi^1_*).
\end{equation}

As $A$ is separable, there is a countable family of representations from $\mathcal{R}$ that separates points.  Take the direct sum of this family, then take the direct sum of that representation with itself countably many times.  We use this representation to identify $A$ with a unital $C^*$-subalgebra of $\mathcal{B}(H)$ (with $H$ the separable Hilbert space underlying the representation we just built), and to define the controlled $K$-homology groups of Definition \ref{con k hom}.  

Let $\gamma>0$ be as in the conclusion of Lemma \ref{partial kappa} part \eqref{pk2} for $(Y,2\sqrt{\delta})$; we may assume that $\gamma<\delta$.  For $i\in \{0,1\}$, Lemma \ref{ucp mult com} gives an isometry $w_i:\C^n\to H$ such that if $p_i:=w_iw_i^*$ then $\|w_i^*a  w_i-\phi^i(a)\|<\gamma$ and $\|[p_i,a]\|<2\sqrt{\delta}$ for all $a\in Y$.  As the representation $A\subseteq \mathcal{B}(H)$ is a direct sum of finite-dimensional representations, we may assume on perturbing  $w_i$ that $p_i$ is dominated by a finite rank projection in the commutant of $A$ with image a representation from $\mathcal{R}$, and that the perturbed operators still satisfy $p_i=w_iw_i^*$ and 
\begin{equation}\label{pphi est}
\|[p_i,a]\|<2\sqrt{\delta} \quad \text{and}\quad \|w^*_iaw_i-\phi^i(a)\|<\gamma \quad \text{for all}\quad a\in Y.
\end{equation}

For $i\in \{0,1\}$, define 
\begin{equation}\label{phi' def}
\psi^i:A\to \mathcal{B}(p_i H),\quad a\mapsto p_i a p_i
\end{equation}
to be the ucp map given by compression by $p_i$.  Using line \eqref{pphi est}, the choice of $\gamma$, and unitary invariance of $K$-theory, we have that 
\begin{equation}\label{phi and phi'}
\kappa^{(Y,2\sqrt{\delta})}_P(\phi^i_*)=\kappa^{(Y,2\sqrt{\delta})}_P(\psi^i_*).
\end{equation}
Define 
$$
q_i:=\begin{pmatrix} 1 & 0 \\ 0 & p_i\end{pmatrix}\in M_2(\mathcal{B}(H))
$$
and note that by line \eqref{pphi est}, $q_i$ defines an element of the set $\mathcal{P}_{5\sqrt{\delta}}(Y)$ of Definition \ref{con k hom}.   Note also that  \begin{equation}\label{psi and q}
\kappa^{(Y,2\sqrt{\delta})}_P(\Phi^{(Y,2\sqrt{\delta})}(q_i))=\kappa^{(Y,2\sqrt{\delta})}_P(\psi^i_*(q_i))
\end{equation}
for $i\in \{0,1\}$ (by definition - see Proposition \ref{phi lem} and line \eqref{pair form} above).  Combining lines \eqref{pi and phi}, \eqref{phi and phi'}, and \eqref{psi and q} we have that 
$$
\kappa^{(Y,2\sqrt{\delta})}_P(\Phi^{(Y,2\sqrt{\delta})}(q_0))=\kappa^{(Y,2\sqrt{\delta})}_P(\Phi^{(Y,2\sqrt{\delta})}(q_1))
$$
in $\text{Hom}_{\Lambda_0^{(N)}}(P,\underline{K}_0(\C))$ and therefore by choice of $(Y,\delta)$ in the first paragraph of the proof, the images of $[q_0]$ and $[q_1]$ under the forgetful map
$$
K^0_{2\sqrt{\delta}}(Y)\to K^0_{\epsilon/(8\cdot 10^5)}(X)
$$
of Remark \ref{dir set rem}, part \eqref{inc part} are the same.

Hence with notation as in Definition \ref{con k hom} there is a homotopy $(q_t)_{t\in [0,1]}$ passing through the space $\mathcal{P}_{\epsilon/8\cdot 10^5}(X)$ of Definition \ref{con k hom} and connecting $q_0$ to $q_1$.  As the image of this homotopy is a compact subset of $M_2(\mathcal{K}(H))+\begin{psmallmatrix} 1 & 0 \\ 0 & 0 \end{psmallmatrix}$, and as the original representation on $H$ is a direct sum of finite-dimensional representations from $\mathcal{R}$, and as $p_0$ and $p_1$ are dominated by finite rank projections in the commutant of $A$ with image representations from $\mathcal{R}$, there is a finite rank projection $r\in \mathcal{B}(H)$ that commutes with $A$, so that compression by $r$ is a representation from $\mathcal{R}$, so that $r$ dominates $p_0$ and $p_1$, and so that $r$ also satisfies 
$$
\Bigg\|\begin{pmatrix} r & 0 \\ 0 & r \end{pmatrix} \Bigg(q_t-\begin{pmatrix} 1 & 0 \\ 0 & 0 \end{pmatrix}\Bigg) -  \Bigg(q_t-\begin{pmatrix} 1 & 0 \\ 0 & 0 \end{pmatrix}\Bigg)\Bigg\|<\frac{\epsilon}{4\cdot 10^5}
$$
for all $t\in [0,1]$.   Rearranging the above inequality, we see that for all $t$
$$
\Bigg\|\begin{pmatrix} 1-r & 0 \\ 0 & 1-r \end{pmatrix} q_t - \begin{pmatrix} 1-r & 0 \\ 0 & 0 \end{pmatrix}\Bigg\|<\frac{\epsilon}{4\cdot 10^5}.
$$
Taking adjoints gives 
$$
\Bigg\|q_t\begin{pmatrix} 1-r & 0 \\ 0 & 1-r \end{pmatrix} - \begin{pmatrix} 1-r & 0 \\ 0 & 0 \end{pmatrix}\Bigg\|<\frac{\epsilon}{4\cdot 10^5}
$$
and combining the last two displayed lines gives
$$
\Bigg\|q_t\begin{pmatrix} 1-r & 0 \\ 0 & 1-r \end{pmatrix}-\begin{pmatrix} 1-r & 0 \\ 0 & 1-r \end{pmatrix} q_t\Bigg\|<\frac{\epsilon}{2\cdot 10^5}.
$$
Hence $\|[q_t,\begin{psmallmatrix} r & 0 \\ 0 & r \end{psmallmatrix}]\|<\epsilon/(2\cdot 10^5)$, and so if $a_t:=\begin{psmallmatrix} r & 0 \\ 0 & r \end{psmallmatrix} q_t\begin{psmallmatrix} r & 0 \\ 0 & r \end{psmallmatrix}$, then $\|a_t^2-a_t\|<\epsilon/(2\cdot 10^5)<1/4$.  In particular, if $\chi$ is the characteristic function of $(1/2,\infty)$, then $\chi$ is continuous on the spectrum of $a_t$.  Moreover, as $r$ commutes with $A$ and as each $q_t$ commutes with all elements of $X$ up to $\epsilon/(2\cdot 10^5)$, Lemma \ref{alm proj com} gives that $\chi(a_t)$ commutes with all elements of $X$ up to $\epsilon/10^5$.  Define $C:=M_2(\mathcal{B}(rH))$, and let $\sigma:A\to C$ be the representation (in $\mathcal{S}$) given by compression by $r\oplus r$.  We thus have a homotopy $(\chi(a_t))_{t\in [0,1]}$ between $r\oplus p_0$ and $r\oplus p_1$ in $M_2(\mathcal{B}(rH))$ that commutes with $\sigma(a)$ up to $\epsilon/10^5$ for all $a\in X$.  

Now, flipping the order of $r\oplus p_i$ to $p_i\oplus r$, Lemma \ref{conj lem} applied to the projections $p_i\oplus r$ gives $l\in \N$ and a unitary $v\in M_{1+2l}(C)$ such that 
\begin{equation}\label{vpq}
v(p_0\oplus r\oplus 1_l \oplus 0_l)v^* = (p_1\oplus r\oplus 1_l\oplus 0_l)
\end{equation}
(where $1_l$ and $0_l$ denote the unit and zero element respectively in $M_l(C)$) and moreover so that if $\widetilde{\sigma}:A\to M_{1+2l}(C)$ is the amplification of $\sigma$, then  
\begin{equation}\label{va small}
\|[v,\widetilde{\sigma}(a)]\|<\epsilon/2\quad \text{for all} \quad a\in X.
\end{equation}

Define now $s_i:=p_i\oplus r\oplus 1_l\oplus 0_l$ for $i\in \{0,1\}$.  Let $k=\text{rank}(r)(l+1)$, and let $x:\C^k\to rH\oplus rH\oplus (rH)^{\oplus l}\oplus (rH)^{\oplus l}$ be any isometry with range the middle two direct summands $rH\oplus (rH)^{\oplus l}$.  Define $\theta:A\to M_k(\C)$ by $\theta(a):=x^*\widetilde{\sigma}(a)x$, which is a representation of $A$ in the class $\mathcal{R}$.  Finally, define 
$$
u\in M_{n+k}(\C),\quad u=(w_1^*\oplus x^*)v(w_0\oplus x)\in M_{n+k}(\C),
$$
where $w_i$ is as in line \eqref{pphi est}, but adjusted so that its image is the range of $p_i$ considered as a subspace of the first summand of $rH\oplus rH\oplus (rH)^{\oplus l}\oplus (rH)^{\oplus l}$.  We claim that this $u$ and $\theta$ have the desired properties.

Indeed, for any $a\in X$,
\begin{align*}
\|u(\phi^0(a)& \oplus \theta(a))u^*-\phi^1(a)\oplus \theta(a)\| \\& = \|v(w_0\phi_0(a)w_0^*\oplus x\theta(a)x^*)v^* - w_1\phi_1(a)w_1^*\oplus x\theta(a)x^*\|.
\end{align*}
Using the estimate in line \eqref{pphi est}, we thus have that 
\begin{align*}
\|u(\phi^0(a) & \oplus \theta(a))u^*-\phi^1(a)\oplus \theta(a)\| \\ & < \|v(p_0ap_0\oplus x\theta(a)x^*)v^* - (p_1ap_1\oplus x\theta(a)x^*)\|+2\gamma.
\end{align*}
Using the definition of $s_i$, that $vs_0=s_1v^*$, and that $\gamma<\delta<\epsilon/4$, we thus have that 
\begin{align*}
\|u(\phi^0(a) \oplus \theta(a))u^*-\phi^1(a)\oplus \theta(a)\| & < \|v(s_0\widetilde{\sigma}(a)s_0v^* - (s_1\widetilde{\sigma}(a)s_1)\| +\epsilon/2 \\ 
& =\|s_1(v\widetilde{\sigma}(a)v^*-\widetilde{\sigma}(a))s_1\|+\epsilon/2 \\
&\leq \|v\widetilde{\sigma}(a)v^*-\widetilde{\sigma}(a)\|+\epsilon/2.
\end{align*}
Finally, using line \eqref{va small}, we conclude that 
$$
\|u(\phi^0(a) \oplus \theta(a))u^*-\phi^1(a)\oplus \theta(a)\| <\epsilon
$$
and are done.
\end{proof}

\section{Stable uniqueness for representations of groups}\label{main sec}

Before getting to the main result on representation stability we one more technical lemma.  



\begin{lemma}\label{to ucp}
Let $\Gamma$ be a countable discrete group with the weak matricial LLP as in Definition \ref{llp}.  Let $S$ be a finite symmetric subset of $\Gamma$, let $\epsilon>0$, and let $P$ be a $K$-datum as in Definition \ref{fin set tot hom} such that the homomorphism 
$$
\kappa^{(S,\epsilon)}_P: \text{Hom}_{\Lambda_0}(\underline{K}_0^\epsilon(S),\underline{K}_0(\C))\to \text{Hom}_{\Lambda_0^{(N)}}(P,\underline{K}_0(\C)) 
$$
of Lemma \ref{partial kappa} is defined\footnote{To apply Lemma \ref{partial kappa}, we take $A=C^*(\Gamma)$ and consider $S$ also as a finite subset of $A$.}.  Then there is a finite subset $T$ of $\Gamma$, $\delta>0$, and $\gamma\in (0,\epsilon)$ with the following property.

Assume $\psi:\Gamma\to M_n(\C)$ is a $(T,\delta)$-representation as in Definition \ref{quasi rep}.  Then exists a ucp $(S,\epsilon)$-representation $\phi:\Gamma\to M_n(\C)$ such that 
\begin{equation}\label{td con}
\|\phi(s)-\psi(s)\|<\gamma \quad \text{for all} \quad s\in S,
\end{equation}
and moreover so that for any two ucp $(S,\epsilon)$-representations $\phi,\phi':\Gamma\to M_n(\C)$ satisfying the estimate in line \eqref{td con}, the elements 
$$
\phi_*,\phi'_*\in \text{Hom}_{\Lambda_0}(\underline{K}_0^\epsilon(S), \underline{K}_0(\C))
$$
defined using Proposition \ref{ucp extend} part \eqref{phi ext} and Lemma \ref{x eps lambda module} satisfy that 
$$\kappa^{(S,\epsilon)}_P(\phi_*)=\kappa^{(S,\epsilon)}_P(\phi'_*)$$ 
as elements of $\text{Hom}_{\Lambda_0^{(N)}}(P,\underline{K}_0(\C))$.
\end{lemma}

\begin{proof}
Let $\gamma$ be as in Lemma \ref{partial kappa} part \eqref{pk2}.  Propositon \ref{llp approx} gives a finite subset $T$ of $\Gamma$ and $\delta>0$ such that if $\psi:\Gamma\to \mathcal{B}(H)_1$ is a $(T,\delta)$-representation, then there exists a ucp $(S,\epsilon)$-representation $\phi:C^*(\Gamma)\to M_n(\C)$ such that $\|\psi(s)-\phi(s)\|<\gamma$ for all $s\in S$.  This gives the properties in the statement. 
\end{proof}

\begin{definition}\label{rep phi*}
With notation as in Lemma \ref{to ucp}, for a fixed $(S,\epsilon)$ let $(T,\delta)$ be as in the conclusion, and let $\psi:\Gamma\to M_n(\C)$ be a $(T,\delta)$-representation.  Then we write $\psi_*\in \text{Hom}_{\Lambda_0^{(N)}}(P,\underline{K}_0(\C))$ for the map $\kappa^{(S,\epsilon)}_P(\phi_*)$ for any $\phi_*$ as in the conclusion of Lemma \ref{to ucp}. 
\end{definition}

Here is our main theorem.  It is essentially just Theorem \ref{stab main} specialized to the case where $A$ is a group $C^*$-algebra, and $\mathcal{X}$ consists of unitary group elements, although there is a little more to say in the LLP case.

\begin{theorem}\label{gp main}
Let $\Gamma$ be a countable discrete group such that $C^*(\Gamma)$ satisfies the approximate $K$-homology UCT of Definition \ref{kh auct}.  Let $\mathcal{R}$ be a class of good representations of $\Gamma$ as in Definition \ref{good rep}\footnote{The existence of such a class $\mathcal{R}$ forces $C^*(\Gamma)$ to be RFD as in Example \ref{good ex}.  We will only apply the theorem to the class $\mathcal{R}_q$ from Example \ref{good ex}, but state it in extra generality in case this is useful elsewhere.}.  Let $\mathcal{S}$ be a collection of finite symmetric subsets of $\Gamma$ that is closed under finite unions, and with the property that $\bigcup_{S\in \mathcal{S}}S$ generates $\Gamma$\footnote{The most interesting case is when $\Gamma$ is finitely generated and $\mathcal{S}=\{S\}$ for some finite symmetric generating set $S$; the reader should bear this case in mind.}.  Then for any $S\in \mathcal{S}$ and $\epsilon>0$ there exist a $\underline{K}_0$-triple $(P,T,\delta)$ as in Definition \ref{k0 trip} with $T\in \mathcal{S}$ and with the following property.

Let $\phi,\psi:\Gamma\to M_n(\C)$ be ucp $(T,\delta)$-representations in the sense of Definition \ref{quasi rep} with the property that (with notation as in as in Lemmas \ref{x eps lambda module} and \ref{partial kappa})
$$
\kappa^{(S,\epsilon)}_P(\phi_*) = \kappa^{(S,\epsilon)}_P(\psi_*),
$$
in $\text{Hom}_{\Lambda_0^{(N)}}(P,\underline{K}_0(\C))$.

Then there is a unitary representation $\theta:\Gamma\to M_k(\C)$ in $\mathcal{R}$ and a unitary $u\in M_{n+k}(\C)$ such that 
$$
\|u(\phi(s)\oplus \theta(s))u^*-(\psi(s)\oplus \theta(s))\|<\epsilon
$$
for all $s\in S$.

If moreover $\Gamma$ has the weak matricial LLP of Definition \ref{llp} and if $\mathcal{S}$ is the collection of all finite subsets of $\Gamma$, then one can drop the assumption that $\phi.\psi$ are ucp and use Definition \ref{rep phi*} for the definition of $\phi_*$ and $\psi_*$.
\end{theorem}

\begin{proof}
Note that Proposition \ref{ucp extend} part \eqref{phi ext} implies that a ucp representation $\phi:\Gamma\to M_n(\C)$ extends uniquely by linearity and continuity to a ucp map $\phi:C^*(\Gamma)\to M_n(\C)$, and that if $\phi:\Gamma\to M_n(\C)$ is a ucp $(T,\delta)$-representation, then the extended map $\phi:C^*(\Gamma)\to M_n(\C)$ is a $(T,\delta)$-$*$-homomorphism as in Definition \ref{x eps hom} (note that ucp maps are automatically $*$-preserving).  Having made this observation, the first two paragraphs of Theorem \ref{gp main} are just Theorem \ref{stab main} specialized to the case that $A=C^*(\Gamma)$ and the set $\mathcal{X}$ from Theorem \ref{stab main} consists of the sets of unitaries coming from $\mathcal{S}$.

The part with the additional assumption that $C^*(\Gamma)$ satisfies the LLP follows from the first part and Lemma \ref{to ucp}.
\end{proof}

\begin{remark}
The condition in the theorem is in some sense also necessary for the conclusion to hold.  To be more precise, the following holds.  

``Let $P$ be a $K$-datum as in Definition \ref{fin set tot hom}.  Then there exists a finite subset $S$ of $\Gamma$ and $\epsilon>0$ with the following property.  Let $T$ be a finite subset of $\Gamma$ and $\delta>0$ be such that the various components of $P$ are contained in the image of the maps $K_0^\delta(T;\cdot)\to K_0(A;\cdot )$ of line \eqref{dir lim maps} above; such exist by Lemma \ref{dir lim k thy}.   Then if $\phi:\Gamma\to M_n(\C)$ is a ucp $(T,\delta)$-representation and $\pi:\Gamma\to M_n(\C)$ is a unitary representation such that $\phi_*$ and $\pi_*$ do not agree on $P$, then for any representation $\theta:\Gamma\to M_k(\C)$ and any unitary $u\in M_{n+k}(\C)$,
$$
\|u(\phi(s)\oplus \theta(s))u^*-(\pi(s)\oplus \theta(s))\|\geq \epsilon
$$
for all $s\in S$.''

This follows from basic continuity properties of $K$-theory and the fact that $K$-theory allows cancellations of block sums.  This observation underlies many of the known obstructions to representation stability in operator norm: see for example \cite{Dadarlat:384aa} for far-reaching results in this direction.
\end{remark}

In the next three remarks we say a little bit more about the known range of validity of the assumptions in Theorem \ref{gp main}.

\begin{remark}\label{llp rem}
We first look at the weak matricial LLP assumption from Theorem \ref{gp main}.  

We do not know of any examples of a group that satisfies this assumption but do not satisfy the stronger LLP; conversely, most of the known examples of groups without the LLP come from work of Ioana-Spaas-Wiersma \cite{Ioana:2020aa}, and their techniques also show that these groups fail the weak matricial LLP.  We thus focus on the LLP, as this is the much more prominent property.

The LLP was codified\footnote{There are antecedents: see for example \cite[Theorem 3.2]{Effros:1985aa}.}  for general $C^*$-algebras by Kirchberg \cite[Section 2]{Kirchberg:1993aa}, and a wealth of information about the property and its connection to several deep conjectures can be found in the recent textbook \cite{Pisier:2020aa}.  The recent paper \cite{Enders:2024aa} also contains a good summary of known facts from a $C^*$-algebraic point of view (as well as its new results).

Specializing to groups, the most important class of examples with the LLP are amenable groups: for countable amenable groups this follows from the Choi-Effros lifting theorem \cite{Choi:1976aa} (see \cite[Section 3]{Arveson:1977aa} or \cite[Theorem C.3]{Brown:2008qy} for simpler proofs) and nuclearity of the group $C^*$-algebra \cite[Theorem 4.2]{Lance:1973aa} (see for example \cite[Theorem 2.6.8]{Brown:2008qy} for a textbook treatment).  This implies the LLP even for uncountable amenable groups, as any finite-dimensional subspace of $C^*(\Gamma)$ is contained in the $C^*$-subalgebra $C^*(\Gamma_0)$ for some countable subgroup.

The fundamental class of non-amenable groups with the LLP are free groups: this was established by Kirchberg in \cite[Lemma 3.3]{Kirchberg:1994aa}.  The class of groups with the LLP is closed under free products with finite amalgam \cite[Comments below Proposition 3.21]{Ozawa:2004ab}, and direct products with amenable groups: this follows from \cite[Corollary 2.6 (iv)]{Kirchberg:1993aa} and semidirect products with amenable groups (see \cite[Theorem 7.2 and Proposition 5.10]{Buss:2022aa} or \cite[Corollary 8.5]{Enders:2024aa}); this implies in particular that free-by-cyclic groups (compare Theorem \ref{fbc intro}) have the LLP.

The state of the art on examples of groups with the LLP at time of writing is surveyed in \cite{Fournier-Facio:2026aa}.  This also shows some new results: in particular that fundamental groups of manifolds of dimension at most three (compare Theorems \ref{intro surf the} and \ref{intro 3man the} from the introduction), and the class of one-relator groups from Theorem \ref{o r intro} all have the LLP.  It is open whether all one-relator groups have the LLP, but it is known for a very large class: as well as those from Theorem \ref{o r intro}, the LLP is known for one-relator groups with torsion, and for Baumslag-Solitar groups, for example.

On the other hand, many groups are known not to have the LLP.  The first existence proofs for groups without the LP are due to Ozawa \cite[Corollary 5]{Ozawa:2004aa}, and Thom gave the first explicit examples in \cite[page 198]{Thom:2010aa}.  Further examples without the LLP were given by Buss, Echterhoff, and the author \cite[Corollary 4.8]{Buss:2018nm}; these are again not explicit.  The most important class of examples is due to Ioana, Spaas, and Wiersma \cite{Ioana:2020aa}; the latter examples include natural and well-studied groups like $\Z^2\rtimes SL(2,\Z)$ and $SL(3,\Z)$.  It seems to be open whether there is any property (T) group with the LLP. 
\end{remark}

\begin{remark}\label{uct rem}
We now say some more about the approximate $K$-homology UCT assumption in Theorem \ref{gp main}.  

The UCT for $K$-homology (which is all we really use in this paper) was introduced by Brown in \cite{Brown:1984rx}.  The stronger UCT for $KK$-theory was introduced later by Rosenberg and Schochet \cite{Rosenberg:1987bh}; following standard practice in the literature, when we say ``UCT'' we mean the Rosenberg-Schochet $KK$-theory version.  We will focus here on the UCT: while this is stronger than the assumptions we need, it is much better understood; in particular, we do not know any examples of group $C^*$-algebras that satisfy the approximate $K$-homology UCT from Definition \ref{kh auct} without also satisfying the $KK$-theory UCT.

The most prominent class of groups known to satisfy the UCT are the (countable) a-T-menable groups\footnote{Also called groups with the \emph{Haagerup property} due to the influence of \cite{Haagerup:1979rq}.}: this is a consequence of the work of Higson and Kasparov \cite{Higson:2001eb} on the Baum-Connes conjecture for such groups, as shown by Tu \cite[Proposition 10.7]{Tu:1999bq} (in a more general context).  The class of a-T-menable groups includes all the groups covered by Theorems \ref{intro surf the}, \ref{fbc intro}, \ref{o r intro} and \ref{intro 3man the} from the introduction: this is explained in \cite[Remark 4.4]{Fournier-Facio:2026aa}, and is a consequence of several deep results in geometry and geometric group theory.

Using known permanence results, the class of groups for which $C^*(\Gamma)$ satisfies the UCT can be pushed a bit further.  One important source of examples is Tu's \emph{condition (BC$\,'$)} \cite[Definition 2.5]{Tu:1999aa}: this is recorded in \cite[Theorem 3.5]{Matthey:2008aa}\footnote{There seems to be a slight gap in the statement given there: as well as (BC$'$), one also needs the $C^*$-algebra $\mathcal{A}$ appearing in the definition of (BC$'$) to be locally induced from $C^*$-algebras satisfying the UCT.  This is satisfied in all the examples we give.}.  Examples of groups satisfying (BC$'$) include all one-relator groups \cite[Theorem 4.5]{Tu:1999aa}; we believe a-T-menability is open at this level of generality.  There are also interesting examples that are known not to be a-T-menable: for example $\Gamma=\Z^2\rtimes SL(2,\Z)$ then $C^*(\Gamma)$ satisfies the UCT\footnote{This does not appear to be in the literature, but essentially the same proof as for $C^*(SL(2,\Z))$ works, once one has observed that the crossed product of $C^*(\Z^2)$ by any finite group is type I.}, even though $\Gamma$ has relative property (T) with respect to the infinite subgroup $\Z^2$ and therefore cannot be a-T-menable.

As far as we are aware, there are no examples of groups for which $C^*(\Gamma)$ is known to not satisfy the UCT.  However, it seems very likely that the UCT fails for $C^*(\Gamma)$ for some (and maybe all) infinite property (T) groups: the computations Skandalis carries out in \cite[Section 4]{Skandalis:1991aa} to show that the \emph{reduced} group $C^*$-algebras of certain hyperbolic property (T) groups do not satisfy the UCT strongly suggest the same happens for the corresponding \emph{maximal} group $C^*$-algebras (see also \cite{Skandalis:1991aa} and \cite[Sections 6.2 and 6.3]{Higson:2004la}), but we were not able to show this.

We do not know of any (separable) $C^*$-algebras that fail the approximate $K$-homology UCT of Definition \ref{kh auct}, which is all we actually need in this paper. 
\end{remark}

\begin{remark}\label{rfd rem}
We now say some more about the assumption on a good class of representations in Theorem \ref{gp main}.

First, note that if a good class of representations exists at all, then $C^*(\Gamma)$ is RFD as in Example \ref{good ex}.  Recall that a group is \emph{maximally almost periodic} (MAP) if it has a separating family of finite-dimensional unitary representations: for finitely generated groups, this is equivalent to residual finiteness by Mal'cev's theorem (see for example \cite[Theorem 6.4.12]{Brown:2008qy}), but not in general as shown by $\Gamma=\Q$.  If $\Gamma$ is RFD, it is clearly MAP.  For amenable groups the converse is true \cite[Proposition 1]{Bekka:1999kx}, but not in general: for example Bekka shows in  \cite{Bekka:1999kx} that many arithmetic groups such as $SL(n,\Z)$ with $n> 2$ are not RFD; they are, however, residually finite (whence MAP) by Mal'cev's theorem again.

In general, the RFD property is known to be preserved under direct products with MAP amenable groups (this is straightforward), free products \cite{Exel:1992aa}, certain free products amalgamated over finite subgroups \cite[Theorem 2]{Li:2012aa}, and certain free products amalgamated over central subgroups and HNN extensions \cite{Shulman:2022aa}.   A large class of one-relator groups are shown to be RFD in \cite[Theorem 3.11]{Hadwin:2018aa}.  

On the other hand, for $F_2$ the free group on two generators, whether $F_2\times F_2$ is RFD is equivalent to Kirchberg's QWEP conjecture: this equivalence is implicit in \cite[Proof of Proposition 8.1]{Kirchberg:1993aa}, and made explicit in \cite[Proposition 3.19]{Ozawa:2004ab}.  Thanks to the recent negative solution of the QWEP conjecture (see \cite{Salle:2023aa} for a survey), $F_2\times F_2$ is therefore not RFD.  The difficulty of the problem for $F_2\times F_2$ suggests the difficulty of determining the RFD property in general.

A stronger property than being RFD that is both interesting in its own right and a good source of examples is \emph{property FD} of Lubotzky and Shalom \cite{Lubotzky:2004xw}: recall from Example \ref{good ex} that $\Gamma$ has FD if the family of all finite-dimensional unitary representations that factor through finite quotients is good.  Property FD was shown to hold for fundamental groups of surfaces \cite[Theorem 2.8 (2)]{Lubotzky:2004xw} and free by cyclic groups (i.e.\ groups of the form $F\rtimes \Z$ for $F$ a finitely generated free group) \cite[Theorem 2.8 (1)]{Lubotzky:2004xw} in the original paper.  Thanks to known structural results on one-relator groups, this also implies FD for the class of groups in Theorem \ref{o r intro}: compare \cite[Section 6.2]{Fournier-Facio:2026aa}.  Property FD for fundamental groups of compact three-manifolds was established in \cite[Theorem 6.1]{Fournier-Facio:2026aa}; this requires many deep results from geometry and geometric group theory.
\end{remark}

Finally, let us give some remarks about other techniques that seem relevant and that one might guess could be used to strengthen Theorem \ref{gp main}.  We would be very interested to see progress along these lines.

\begin{remark}\label{kub rem}
Inspired by work of Kubota \cite[Section 5.2]{Kubota:2019aa}, it is tempting to work not with $C^*(\Gamma)$, but with the $C^*$-algebra $C^*_{fd}(\Gamma)$ one gets by completing $\C[\Gamma]$ in the direct sum of all finite-dimensional representations.  This has the advantage that $C^*_{fd}(\Gamma)$ is automatically RFD; one could also do something similar with representations that factor through finite quotients.  However, there are two important disadvantages.  

First, even if $\Gamma$ is a-T-menable, it is not clear that the UCT holds for $C^*_{fd}(\Gamma)$.  Indeed, this would hold for any completion  of the group algebra that comes from a so-called \emph{correspondence functor} by \cite[Theorem 6.6]{Buss:2014aa}.  However, any completion of $\C[\Gamma]$ coming from a correspondence functor has the property that its dual space is an ideal in the dual $C^*(\Gamma)^*$ of $C^*(\Gamma)$ \cite[Corollary 5.7]{Buss:2014aa}.  The trivial representation is the unit of $C^*(\Gamma)^*$; as the trivial representation is an element of $C^*_{fd}(\Gamma)^*$, we see that $C^*_{fd}(\Gamma)$ comes from a correspondence functor if and only if it equals $C^*(\Gamma)$.  

The second problem with $C^*_{fd}(\Gamma)$ is that it is not clear which ucp quasi-representations extend to it, unlike for $C^*(\Gamma)$ where this is automatic by Proposition \ref{ucp extend} part \eqref{phi ext} above.
\end{remark}

\begin{remark}\label{bost rem}
Following the work of Dadarlat \cite{Dadarlat:2011kx} it is also very tempting to work not with $C^*(\Gamma)$ but with the group convolution algebra $\ell^1(\Gamma)$.  This would have three big advantages.  First, it is always RFD when $\Gamma$ is MAP.  Second, $(X,\epsilon)$-$*$-homomorphisms always extend to it, so one does not have to worry about ucp maps and analogues of Proposition \ref{ucp extend}; thus the LLP is likely not at all relevant.  Third, one might be able to work directly with the Bost assembly map \cite[Section 1.2]{Aparicio:2020aa} in place of the maximal Baum-Connes assembly map for explicit computations (see Section \ref{bc sec} below).  The Bost assembly map is known to be an isomorphism in quite a lot more generality than the maximal Baum-Connes assembly map thanks to deep work of Lafforgue \cite{Lafforgue:2002zl}.

The ingredient that is missing in this case is an analogue of Theorem \ref{kl ck hom}.  Indeed, it is crucial for our methods that homomorphisms between $K$-theory groups can be realized by elements of bivariant $KK$-theory, and that these elements of $KK$-theory can be modeled by almost multiplicative maps.  Lafforgue's Banach algebra $KK$-theory \cite{Lafforgue:2002qm} might help here, but this does not seem obvious.
\end{remark}

\section{Index theory and the Baum-Connes conjecture}\label{bc sec}

In order to be able to apply Theorem \ref{gp main}, we need to be able to say more about the $K$-theory of $C^*(\Gamma)$.  The key tools here are the assembly maps of Baum-Connes \cite[Section 3]{Baum:1994pr} and Kasparov \cite[Section 6]{Kasparov:1988dw}: these relate the $K$-homology of the classifying space $B\Gamma$ to the $K$-theory of the group $C^*$-algebra $C^*(\Gamma)$.  In this section, we use the Baum-Connes assembly map to compute the map
$$
\pi_*\in \text{Hom}_{\Lambda_0}(\underline{K}_0(C^*(\Gamma)),\underline{K}_0(\C))
$$
of Example \ref{basic lambda 0} induced by a finite-dimensional unitary representation $\pi:\Gamma\to M_n(\C)$, under appropriate assumptions.   These computations will be used to apply our abstract main result - Theorem \ref{gp main} above - to concrete examples.  

To summarize the rest of this section, Subsection \ref{basic bc subsec} gives properties of the map $\pi_*:K_0(C^*(\Gamma))\to K_0(\C)$ induced by a unitary representation $\pi:\Gamma\to M_n(\C)$ on $K$-theory with integer coefficients; this is all folklore and we are just summarizing what we need for the reader's convenience.   Subsection \ref{fc subsec} computes the maps $\pi_*:K_0(C^*(\Gamma);\Z/n)\to K_0(\C;\Z/n)$ induced by the representation $\pi$ on $K$-theory with finite coefficients in terms of the \emph{relative eta invariants} of Atyiah, Patodi, and Singer \cite{Atiyah:1975aa}, at least in some special cases.  As far as we know, this material is new, but we suspect at least some of it is known to experts.

\subsection{The Baum-Connes-Kasparov assembly map}\label{basic bc subsec}

In this subsection we recall background on the Baum-Connes-Kasparov assembly map, use it to compute the map 
$$
\pi_*\in \text{Hom}(K_0(C^*(\Gamma)),K_0(\C))
$$
arising from a finite-dimensional unitary representation, and collect some related facts. The computations in this subsection are essentially part of the folklore of the subject, as will be clear from the references below.

Let $X$ be a CW complex.  We define the \emph{representable $K$-homology} of $X$ by 
\begin{equation}\label{rep k hom}
RK_*(X):=\lim_{\to} K_*(Y),
\end{equation}
where the direct limit is taken over all finite subcomplexes $Y\subseteq X$ and $K_*(Y)=KK_*(C(Y),\C)$ is analytic $K$-homology.  We will mainly apply this to the case when $X$ is the classifying space $B\Gamma$ of a group.  

Assume now that $\Gamma$ is a countable discrete group equipped with a CW complex model for the  classifying space $B\Gamma$.    The \emph{Baum-Connes-Kasparov (BCK) assembly map} (see \cite[Sections 3 and 7]{Baum:1994pr} and \cite[Definition 6.2]{Kasparov:1988dw}; the earliest reference is \cite[Section 8]{Kasparov:1975ht}) is a graded group homomorphism 
\begin{equation}\label{max bc}
\mu:RK_*(B\Gamma)\to K_*(C^*(\Gamma)).
\end{equation}
To be explicit: we define the map $\mu$ above to be the same as the map $\beta$ of \cite[Definition 6.2]{Kasparov:1988dw}, i.e.\ (up to a small technicality with direct limits) it is given by Kasparov product with the explicit element\footnote{Sometimes called the \emph{Miscenko bundle}.}
\begin{equation}\label{mis bun}
[\beta_\Gamma]\in RKK(B\Gamma;\C,C^*(\Gamma))
\end{equation}
defined there.  

In the next section, we will also need the BCK assembly map with finite coefficients, and record this now.  For $n\geq 2$, we define \emph{representable $K$-homology with finite coefficients} for a CW complex $X$ to be 
$$
RK_*(X;\Z/n):=\lim_{\to} KK_*(C(Y),O_{n+1})
$$
where the limit is taken over finite subcomplexes $Y$ of $X$, $O_{n+1}$ is as usual the Cuntz algebra, and $KK$ is Kasparov's bivariant $KK$-theory \cite{Kasparov:1988dw}.  The \emph{BCK assembly map with coefficients in $\Z/n$}
\begin{equation}\label{bc fc}
\mu:RK_*(B\Gamma;\Z/n)\to K_*(C^*(\Gamma);\Z/n)
\end{equation}
is again defined by Kasparov product with the element in line \eqref{mis bun} above.

\begin{lemma}\label{bck fc lem}
Assume that the BCK assembly map of line \eqref{max bc} is an isomorphism.  Then for each $n\geq 2$ the BCK assembly map with finite coefficients of line \eqref{bc fc} is also an isomorphism.
\end{lemma}

\begin{proof}
We focus on the even case of the assembly map $RK_0(B\Gamma;\Z/n)\to K_0(C^*(\Gamma);\Z/n)$; the odd case is essentially the same.  For an abelian group $G$, let $^nG$ and $_nG$ denote respectively the cokernel and kernel of the multiplication by $n$ map $\times n:G\to G$; equivalently, $^nG=G\otimes \Z/n$ and $_nG$ is the $n$-torsion subgroup of $G$.  The universal coefficient exact sequence\footnote{\label{uct fn}Potentially confusingly, there are (at least) two distinct short exact sequences that go by the name ``universal coefficient exact sequence'' in $K$-theory; the one we are using here should not be confused with the UCT assumption from \cite{Rosenberg:1987bh} that we have used at several other points in this paper.} in $K$-theory and $K$-homology (see \cite[Proposition 1.8]{Schochet:1984ab} for $K$-theory, and the comments in \cite[Section 5]{Schochet:1984ab} for $K$-homology) give a diagram
$$
\xymatrix{ 0\ar[r] & ^nRK_0(B\Gamma) \ar[r] \ar[d] & RK_0(B\Gamma;\Z/n) \ar[r] \ar[d] & _nRK_1(B\Gamma) \ar[d] \ar[r] & 0 \\
0\ar[r] & ^nK_0(C^*(\Gamma)) \ar[r]  & K_0(C^*(\Gamma);\Z/n) \ar[r] & _nK_1(C^*(\Gamma))  \ar[r] & 0 }
$$
where the bottom and top lines are the exact universal coefficient sequences (more precisely, for the top line we take the exact universal coefficient sequence for each finite subcomplex $Y$, then take a direct limit), and the vertical maps are those induced by the BCK assembly map.  All the maps are induced by Kasparov product with appropriate elements of $KK$-theory (see for example \cite[Section 1.2]{Dadarlat:1996aa} for the horizontal maps and \cite[Definition 6.2]{Kasparov:1988dw} for the BCK assembly maps) whence the diagram commutes by associativity of the Kasparov product.  The result follows from the five lemma.  
\end{proof}

For our main theorems, we will need to assume that the BCK assembly map is an isomorphism.  Let us say a little more about this assumption, via its relation with the closely related \emph{Baum-Connes assembly map}.

\begin{remark}\label{bc rem}
Recall that the \emph{reduced group $C^*$-algebra} $C^*_r(\Gamma)$ is the completion of the complex group algebra in its image under the left regular representation on $\ell^2(\Gamma)$.
The \emph{Baum-Connes assembly map} \cite[Line (3.15)]{Baum:1994pr} for a discrete group $\Gamma$ is a group homomorphism 
\begin{equation}\label{bcc ass}
\mu:K_*^\Gamma(\underline{E}\Gamma)\to K_*(C^*_r(\Gamma))
\end{equation}
from the equivariant topological $K$-homology of the classifiying space $\underline{E}\Gamma$ for proper actions to the $K$-theory of the reduced group $C^*$-algebra.  If $\Gamma$ is torsion-free $\underline{E}\Gamma$ equals the universal cover $E\Gamma$ of the classifying space $B\Gamma$, and the assembly map in line \eqref{bcc ass} reduces to a map
\begin{equation}\label{rk bg}
\mu:RK_*(B\Gamma)\to K_*(C^*_r(\Gamma)).
\end{equation}
The assembly map in line \eqref{rk bg} is consistent with the BCK assembly map of line \eqref{max bc} in the sense that it is the composition of the BCK assembly map, and the map on $K$-theory induced by the canonical quotient map 
\begin{equation}\label{max to red}
C^*(\Gamma)\to C^*_r(\Gamma).
\end{equation}

We will mainly be interested in cases where the BCK assembly map of line \eqref{max bc} is an isomorphism\footnote{Or at least surjective, but this does not seem to be known in any significantly greater generality than isomorphism.}.  This is obstructed in two important ways connected to the definition of the Baum-Connes assembly map in line \eqref{bcc ass}.

First, the BCK assembly map is never surjective if $\Gamma$ has non-trivial torsion elements.  Indeed, let 
$$
\text{tr}:\C[\Gamma]\to \C,\quad \sum_{g\in \Gamma} a_g g\mapsto a_e
$$ 
be the canonical trace on the complex group algebra.  This extends to a trace on $C^*(\Gamma)$, and this trace induces a map $\text{tr}_*:K_0(C^*(\Gamma))\to \R$.  Surjectivity of the maximal BCK assembly map implies that $\text{tr}_*$ is integer-valued: this is essentially Atiyah's covering index theorem \cite{Atiyah:1976th} (see for example \cite[Section 10.1]{Willett:2010ay} for a textbook treatment of this integrality result, and \cite{Luck:2002lk} for a further-reaching analysis).  However, if $g\in \Gamma$ is a non-trivial element of order $n$ for some $n\geq 2$, then $p:=\frac{1}{n}\sum_{k=0}^{n-1} g^k$ defines a projection in the group ring such that $\text{tr}_*[p]=1/n$.  This is the basic reason for preferring the equivariant $K$-homology of $\underline{E}\Gamma$ (see \cite[Sections 1 and 2]{Baum:1994pr}) to the $K$-homology of $B\Gamma$ in the definition of the Baum-Connes assembly map as in line \eqref{bcc ass}.

Second, property (T) obstructs surjectivity of the BCK assembly map due to the existence of \emph{Kazhdan projections} $p\in C^*(\Gamma)$ whose $K$-theory classes cannot be in the image of the assembly map; the essential point is again Atiyah's covering index theorem as explained in \cite[Section 5]{Higson:1998qd}.  This is one of the motivations for using the reduced group $C^*$-algebra $C^*_r(G)$ on the right hand side of the Baum-Connes conjecture as in line \eqref{bcc ass}.  

Nonetheless, the BCK assembly map in line \eqref{max bc} is known to be an isomorphism in many interesting cases.  In all the cases we know this to hold, it happens for the conjunction of three reasons: the Baum-Connes assembly map as in line \eqref{bcc ass} is known to be an isomorphism; the group $\Gamma$ is torsion-free, whence $K_*^\Gamma(\underline{E}\Gamma)$ agrees with $RK_*(B\Gamma)$; and the group $\Gamma$ is \emph{$K$-amenable} in the sense of Cuntz \cite[Theorem 2.1 and Definition 2.2]{Cuntz:1983jx}, which implies that the canonical map in line \eqref{max to red} induces an isomorphism on $K$-theory.  

The most prominent case where these three conditions hold is that of torsion-free a-T-menable groups \cite[Theorems 1.1 and 1.2]{Higson:2001eb}, and in particular for torsion-free amenable groups.  Another important case is torsion-free one-relator groups as in \cite[Theorems 3 and 4]{Beguin:2999aa}.   A useful general criterion for the conditions above to hold is when $\Gamma$ is torsion-free and satisfies Tu's \emph{condition (BC$\,'$)} \cite[Definition 2.5 and Lemma 2.6]{Tu:1999aa}.  
\end{remark}

Going back to generalities, we need a slightly ad-hoc\footnote{We say ``ad-hoc'' as it is perhaps more natural to use \emph{representable $K$-theory} as in \cite[Definition 2.19]{Kasparov:1988dw} or \cite[Section 5]{Segal:1970aa}.  The inverse limit version described here is more convenient for us; the two are related by a Steenrod-Milnor sequence (see \cite[Proposition 4.1]{Atiyah:1969aa}).} definition, based on \cite[Section 4]{Atiyah:1961uq}\footnote{More precisely, in \cite[Section 4]{Atiyah:1961uq}, Atiyah defines $\mathcal{K}^*(X)$ to be the inverse limit ${\displaystyle \lim_{\leftarrow} K^*(X^{(n)})}$, where $X^{(n)}$ is the $n$-skeleton of $X$; however, Atiyah is working with CW complexes with finitely many cells in each dimension; the definition we give here seems the correct generalization.}.

\begin{definition}\label{inv lim k}
Let $X$ be a CW complex.  We define the \emph{limit $K$-theory} of $X$ to be the inverse limit $\mathcal{K}^*(X):={\displaystyle \lim_{\leftarrow} K^*(Y)}$ over all finite subcomplexes $Y$ of $X$
\end{definition}

Note that with this definition the usual pairings\footnote{See for example \cite[Section 7.2]{Higson:2000bs}, or treat it as a special case of the Kasparov product \cite[Theorem 2.14]{Kasparov:1988dw}.} $K^*(Y)\otimes K_*(Y)\to \Z$ between $K$-theory and $K$-homology of a compact space $Y$ induce a canonical pairing 
\begin{equation}\label{pair}
\mathcal{K}^*(X)\otimes RK_*(X)\to \Z
\end{equation}
defined for any CW complex $X$.  Analogously, the Kasparov product 
\begin{equation}\label{kas 1 pair}
KK_i(\C,C(Y))\otimes KK_i(C(Y),O_{n+1})\to KK(\C,O_{n+1})=K_0(O_{n+1})=\Z/n
\end{equation}
for $i\in \{0,1\}$ gives rise to a pairing 
\begin{equation}\label{pair n}
\mathcal{K}^i(X)\otimes RK_i(X;\Z/n)\to \Z/n.
\end{equation}

For the next result, note that a vector bundle $E$ over a CW complex $X$ defines a class in $\mathcal{K}^0(X)$ in a canonical way: indeed, the restrictions of $E$ to all finite subcomplexes define an element of the inverse limit.  Recall moreover that if $\pi:\Gamma\to M_m(\C)$ is a unitary representation and $E\Gamma$ is the universal cover of the classifying space $B\Gamma$ then the associated \emph{flat bundle} is defined by 
\begin{equation}\label{fb}
E_{\pi}:=(E\Gamma\times \C^m) / \Gamma
\end{equation}
where $\Gamma$ acts diagonally via deck transformations on $E\Gamma$, and via $\pi$ on $\C^m$.  Then the projection map $E_\pi\to E\Gamma/\Gamma=B\Gamma$ makes $E_\pi$ into a rank $m$ vector bundle over $B\Gamma$.

\begin{lemma}\label{flat pair}
Let $\Gamma$ be a discrete group, let $\pi:\Gamma\to M_m(\C)$ be a unitary representation, and let $E_\pi$ be the corresponding flat bundle over $B\Gamma$.  Let $[E_\pi]\in \mathcal{K}^0(B\Gamma)$ be the corresponding class, and let $p_\pi:RK_0(B\Gamma)\to \Z$ be the operation of pairing with $E_\pi$ as in line \eqref{pair} above.  Then the diagram 
$$
\xymatrix{ RK_0(B\Gamma) \ar[d]^-{p_\pi} \ar[r]^-\mu & K_0(C^*(\Gamma)) \ar[d]^-{\pi_*} \\
\Z \ar@{=}[r] & \Z}
$$
commutes.  Similarly, the for any $n\geq 2$, the diagram 
$$
\xymatrix{ RK_0(B\Gamma;\Z/n) \ar[d]^-{p_\pi} \ar[r]^-\mu & K_0(C^*(\Gamma);\Z/n) \ar[d]^-{\pi_*} \\
\Z/n \ar@{=}[r] & \Z/n}
$$
commutes.
\end{lemma}

\begin{proof}
This can be established in the same way as \cite[Lemma 4.2]{Ramras:2012fk} (in the special case that the auxiliary space $X$ appearing there is a point): we leave the details to the reader.
\end{proof}

For the next result, let $\tau:C^*(\Gamma)\to \C$ be the trivial representation.  As $\tau$ splits the unit inclusion $\C\to C^*(\Gamma)$, there is a canonical splitting
\begin{equation}\label{red k hom}
K_0(C^*(\Gamma))=\Z[1]\oplus \widetilde{K}_0(C^*(\Gamma))
\end{equation}
where $[1]\in K_0(C^*(\Gamma))$ is the class of the unit, and $\widetilde{K}_0(C^*(\Gamma))$ is the kernel of $\tau_*:K_0(C^*(\Gamma))\to \Z$.  Let also $\widetilde{RK}_0(B\Gamma)$ be the kernel of the map $RK_0(B\Gamma)\to \Z$ induced by collapsing $B\Gamma$ to a single point, and let $[pt]\in RK_0(B\Gamma)$ be the class induced by the inclusion of any point in $B\Gamma$, which also gives rise to a direct sum decomposition
\begin{equation}\label{red k hom 2}
RK_0(B\Gamma)=\Z[pt]\oplus \widetilde{RK}_0(B\Gamma)
\end{equation}

\begin{corollary}\label{red cor}
Let $\Gamma$ be a discrete group.  With notation as in the paragraph above, the BCK assembly map of line \eqref{max bc} splits as a direct sum 
$$
\widetilde{\mu}\oplus \text{id}_\Z:\widetilde{RK}_0(B\Gamma)\oplus \Z[pt] \to \widetilde{K}_0(C^*(\Gamma))\oplus \Z[1].
$$
\end{corollary}

\begin{proof}
Lemma \ref{flat pair} implies that the diagram below
$$
\xymatrix{ \text{Kernel}(p_\tau) \ar[d]  \ar[r]^-\mu & \widetilde{K_0}(C^*(\Gamma)) \ar[d] \\ RK_0(B\Gamma) \ar[d]^-{p_\tau} \ar[r]^-\mu & K_0(C^*(\Gamma)) \ar[d]^-{\tau_*} \\
\Z \ar@{=}[r] & \Z}
$$
commutes.  The map $p_\tau$ is the operation of pairing with the trivial line bundle: this agrees with the map $RK_0(B\Gamma)\to \Z$ defined by collapsing $B\Gamma$ to a point, so $\text{Kernel}(p_\tau)=\widetilde{RK}_0(B\Gamma)$.  The lower two vertical maps $p_\tau$ and $\tau_*$ are split by the inclusion of a point, and the inclusion of the unit respectively, and the corresponding diagram still commutes by naturality of the assembly map (the bottom horizontal arrow can be thought of as the BCK assembly map for the trivial group).  The result follows from these observations.
\end{proof}

The next lemma is again part of the folklore of the subject: see for example \cite[Theorem 2.4.3]{Keswani:2000aa} or \cite[Proposition 3.2]{Dadarlat:2011uq} for a proof.

\begin{lemma}\label{kes same}
Let $\Gamma$ be a discrete group, and let $\mu$ be the BCK assembly map of line \eqref{max bc}.  Let $\pi:\Gamma\to M_m(\C)$ be a finite-dimensional unitary representation.  Then with respect to the splitting in line \eqref{red k hom 2}, $\pi_*(\mu(\widetilde{RK}_0(B\Gamma)))=0$ and $\pi_*(\mu[pt])=m$. \qed
\end{lemma}

The next corollary is immediate from Corollary \ref{red cor} and Lemma \ref{kes same}.

\begin{corollary}\label{fd triv}
Let $\Gamma$ be a discrete group.  If the BCK assembly map $\mu:RK_*(B\Gamma)\to K_*(C^*(\Gamma))$ is surjective then for any unitary representation $\pi:\Gamma\to M_m(\C)$ we have
$$
\pi_*( \widetilde{K}_0(C^*(\Gamma)))=0 \quad \text{and}\quad \pi_*([1])=m.\eqno\qed
$$
\end{corollary}

\begin{remark}
The formulas in Corollary \eqref{fd triv} do not need the BCK assembly map for their statement, and one might hope that they admit a more elementary proof.  Indeed, the equality ``$\pi_*([1])=m$'' is true for any group $\Gamma$ and unitary representation $\pi$, and straightforward to prove.  However, the statement ``$\pi_*( \widetilde{K}_0(C^*(\Gamma)))=0$'' is not true without further assumptions as we now explain using the theory of \emph{Kazhdan projections}: see for example \cite[Section 3.7]{Higson:2000bs} for background.

Indeed, let $\Gamma$ be an infinite property (T) group that admits a non-trivial finite-dimensional irreducible representation $\pi:\Gamma\to M_n(\C)$ (for example, $\Gamma=SL(3,\Z)$).  Then there is a \emph{Kazhdan projection} $p\in C^*(\Gamma)$ such that $\pi(p)=1_{M_n(\C)}$ but $\tau(p)=0$ for $\tau$ the trivial representation.  Hence in particular, $[p]\in  \widetilde{K_0}(C^*(\Gamma))$ and $\pi_*([p])\neq 0$, so the first equation in line \eqref{fd triv} fails.  
\end{remark}

We conclude this subsection with a general result: it should be thought of as characterizing when an element of $\text{Hom}_{\Lambda_0}(\underline{K}_0(C^*(\Gamma)),\underline{K}_0(\C))$ can be approximated by the class of an honest representation, up to taking an integer multiple.  In the next subsection we will be able to give results that do not need one to take integer multiples, but we can only show these under more stringent hypotheses.

We need one more technical lemma.   For an abelian group $G$, let $^nG$ and $_nG$ denote respectively the cokernel and kernel of the multiplication by $n$ map $\times n:G\to G$ (equivalently, $^nG=G\otimes \Z/n$ and $_nG$ is the $n$-torsion subgroup of $G$).   Recall also (see for example \cite[Proposition 1.6]{Schochet:1984ab} or \cite[line (1.6)]{Dadarlat:1996aa}) that that for a $C^*$-algebra $A$ and $n\geq 2$ there is a \emph{Bockstein six-term exact sequence}
\begin{equation}\label{bs}
\xymatrix{ K_0(A) \ar[r]^-{\times n} & K_0(A) \ar[r]^-{\rho_n} & K_0(A;\Z/n) \ar[d]^{\beta_n} \\
K_1(A;\Z/n) \ar[u]^-{\beta_n} & K_1(A) \ar[l]^-{\rho_n} & K_1(A) \ar[l]^-{\times n} }.
\end{equation}
The maps are all induced by Kasparov products with particular elements of $KK$-groups; we will not need the precise descriptions, but note that the maps labeled $\rho_n$ are induced by morphisms from the category $\Lambda_0$ of Definition \ref{lambda0}.

\begin{lemma}\label{zn split}
Let $A$ be a $C^*$-algebra.  Then for each $n\geq 2$, $^nK_0(A)$ canonically identifies with the image of the group $K_0(A)$ inside $K_0(A;\Z/n)$ under the map $\rho_n$ of line \eqref{bs}.  Moreover, one can choose splittings  
$$
K_0(A;\Z/n)\cong \, ^nK_0(A)\oplus \, _nK_1(A),
$$
that respect this subgroup, and that are compatible with the $\Lambda_0$-module structure on $\underline{K}_0(A)$. 
 
If moreover $A=C^*(\Gamma)$ for a discrete group $\Gamma$ and $\tau:C^*(\Gamma)\to \C$ is the trivial representation, then we may choose the splitting so that in addition $_nK_1(A)$ is contained in the kernel of the induced map $\tau_*:K_0(C^*(\Gamma);\Z/n)\to K_0(\C;\Z/n)$.
\end{lemma}

\begin{proof}
Splitting up the six-term exact sequence of line \eqref{bs} along the maps labeled ``$\times n$'' gives the short exact universal coefficient theorem sequence
$$
0\to\,^n\!K_0(A)\xrightarrow{\rho_n} K_0(A;\Z/n)\xrightarrow{\beta_n} \,_n K_1(A)\to 0
$$  
(compare the proof of Lemma \ref{bck fc lem}).  As in \cite[Lemma A.5]{Carrion:2020aa}, this sequence admits a splitting 
$$
s_0:\,_n K_1(A)\to K_0(A;\Z/n)
$$
that is compatible with the operations making up the category $\Lambda_0$\footnote{It is not possible to choose such a splitting that is natural for $*$-homomorphisms between $C^*$-algebras.} from Definition \ref{lambda0}.  

Specialize now to $A=C^*(\Gamma)$, and consider the commutative diagram
\begin{equation}\label{part uct}
\xymatrix{ 0\ar[r] & ^n K_0(C^*(\Gamma)) \ar[d]^{\tau_*} \ar[r]^-{\rho_n}  & K_0(C^*(\Gamma);\Z/n) \ar[d]^{\tau_*} \ar[r]^-{\beta_n} & _nK_1(C^*(\Gamma)) \ar[r] & 0 \\ & \Z/n \ar@{=}[r] &  \Z/n & & }.
\end{equation}
The vertical maps are split by the unit inclusion, so we have a direct sum decomposition
$$
K_0(C^*(\Gamma);\Z/n)\cong \text{Ker}(\tau_*)\oplus (\Z/n)[1]
$$
and in particular, there is an idempotent map $p:K_0(C^*(\Gamma);\Z/n)\to K_0(C^*(\Gamma);\Z/n)$ with image $\text{Ker}(\tau_*)$.  Define 
\begin{equation}\label{s split}
s:\,_n K_1(C^*(\Gamma))\to K_0(C^*(\Gamma);\Z/n),\quad s:=p\circ s_0.
\end{equation}
Using the diagram in line \eqref{part uct}, one sees that this is still a splitting for $\beta_n:K_0(C^*(\Gamma);\Z/n)\to\, _n K_1(C^*(\Gamma))$, and thus we get a direct sum decomposition 
$$
K_0(C^*(\Gamma);\Z/n)\cong \, ^nK_0(C^*(\Gamma))\oplus \, _nK_1(C^*(\Gamma))
$$
with the properties in the statement.
\end{proof}

\begin{proposition}\label{match the 0}
Let $\Gamma$ be a discrete group for which the BCK assembly map $\mu$ of line \eqref{max bc} is an isomorphism.  

With notation as in Definition \ref{lambda0}, Example \ref{basic lambda 0}, and Definition \ref{fin set tot hom}, let $P$ be a $K$-datum for $\underline{K}_0(C^*(\Gamma))$ with associated integer $N=N(P)$.

Let $P_0$ (respectively, $P_n$ for $n\geq 2$ and $n|N$) denote the (finitely generated) subgroup of $K_0(C^*(\Gamma))$ (respectively, $K_0(C^*(\Gamma);\Z/n)$ that forms part of the $\Lambda_0^{(N)}$-module $\langle P\rangle$ generated by $P$.  As $P$ contains the class of the unit of $C^*(\Gamma)$, we have a splitting 
\begin{equation}\label{g split}
P_0=\Z[1]\oplus \widetilde{P_0}
\end{equation}
where $\widetilde{P_0}:=P_0\cap \widetilde{K}_0(C^*(\Gamma))$.  We moreover write $\widetilde{P_n}$ for the intersection of $P_n$ and the subgroup $\rho_n(\widetilde{K}_0(C^*(\Gamma)))$ of $K_0(C^*(\Gamma);\Z/n)$ from Lemma \ref{zn split}.  Let $M\in \N$ be large enough so that multiplication by $M$ annihilates all classes in $P_n/\widetilde{P_n}$ for all $n\geq 2$ with $n|N$.  

Let $\alpha\in \text{Hom}_{\Lambda_0^{(N)}}(P,\underline{K}_0(\C))$, and write $\alpha_n:P_n\to K_0(\C;\Z/n)$ for the induced homomorphism for each $n$ such that $n=0$, or $n\geq 2$ and $n|N$.   Then the following are equivalent:
\begin{enumerate}[(i)]
\item \label{m01} with respect to the decomposition in line \eqref{g split}, $\alpha_0[1]> 0$ and $M\alpha_n(\widetilde{P_n})=\{0\}$ for all $n$ with $n=0$, or $n\geq 2$ and $n|N$;
\item \label{m02} there is a finite dimensional representation $\pi$ of $\Gamma$ such that $M\alpha=\pi_*$ in $\text{Hom}_{\Lambda_0^{(N)}}(P,\underline{K}_0(\C))$.
\end{enumerate}
\end{proposition}

\begin{proof}
Assume first that \eqref{m02} holds.  Then Corollary \ref{fd triv} implies that for any finite-dimensional representation $\pi$, the element $\pi_*\in \text{Hom}(P_0,\Z)$ satisfies $\pi_*[1]=\text{dim}(\pi)>0$, and $\pi_*(\widetilde{P_0})=\{0\}$.  As $\pi_*$ is moreover compatible with the morphism $\rho_n$ coming from $\Lambda_0$, we also have that $\rho_n(\widetilde{P_n})=\{0\}$ for all $n\geq 2$.  Hence we see that the conditions in part \eqref{m01} must hold.

Conversely, assume condition \eqref{m01} holds.  We claim that if $\tau:\Gamma\to \C$ is the trivial representation, then taking $\pi$ to be the direct sum of $M\cdot \alpha_0[1]$ copies of $\tau_*$ works.  Indeed, we must show that $\pi_*-M\alpha$ is zero on $P_0$ and $P_n$ for all $n\geq 2$ with $n|N$.  The observations in the first part of the proof and the assumptions from \eqref{m01} show that $\pi_*-M\alpha$ vanishes on $P_0$, and also on the subgroup $\widetilde{P_n}$ of $P_n$.  Hence as a map on $P_n\to \Z/n$, $\pi_*-M\alpha$ descends to a map from a subgroup $Q_n=P_n/\rho_n(P_0)$ of the direct summand $_nK_1(C^*(\Gamma))$ to $\Z/n$ (compare Lemma \ref{zn split}).  The map $\tau_*$ is zero on this direct summand by choice of the splitting in Lemma \ref{zn split}.  On the other hand, $M\alpha$ vanishes on $Q_n$ by choice of $M$, so we are done.
\end{proof}

\subsection{Finite coefficients and eta invariants}\label{fc subsec}

In Subsection \ref{basic bc subsec}, we essentially used the BCK assembly map of line \eqref{max bc} to compute the map
$$
\pi_*\in \text{Hom}(K_0(C^*(\Gamma)),K_0(\C))
$$
induced by a unitary representation $\pi:\Gamma\to M_m(\C)$.  In this section, we aim to compute the maps 
$$
\pi_*\in \text{Hom}(K_0(C^*(\Gamma);\Z/n),K_0(\C;\Z/n))
$$
on $K$-theory with finite coefficients induced by $\pi$.  This is more subtle: it turns out to be more convenient to compute the difference 
\begin{equation}\label{rep diff}
\tau_*-\pi_*:\text{Hom}(K_0(C^*(\Gamma);\Z/n),K_0(\C;\Z/n))
\end{equation}
where $\tau$ is a trivial representation of the same dimension as $\pi$; we are then able to do this in terms of the \emph{relative eta invariant} of Atiyah, Patodi, and Singer \cite{Atiyah:1975aa}.   

We will follow the approach of Higson and Roe to relative eta invariants from \cite{Higson:2008qb}.  For this, we need to recall the Baum-Douglas geometric model of $K$-homology \cite{Baum:1980pt}; more specifically, we will use the variant of this discussed by Higson and Roe in \cite[Section 3]{Higson:2008qb}, which is in turn based on work of Keswani \cite[Section 2.2]{Keswani:1999aa}.  Keswani's version of the geometric model of $K$-homology is shown to be equivalent to the original Baum-Douglas version in  \cite[Section 2.3]{Keswani:1999aa}.  We briefly recall the cycles for the odd geometric $K$-homology group $K_1(X)$ of a compact second countable space $X$ in this model; the even group is defined analogously, but we will not need this.

\begin{definition}\label{k cycle}
An \emph{odd $K$-cycle} for a compact second countable space $X$ is a triple $(M,S,f)$ where:
\begin{enumerate}[(i)]
\item $M$ is a smooth, closed, odd-dimensional, orientable, Riemannian manifold (it need not be connected, and the connected components need not have the same dimension as long as all are odd-dimensional);
\item $S$ is a Dirac bundle over $M$ in the sense of \cite[Definition 3.1]{Higson:2008qb} (the precise details of what this means will not be important to us);
\item $f:M\to X$ is a continuous map.
\end{enumerate}
\end{definition}

The odd geometric $K$-homology group $K_1(X)$ is then defined to consist of $K$-cycles modulo the equivalence relation defined in \cite[Definition 3.11]{Higson:2008qb} (again, the exact details are not important to us).  The groups $K_1(X)$ are covariantly functorial: if $g:X\to Y$ is a continuous map then $g_*:K_1(X)\to K_1(Y)$ is the map induced on $K$-cycles by 
\begin{equation}\label{bd func}
(M,S,f)\mapsto (M,S,g\circ f).
\end{equation}

We will also need Dirac operators associated to $K$-cycles.  Given an odd $K$-cycle $(M,S,f)$ one can associate a \emph{Dirac operator} 
\begin{equation}\label{dir op}
D=D^{M,S}
\end{equation}
as in \cite[Definition 3.5]{Higson:2008qb}: this is a first order elliptic partial differential operator acting on smooth sections of the bundle $S$.  The Dirac operator $D$ is not uniquely determined -- it depends on a choice of `compatible connection' as in \cite[Definition 3.4]{Roe:1998ad} -- but the choices involved will not be important for us.

There is a well-defined map from geometric $K$-homology to analytic $K$-homology defined on $K$-cycles by
\begin{equation}\label{geo ana}
(M,S,f)\mapsto f_*[D],
\end{equation}
where $D$ is a choice of Dirac operator as in line \eqref{dir op}, and $[D]$ is the class it defines in analytic $K$-homology (see for example \cite[Chapter 10]{Higson:2000bs}).  This map defines an isomorphism between geometric and analytic $K$-homology for finite CW complexes: see \cite[Section 4]{Jakob:1998aa}, \cite[Chapter 4]{Raven:2004aa}, or \cite{Baum:2009hq} for a proof.  Note that if $X$ is a possibly infinite CW complex and $RK_1(X)$ is defined via a direct limit as in line \eqref{rep k hom}, then classes in $RK_1(X)$ can still be described as triples $(M,S,f)$ with $f:M\to X$ a continuous map as in Definition \ref{k cycle} above: indeed, as $f:M\to X$ is continuous and $M$ is compact, $f$ takes image in a finite subcomplex of $X$.  

Let now $M$ be a closed, odd-dimensional, Riemannian manifold, equipped with a Dirac bundle $S$ as in Definition \ref{k cycle} and a choice of Dirac operator $D$ on $S$.  In their first paper on the subject \cite{Atiyah:1975ab}, Atiyah, Patodi, and Singer use spectral data to associate a real number $\eta(D)$ to $D$ called the \emph{eta invariant}.  Very roughly, $\eta(D)$ is the difference between the number of positive eigenvalues of $D$ and the number of negative eigenvalues; however, both numbers are typically infinite so this does not make literal sense, and $\eta(D)$ is actually defined using a sort of `zeta function regularization'.

The invariant $\eta(D)$ is quite delicate: it depends on $D$, not only on the underlying topological data from the manifold $M$ and bundle $S$.  In the second paper \cite{Atiyah:1975aa} of their series, Atiyah, Patodi, and Singer use $\eta(D)$ to define a more robust \emph{relative eta invariant} as follows.  Assume that $M$ is equipped with a homomorphism $\pi_1(M)\to \Gamma$, and let $\pi:\Gamma\to M_m(\C)$ be a finite dimensional unitary representation.  This data defines a rank $m$ flat bundle $E_\pi$ over $M$ analogously to line \eqref{fb} above, and we write $D_\pi$ for the Dirac operator $D$ twisted\footnote{\label{twist fn}See for example \cite[pages 138-9]{Jr.:1989uq} or \cite[Example 3.24]{Roe:1998ad} for a Dirac operator twisted by (or equivalently, `with coefficients in') a vector bundle.} by this bundle.  On the other hand, let $D_m$ denote $D$ twisted by the rank $m$ trivial bundle $\C^m$ on $M$.  Let $k_\pi$ (respectively, $k_m$) denote the dimension of the kernel of $D_\pi$ (respectively, of $D_m$) acting on smooth sections of the associated bundle tensored by $E_\pi$ (respectively, by $\C^m$).  Then the \emph{relative eta invariant}\footnote{We warn the reader that we use the notation ``$\rho$'' by analogy with \cite[Definition 2.3]{Higson:2008qb}, but Atiyah-Patodi-Singer instead use ``$\widetilde{\xi}_\pi(0)$''; what we call ``$\rho_\pi$'' is not the same as what Atiyah-Patodi-Singer call ``$\rho_\pi$'' in \cite[Theorem 2.4]{Atiyah:1975aa}.} is defined in \cite[Line (3.2)]{Atiyah:1975aa} by 
\begin{equation}\label{rel eta 0}
\rho_\pi(D):=\frac{k_\pi+\eta(D_\pi)}{2}-\frac{k_m+\eta(D_m)}{2}\in \R/\Z.
\end{equation}
The reason for considering $\rho_\pi(D)$ as an element of $\R/\Z$ rather than $\R$ is that is has stronger invariance properties this way.   From a modern point of view, the invariance properties of $\rho_\pi$ are well-summarized by the following theorem, which is \cite[Theorem 6.1]{Higson:2008qb}; the reader might also usefully compare this to the material in \cite[Section 6]{Antonini:2014aa} which gives an approach based on von Neumann algebras rather than the Baum-Douglas geometric model for $K$-homology.

\begin{theorem}[Higson-Roe]\label{rel eta the}
Let $\Gamma$ be a countable group, let $(M,S,f)$ be an odd $K$-cycle for $B\Gamma$ as in Definition \ref{k cycle}, and $D=D^{M,S}$ a choice of associated Dirac operator as in line \eqref{dir op}.   Equip $\pi_1(M)$ with the homomorphism $f_*:\pi_1(M)\to \pi_1(B\Gamma)=\Gamma$ induced by $f$, and let $\pi:\Gamma\to M_m(\C)$ be a finite dimensional unitary representation. Then the assignment 
$$
(M,S,f)\mapsto \rho_\pi(D)
$$
descends to a well-defined homomorphism 
$$
\rho_\pi:RK_1(B\Gamma)\to \R/\Z.\eqno\qed
$$
\end{theorem}

Our first main goal in this section is to relate the map in line \eqref{rep diff} to the relative eta homomorphism $\rho_\pi$ of Theorem \ref{rel eta the}.  We need some more $K$-theoretic ingredients that we now describe.  

There is a Bockstein six-term exact sequence in (representable) $K$-homology, analogous to the Bockstein sequence in line \eqref{bs}
\begin{equation}\label{kh bs}
\xymatrix{ RK_0(B\Gamma) \ar[r]^-{\times n} & RK_0(B\Gamma) \ar[r]^-{\rho_n} & RK_0(B\Gamma;\Z/n) \ar[d]^{\beta_n} \\
RK_1(B\Gamma;\Z/n) \ar[u]^-{\beta_n} & RK_1(B\Gamma) \ar[l]^-{\rho_n} & RK_1(B\Gamma) \ar[l]^-{\times n} }
\end{equation}
defined using Kasparov product with the same $KK$ elements mentioned above, and then taking a direct limit over finite subcomplexes of $B\Gamma$ (here we use that direct limits preserve exactness).  The maps $\beta_n$ in line \eqref{bs} or \eqref{kh bs} are usually called \emph{Bockstein homomorphisms}.

Here is the first main result of this section.  We suspect it may be know to some experts, but are not aware of it appearing in the literature before.

\begin{theorem}\label{eta inv}
Let $\Gamma$ be a discrete group.  Let $\pi:\Gamma\to M_m(\C)$ be a unitary representation of $\Gamma$ such that the restriction of $\pi$ to any finitely generated subgroup factors through a finite quotient.  Let $\tau=\tau^{(m)}:\Gamma\to M_m(\C)$ be the $m$-dimensional trivial representation.  Let $n\geq 2$.  Then the following diagram commutes
$$
\xymatrix{ RK_0(B\Gamma;\Z/n) \ar[d]^{\beta_n} \ar[r]^-\mu & K_0(C^*(\Gamma);\Z/n) \ar[rr]^-{\tau_*-\pi_*}  & & \Z/n \ar[d]  \\
RK_1(B\Gamma) \ar[rrr]^-{\rho_{\pi}} & & & \R/\Z }
$$
where: the top right horizontal arrow is the map of line \eqref{rep diff}; the top left horizontal arrow is the BCK assembly map with finite coefficients of line \eqref{bc fc}; the left vertical arrow is the Bockstein homomorphism appearing in the exact sequence of line \eqref{kh bs}; the right vertical arrow sends the canonical generator $1\in \Z/n$ to $1/n$; and the bottom horizontal arrow $\rho_\pi$ is the relative eta invariant homomorphism of Theorem \ref{rel eta the}.
\end{theorem}

This is almost implict in the papers of Atiyah, Patodi, and Singer \cite{Atiyah:1975aa,Atiyah:1976aa}.  However, those authors do not work with $K$-homology (which was not available in a convenient model at that time), but rather with $K$-theory of tangent bundles.  The two are related by Kasparov's Poincar\'{e} duality: see \cite[Section 4]{Kasparov:2016aa} for the version we will use.  The proof will consist largely of combining these ingredients, but we must first recall some background.

The first collection of facts we need is based on \cite[Section 5]{Atiyah:1975aa}.  Let $A$ be a (unital) $C^*$-algebra, and for each $n,m\geq 1$, let 
\begin{equation}\label{kappa nm}
\kappa_{nm,n}:K_*(A;\Z/n)\to K_*(A;\Z/(nm))
\end{equation} 
be the natural transformation induced by any unital $*$-homomorphism $O_{n+1}\to M_m(O_{nm+1})$ inducing the canonical injection $\Z/n\to \Z/(nm)$ defined by sending $1$ to $m$; such a homomorphism exists by the Kirchberg-Phillips classification theory as for example in \cite[Theorem A]{Gabe:2019ws}.  Let $\mathcal{N}$ denote the set $\{n\in \N\mid n\geq 2\}$ equipped with the partial order where $n\leq m$ if $n$ divides $m$.  Then the collection 
\begin{equation}\label{dir sys qz}
(K_*(A;\Z/n))_{n\in \mathcal{N}}
\end{equation}
is a directed system with the maps $\kappa_{nm,n}$ as connecting maps.  The following definitions are based on \cite[Section 5]{Atiyah:1975aa}.

\begin{definition}\label{qz def}
Let $A$ be a $C^*$-algebra.  We define the \emph{$K$-theory of $A$ with $\Q/\Z$ coefficients} as the direct limit  
$$
K_*(A;\Q/\Z):=\lim_{\to} K(A;\Z/n)
$$
over the directed set of line \eqref{dir sys qz}.  If $X$ is a possibly infinite CW complex, we define 
$$
\mathcal{K}^*(X;\Q/\Z):=\lim_{\leftarrow} K^*(Y;\Q/\Z)
$$
where the inverse limit is taken over all finite subcomplexes $Y$ of $X$ analogously to Definition \ref{inv lim k}.
\end{definition}

Analogously to line \eqref{pair n} above for any compact metric space $Y$ and each $n\geq 2$, the Kasparov product induces a pairing 
$$
KK_1(\C,C(Y)\otimes O_{n+1})\otimes KK_1(C(Y),\C)\to KK(\C,O_{n+1})
$$
or in other words 
$$
K^1(Y;\Z/n)\otimes K_1(Y)\to K_0(O_{n+1})=\Z/n.
$$
Taking the direct limit over the directed system in line \eqref{dir sys qz} then gives a pairing
\begin{equation}\label{base pair qz}
K^1(Y;\Q/\Z)\otimes K_1(Y)\to \lim_{n\in \mathcal{N}} K_0(O_{n+1})=\Q/\Z.
\end{equation}
Finally, if $X$ is a general CW complex, these pairings induce a pairing 
\begin{equation}\label{lim qz pair}
\mathcal{K}^1(X;\Q/\Z)\otimes RK_1(X)\to \Q/\Z
\end{equation}
by taking limits over finite CW subcomplexes.

Now, the Bockstein exact sequences of line \eqref{bs} for $n$ and $nm$ fit into diagrams 
\begin{equation}\label{bs sys}
\xymatrix{ \cdots \ar[r]^-{\times n} &  K_1(A) \ar[r]^-{\rho_n}  \ar[d]^-{\times m} &  K_1(A;\Z/n) \ar[d]^-{\kappa_{nm,n}} \ar[r]^-{\beta_n} & K_0(A) \ar[d]^{\text{id}} \ar[r]^-{\times n}  & K_0(A) \ar[r]^-{\rho_n} \ar[d]^-{\times m}& \cdots \\
\cdots \ar[r]^-{\times nm} &  K_1(A) \ar[r]^-{\rho_{nm}} &  K_1(A;\Z/(nm))  \ar[r]^-{\beta_{nm}} & K_0(A)  \ar[r]^-{\times nm} & K_0(A) \ar[r]^-{\rho_{nm}} & \cdots }
\end{equation}
(the top and bottom lines are `unrolled' versions of the Bockstein six-term sequences from line \eqref{bs}); these commute by the identities in \cite[Proposition 2.1]{Schochet:1984ab}, plus the essential uniqueness of the natural transformations involved as in \cite[Lemma A.4]{Carrion:2020aa}.  Taking the direct limit of the six-term exact sequences in line \eqref{bs sys} over the directed set $\mathcal{N}$ defined just above line \eqref{dir sys qz} gives a six-term exact sequence
\begin{equation}\label{bs2}
\xymatrix{ K_0(A) \ar[r]^-\iota & K_0(A)\otimes \Q \ar[r]^-{\rho} & K_0(A;\Q/\Z) \ar[d]^-\beta \\
K_1(A;\Q/\Z) \ar[u]^-{\beta} & K_1(A)\otimes \Q \ar[l]^-{\rho} & K_1(A) \ar[l]^-\iota }
\end{equation}
where the arrows labeled $\iota$ are induced by the canonical inclusion $\Z\to \Q$, and the arrows labeled $\rho$ and $\beta$ are the direct limits of the arrows $\rho_n$ and $\beta_n$ respectively.  Moreover, if $A=C(Y)$ is the continuous functions on a connected compact Hausdorff space, then passing to reduced $K^0$ groups $\widetilde{K}^0$ (and noting that the reduced $K^1$ groups are the same as the usual $K^1$ groups) the sequence 
\begin{equation}\label{red bln}
\xymatrix{ \widetilde{K}^0(Y) \ar[r]^-\iota & \widetilde{K}^0(Y)\otimes \Q \ar[r]^-{\rho} & \widetilde{K}^0(Y;\Q/\Z) \ar[d]^-\beta \\
K^1(Y;\Q/\Z) \ar[u]^-{\beta} & K^1(Y)\otimes \Q \ar[l]^-{\rho} & K^1(Y) \ar[l]^-\iota }
\end{equation}
is a direct summand of the sequence in line \eqref{bs2}, and in particular still exact.  

Finally, if $X$ is a possibly infinite connected CW complex, we write 
\begin{equation}\label{lim beta map}
\beta: \mathcal{K}^1(X;\Q/\Z) \to \widetilde{\mathcal{K}}^0(X).
\end{equation}
for the map induced by taking the inverse limits of the left hand vertical maps in line \eqref{red bln} as $Y$ ranges over finite subcomplexes of $X$ (this makes sense by naturality of $\beta$).  Here we let $\widetilde{\mathcal{K}}^0(X)$ be the kernel of the map $\mathcal{K}^0(X)\to \Z$ induced by the inclusion of a point.

The next lemma is implicit in \cite[Section 5]{Atiyah:1975aa}; we provide a proof for the reader's convenience.  Recall first that if $\Lambda$ is a finite group then there is a CW model for its classifying space that is finite in each dimension -- for example, the infinite join construction of the classifying space from \cite[Section 3]{Milnor:1956aa} has this property -- but if $\Lambda$ is not trivial then there is no finite model for $B\Lambda$ (see for example \cite[Corollary VIII.2.5]{Brow:1982rt}).

\begin{lemma}\label{fin gp}
Let $\Lambda$ be a finite group and let $B\Lambda$ be a CW complex model for its classifying space that is finite in each dimension.  Then the map $\beta$ of line \eqref{lim beta map} is an isomorphism for $X=B\Lambda$.
\end{lemma}

\begin{proof}
Let $B\Lambda^{(n)}$ denote the $n$-skeleton of $B\Lambda$; as this is finite for all $n$, $\mathcal{K}^0(B\Lambda)$ is the inverse limit of the inverse system $(K^0(B\Lambda^{(n)}))_{n=1}^\infty$.  Similarly, $\mathcal{K}^1(B\Lambda;\Q/\Z)$ is the inverse limit of the system $(K^1(B\Lambda^{(n)};\Q/\Z))_{n=1}^\infty$.  

As $B\Lambda$ has trivial rational cohomology\footnote{This is well-known, but we were unable to find a good reference.  One can for example show this purely algebraically by showing that $\Q$ with the trivial $\Lambda$-action is a flat $\Z\Lambda$-module.}, considering cellular cohomology shows that $\widetilde{H}^k(B\Lambda^{(n)};\Q)$ is zero if $k\neq n$.  As the Chern character is a rational isomorphism, $\widetilde{K}^0(B\Lambda^{(n)})\otimes \Q=0$ for $n$ odd, and $K^1(B\Lambda^{(n)})\otimes \Q=0$ for $n$ even.  Now, consider the short exact sequence for $n$ odd
\begin{equation}\label{ml ses}
0\to \frac{K^1(B\Lambda^{(n)})\otimes \Q}{\iota(K^1(B\Lambda^{(n)}))}  \stackrel{\rho}{\to} K^1(B\Lambda^{(n)};\Q/\Z) \stackrel{\beta}{\to} \widetilde{K}^0(B\Lambda^{(n)}) \to 0
\end{equation}
induced from line \eqref{red bln} and the above observations.  As $K^1(B\Lambda^{(n)})\otimes \Q=0$ for $n$ even, the connecting maps for the inverse system 
$$
\Bigg(\frac{K^1(B\Lambda^{(n)})\otimes \Q}{\iota(K^1(B\Lambda^{(n)}))}\Bigg)_{n\text{ odd}}
$$
are all zero.  In particular, this inverse system satisfies the Mittag-Leffler condition so the inverse limit of the short exact sequence in line \eqref{ml ses} is still exact (see for example \cite[Pages 31-2]{Atiyah:1961uq} or \cite[Section 3.5]{Weibel:1995ty}), giving a short exact sequence
$$
0\to K  \stackrel{\rho}{\to} \mathcal{K}^1(B\Lambda;\Q/\Z) \stackrel{\beta}{\to} \widetilde{\mathcal{K}}^0(B\Lambda) \to 0.
$$
The group $K$ on the left is an inverse limit of a system with zero connecting maps, so is trivial, and we are done.
\end{proof}

We need one more lemma before the proof of Theorem \ref{eta inv}.

\begin{lemma}\label{pair commute}
Let $\Gamma$ be a countable group, let $y\in \mathcal{K}^1(B\Gamma;\Q/\Z)$, and let $n\geq 2$.  Then the diagram below commutes 
$$
\xymatrix{ RK_0(B\Gamma;\Z/n) \ar[r]^-{p_{\beta(y)}} \ar[d]^-{\beta_n} & \Z/n \ar[d] \\
RK_1(B\Gamma) \ar[r]^-{p_y} & \Q/\Z }
$$
where: the left hand map $\beta_n$ is the right vertical map from line \eqref{kh bs}; $\beta(y)$ is the image of $y$ under the map $\beta$ of line \eqref{lim beta map} and $p_{\beta(y)}$ is the operation of pairing with it as in line \eqref{pair n}; $p_y$ is the operation of pairing with $y$ as in line \eqref{lim qz pair}; and the right hand vertical arrow is the canonical inclusion determined by $1\mapsto 1/n$.
\end{lemma}

\begin{proof}
Thanks to the definitions of the pairings in lines \eqref{pair n} and \eqref{lim qz pair}, it suffices to show that if $Y$ is a finite CW subcomplex of $B\Gamma$ and $z$ is the image of $y$ under the canonical map $\mathcal{K}^1(B\Gamma;\Q/\Z)\to K^1(Y;\Q/\Z)$, then the following diagram commutes 
$$
\xymatrix{ KK(C(Y),O_{n+1}) \ar[r]^-{p_{\beta(z)}} \ar[d]^-{\beta_n} & KK(\C,O_{n+1})=\Z/n \ar[d] \\
KK_1(C(Y),\C) \ar[r]^-{p_{z}} & \lim_m KK(\C,O_{m+1})= \Q/\Z }
$$
where: $\beta$ is the left hand vertical arrow in line \eqref{bs2} for $A=C(Y)$ and $p_{\beta(z)}$ is the result of pairing with $\beta(z)\in KK_1(\C,C(Y))$ as in line \eqref{kas 1 pair}; $p_{z}$ is the result of pairing with $z\in \lim_m KK(\C,C(Y)\otimes O_{m+1})$ as in line \eqref{base pair qz}; and the two vertical maps are defined as before.  Let now $m\in \N$ be such that $n$ divides $m$, and such that $z$ is the image of an element $z_m$ under the map $KK(\C,C(Y)\otimes O_{m+1})\to \lim_m KK(\C,C(Y)\otimes O_{m+1})$.  Then it suffices to show that the diagram below commutes
\begin{equation}\label{red y m}
\xymatrix{ KK(C(Y),O_{n+1}) \ar[r]^-{p_{\beta_m(z_m)}} \ar[d]^-{\beta_n} & KK(\C,O_{n+1})=\Z/n \ar[d]^{\kappa_{m,n}} \\
KK_1(C(Y),\C) \ar[r]^-{p_{z_m}} & KK(\C,O_{m+1})= \Z/m }
\end{equation}
where $\kappa_{m,n}\in KK(O_{n+1},O_{m+1})$ is as in line \eqref{kappa nm}.  Let $b_m\in KK_1(O_{m+1},\C)$ be the element inducing the map $\beta_m$, and similarly for $b_n$.  Then the image of $x\in KK(C(Y),O_{n+1})$ under the right-down composition in line \eqref{red y m} is given by 
$$
((z_m\cdot b_m)\cdot  x)\cdot \kappa_{m,n}
$$
while the down-right composition is given by 
$$
z_m\cdot (x\cdot b_n);
$$
using associativity of the Kasparov product, we therefore must show that 
$$
b_m\cdot x \cdot \kappa_{m,n}=x\cdot b_n
$$
in $KK_1(O_{m+1}\otimes C(Y),O_{m+1})$.  Moreover, $b_m\cdot x=x\cdot b_m$, so it suffices to show that $b_m\cdot \kappa_{m,n}=b_n$ in $KK_1(O_{n+1},\C)$.  It follows from \cite[Proposition 2.1]{Schochet:1984ab} that $b_m\cdot \kappa_{m,n}$ and $b_n$ induce the same natural transformations on $K$-theory with finite coefficients, from which the identity $b_m\cdot \kappa_{m,n}=b_n$ follows (compare for example \cite[Remark 8.8]{Rosenberg:1987bh}).
\end{proof}


We are now ready for the proof of Theorem \ref{eta inv}.

\begin{proof}[Proof of Theorem \ref{eta inv}]
The diagram appearing in Theorem \ref{eta inv} is continuous as one takes direct limits over groups.  Thus we may assume that $\Gamma$ is finitely generated, and that $\pi$ factors through a finite quotient of $\Gamma$.  

Let $E_\pi$ be the flat bundle over $B\Gamma$ corresponding to $\pi$ as in line \eqref{fb} above, and let $\C^m$ be the corresponding trivial bundle of the same rank.  Let 
$$
p_\pi :RK_0(B\Gamma;\Z/n)\to \Z/n,\quad x\mapsto ([E_\pi]-[\C^m])\cdot x 
$$
be the operation of pairing with $[E_\pi]-[\C^m]\in \widetilde{\mathcal{K}}^0(B\Gamma)$ as in line \eqref{pair n} above.  Then the diagram
$$
\xymatrix{ RK_0(B\Gamma;\Z/n) \ar[d]^-{p_\pi} \ar[r]^\mu & K_0(C^*(\Gamma);\Z/n) \ar[d]^-{\pi_*-\tau_*} \\
\Z/n \ar@{=}[r] & \Z/n }
$$
commutes by Lemma \ref{flat pair}.  It thus suffices to show that 
\begin{equation}\label{stp com}
\xymatrix{ RK_0(B\Gamma;\Z/n) \ar[r]^-{p_\pi} \ar[d]^-{\beta_n} & \Z/n \ar[d]  \\
RK_1(B\Gamma) \ar[r]^-{\rho_{\pi}} & \R/\Z }
\end{equation}
commutes.  

Let $q:\Gamma\to \Lambda$ be a finite quotient of $\Gamma$ through which $\pi$ factors.  Let $[E_\pi^\Lambda]-[\C^m]\in \widetilde{\mathcal{K}}^0(B\Lambda)$ be the corresponding class where $E_{\pi}^\Lambda$ is the flat bundle associated to $\pi$ (considered as a representation of $\Lambda$) as in line \eqref{fb}.  Let $q:B\Gamma\to B\Lambda$ also denote the induced map on classifying spaces.  Then the diagram 
\begin{equation}\label{q beta}
\xymatrix{ \widetilde{\mathcal{K}}^0(B\Lambda) \ar[r]^-{q^*} & \widetilde{\mathcal{K}}^0(B\Gamma) \\
\mathcal{K}^1(B\Lambda;\Q/\Z) \ar[u]^-{\beta} \ar[r]^-{q^*} & \mathcal{K}^1(B\Gamma;\Q/\Z) \ar[u]^-\beta} 
\end{equation}
commutes by naturality of the map $\beta$ from line \eqref{lim beta map} above.  The left hand map $\beta$ in line \eqref{q beta} above is an isomorphism by Lemma \ref{fin gp}, so we may define $y\in \mathcal{K}^1(B\Gamma;\Q/\Z)$ by 
$$
y:=q^*(\beta^{-1}([E_\pi^\Lambda]-[\C^m]))\in \mathcal{K}^1(B\Gamma;\Q/\Z)
$$
(compare \cite[page 431]{Atiyah:1975aa}).   Then by commutativity of the diagram in line \eqref{q beta}, $\beta(y)=[E^\pi]-[\C^m]\in \widetilde{\mathcal{K}}^0(B\Gamma)$.  Hence we have an equality of maps 
\begin{equation}\label{pi beta y}
p_{\beta(y)}=p_\pi:RK_0(B\Gamma;\Z/n)\to \Z/n.
\end{equation}
where the map $p_{\beta(y)}$ is as in Lemma \ref{pair commute}.  As Lemma \ref{pair commute} gives us the commutative diagram
$$
\xymatrix{ RK_0(B\Gamma;\Z/n) \ar[r]^-{p_{\beta(y)}} \ar[d]^-{\beta_n} & \Z/n \ar[d] \\
RK_1(B\Gamma) \ar[r]^-{p_y} & \Q/\Z }
$$
and as we have the identity in line \eqref{pi beta y}, to show that the diagram in line \eqref{stp com} commutes (recall this is our goal!) if suffices to show an identity of maps
\begin{equation}\label{desid 2}
\rho_{\pi}=p_y:RK_1(B\Gamma)\to \Q/\Z.
\end{equation}
This is essentially the Atiyah-Patodi-Singer index theorem for flat bundles, as we explain in the remainder of the proof.

Indeed, let $z\in RK_1(B\Gamma)$ be an arbitrary class represented by a $K$-cycle $(M,S,f)$ as in Definition \ref{k cycle}.  Let $[D]\in K_1(M)$ be the class of an associated Dirac operator in analytic $K$-homology, and let $[\sigma_{D}]$ be the symbol class of $D$ in $K^1(TM)$.  Then on \cite[page 87]{Atiyah:1976aa}, Atiyah-Patodi-Singer construct an \emph{analytical index map} $\text{ind}_\pi:K^1(TM)\to \R/\Z$ such that the equality
\begin{equation}\label{aps 3}
\text{ind}_\pi([\sigma_{D}])=\rho_\pi(z)
\end{equation}
holds by definition (compare Theorem \ref{rel eta the} above).  On the other hand on \cite[page 87]{Atiyah:1976aa} again, Atiyah-Patodi-Singer define a \emph{topological index map} $\text{Ind}_\pi:K^1(TM)\to \R/\Z$ by the formula 
\begin{equation}\label{aps 4}
\text{Ind}_\pi([\sigma_{D}]):=\text{Ind}(f^*(y)\cdot_{TM} [\sigma_{D}])
\end{equation}
where: the product $\cdot_{TM}$ on the right hand side is the module action of $K^1(M;\Q/\Z)$ on $K^1(TM)$ 
$$
K^1(M;\Q/\Z)\otimes K^1(TM)\to K^0(TM)
$$
(a variation incorporating suspensions and coefficients of the module action of $K^0(M)$ on $K^0(TM)$ defined on \cite[page 491]{Atiyah:1968ul}); and $\text{Ind}$ is the \emph{Atiyah-Singer topological index map} \cite[Section 3]{Atiyah:1968ul}
$$
\text{Ind}:K^0(TM;\Q/\Z)\to K^0(pt;\Q/\Z)=\Q/\Z
$$
adapted for $\Q/\Z$ coefficients.  Using \cite[Theorem 5.3]{Atiyah:1976aa} (compare also \cite[Remark 2 on page 87]{Atiyah:1976aa}), the left hand sides of lines \eqref{aps 3} and \eqref{aps 4} agree, and so 
\begin{equation}\label{rho pi ind}
\rho_\pi(z)=\text{Ind}(f^*(y)\cdot_{TM} [\widetilde{\sigma_{D}}]).
\end{equation}
To establish the formula in line \eqref{desid 2}, we follow Kasparov's approach to index theory as exposited in \cite[Sections 2-4]{Kasparov:2016aa}.  

Now, analogously to Kasparov \cite[page 1318]{Kasparov:2016aa}, we may consider the symbol $\sigma_D$ of $D$ as defining a class in $KK_1(C(M),C_0(TM))$\footnote{Kasparov just calls this class $[\sigma_D]$, and considers it as an element of a slightly more complicated group $\mathcal{R}K_1(M;C_0(TM))$, where the latter group would be denoted $\mathcal{R}KK_1(M;C(M),C_0(TM))$ in \cite[2.19]{Kasparov:1988dw} (compare \cite[page 1306]{Kasparov:2016aa}).   We only need the image of Kasparov's class under the forgetful map $\mathcal{R}K_1(M;C(M),C_0(TM))\to KK_1(C(M),C_0(TM))$, so just work there for simplicity.}, which we denote $_M[\sigma_D]$.  In the language of the Kasparov product, we have 
$$
f^*(y)\cdot_{TM} [\sigma_{D}]=f^*(y)\cdot\, _M[\sigma_{D}].
$$
Thus we may rewrite line \eqref{aps 4} as 
\begin{equation}\label{aps 5}
\text{Ind}_\pi([\sigma_{D}])=\text{Ind}(f^*(y)\cdot \,_M[\sigma_{D}]).
\end{equation}
Let $[\mathcal{D}_{M}]\in K_0(TM)=KK(C_0(TM),\C)$ denote Kasparov's \emph{Dolbeault $K$-homology class} as in \cite[Definition 2.8]{Kasparov:2016aa}.  Then we have \emph{Kasparov's Poincar\'{e} duality} (see \cite[Theorem 4.2]{Kasparov:2016aa}) 
\begin{equation}\label{kas pd}
_M[\sigma_{D}]\cdot [\mathcal{D}_{M}] = [D] \in K_1(M)
\end{equation}
and (still by \cite[Theorem 4.2]{Kasparov:2016aa}, or more precisely, a slight variant with coefficients) that for any $w$ in
$$
K^0(TM;\Q/\Z)=\lim_n KK(\C,C_0(TM)\otimes O_{n+1})
$$
we have that 
\begin{equation}\label{kas as}
\text{Ind}(w)=[w]\cdot  [\mathcal{D}_{M}].
\end{equation}
Hence if we write $[f]\in \lim_Y KK(C(Y),C(M))$ for the $KK$-class associated to $f$ (the limit is taken over all finite subcomplexes $Y$ of $B\Gamma$), then starting with lines \eqref{rho pi ind} and \eqref{aps 5}, we have
\begin{align*}
\rho_\pi(z) & = \text{Ind}(f^*(y)\cdot\, _M[\sigma_{D}])\\
& =((y\cdot[f])\cdot \, _M[\sigma_{D}]) \cdot [\mathcal{D}_{M}]
\end{align*}
where we have used that $f^*(y)=y\cdot [f]$ (by definition) and line \eqref{kas as}.  Using associativity of the Kasparov product and Poincar\'{e} duality as in line \eqref{kas pd}, we thus get the first equality in the chain below 
\begin{align*}
\rho_\pi(z) =y\cdot [f]\cdot [D]  = y\cdot f_*[D]  = y\cdot z  = p_y(z).
\end{align*}
The other equalities are justified by: the second equality is by definition of $f_*[D]$; the third follows as $f_*[D]$ is the analytic $K$-homology class associated to the $K$-cycle $(M,S,f)$ underlying $z$ (compare line \eqref{geo ana} above); and the last follows by definition of $p_y$ (see Lemma \ref{pair commute} above).  This is the desired identity from line \eqref{desid 2}, so we are done.
\end{proof}

For our concrete examples, we will need the following explicit computation.  Again, it is essentially due to Atiyah, Patodi, and Singer.

\begin{proposition}\label{circle eta}
Let $\Gamma$ be a discrete group.  Let $[S^1,\C,f]\in RK_1(B\Gamma)$ be the class associated to a continuous map $f:S^1\to B\Gamma$.  Let $\pi:\Gamma\to \C$ be a one-dimensional unitary representation.  Let $f_*:\Z=\pi_1(S^1)\to \Gamma$ be the map induced on fundamental groups by $f$, and assume that $\pi\circ f_*$ takes $1\in \Z$ (identified with the class of a loop of winding number one) to $e^{2\pi i q}$ for some $q\in (0,1)$.  Then 
$$
\rho_\pi([S^1,\C,f])=-q \in \R/\Z.
$$
\end{proposition}

\begin{proof}
The relative eta invariant map of Theorem \ref{rel eta the} is natural, as follows directly from the description of functoriality for $K$-cycles in line \eqref{bd func}.  Hence the diagram below commutes
$$
\xymatrix{ K_1(S^1) \ar[d]^{f_*} \ar[r]^-{\rho_{\pi\circ f_*}} & \R/\Z \ar@{=}[d]\\
RK_1(B\Gamma) \ar[r]^{\rho_{\pi}} & \R/\Z }.
$$
It thus suffices to show that if $\sigma:\pi_1(S^1)=\Z\to \C$ is the one-dimensional unitary representation taking $1$ to $e^{2\pi i q}$, and 
$$
\rho_\sigma: K_1(B\Z)\to \R/\Z
$$ 
is the homomorphism of Theorem \ref{rel eta the}, then (having identified $S^1$ with $B\Z$), $\rho_\sigma[S^1,\C,\text{id}]=-q$.  Now, a canonical choice of Dirac operator associated to the $K$-cycle $(S^1,\C,\text{id})$ is $D=i\frac{d}{dx}$ acting on the trivial line bundle (here we treat $S^1$ as $\R/\Z$ and $x$ as the usual variable).   With notation as in line \eqref{rel eta 0} above, we thus need to show that 
$$
\frac{k_\sigma+\eta(D_\sigma)}{2}-\frac{k_1+\eta(D_1)}{2} = -q
$$
The necessary direct computations for this are contained in \cite[page 411]{Atiyah:1975aa}.  Indeed, for $0<q<1$, \cite[page 411]{Atiyah:1975aa} shows that $D_\sigma$ has trivial kernel, and the kernel of $D_1$ is spanned by the constant functions; hence $k_\sigma=0$ and $k_1=1$.  On the other hand, it is also computed in \cite[page 411]{Atiyah:1975aa} that $\eta(D_\sigma)=1-2q$ and that $\eta(D_1)=0$.  The result follows.
\end{proof}

The next theorem puts together our work in Section \ref{bc sec}; it was motivated partly by comments of Marius Dadarlat.

\begin{theorem}\label{match the}
Let $\Gamma$ be a discrete group, for which the BCK assembly map $\mu$ of line \eqref{max bc} is an isomorphism.  Assume moreover that any torsion class in $RK_1(B\Gamma)$ can be represented by a $K$-cycle of the form $(S^1,\C,f)$ for some continuous map $f:S^1\to B\Gamma$.  

With notation as in Definition \ref{lambda0}, Example \ref{basic lambda 0}, and Definition \ref{fin set tot hom}, let $P$ be a $K$-datum for $\underline{K}_0(C^*(\Gamma))$, with associated integer $N=N(P)$.

Let $P_0$ (respectively, $P_n$ for $n\geq 2$ and $n|N$) denote the (finitely generated) subgroup of $K_0(C^*(\Gamma))$ (respectively, $K_0(C^*(\Gamma);\Z/n)$ that forms part of the $\Lambda_0^{(N)}$-module $\langle P\rangle$ generated by $P$.  As $P$ contains the class of the unit of $C^*(\Gamma)$, we have a splitting 
\begin{equation}\label{g split}
P_0=\Z[1]\oplus \widetilde{P_0}
\end{equation}
where $\widetilde{P_0}:=P_0\cap \widetilde{K}_0(C^*(\Gamma))$.  We moreover write $\widetilde{P_n}$ for the intersection of $P_n$ and the subgroup $\rho_n(\widetilde{K}_0(C^*(\Gamma)))$ of $K_0(C^*(\Gamma);\Z/n)$ from Lemma \ref{zn split}.

Let $\alpha\in \text{Hom}_{\Lambda_0^{(N)}}(P,\underline{K}_0(\C))$, and write $\alpha_n:P_n\to K_0(\C;\Z/n)$ for the induced homomorphism for each $n$ with $n=0$, or $n\geq 2$ and $n|N$.  Then the following are equivalent:
\begin{enumerate}[(i)]
\item \label{2m01} with respect to the decomposition in line \eqref{g split}, $\alpha_0[1]\geq  0$ and $\alpha_n(\widetilde{P_n})=\{0\}$ for all $n$ with $n=0$, or $n\geq 2$ and $n|N$;
\item \label{2m02} there is a finite dimensional representation $\pi$ of $\Gamma$ and a one-dimensional representation $\sigma:\Gamma\to \C$ with torsion image such that $\alpha+\sigma_*=\pi_*$ in $\text{Hom}_{\Lambda_0^{(N)}}(P,\underline{K}_0(\C))$.
\end{enumerate}
\end{theorem}

\begin{proof}
Assume first that \eqref{2m02} holds.  Then line \eqref{fd triv} implies that for any finite-dimensional representation $\pi$ (in particular, for the character $\sigma$), the element $\pi_*\in \text{Hom}(P_0,\Z)$ satisfies $\pi_*[1]=\text{dim}(\pi)\geq 1$, and $\pi_*(\widetilde{P_n})=\{0\}$ for all $n$.  Hence if $\alpha$ is such that $\alpha+\sigma_*=\pi_*$ in $\text{Hom}_{\Lambda_0^{(N)}}(P,\underline{K}_0(\C))$, then $\alpha_0$ must satisfy $\alpha_0[1]\geq 0$, and for any $g\in \mathcal{P}_0$, $\alpha_0(g)=0$.

Conversely, assume condition \eqref{2m01} holds.  For each $n\geq 2$, choose splittings 
\begin{equation}\label{bock dec}
K_0(C^*(\Gamma);\Z/n)\cong \, ^nK_0(C^*(\Gamma))\oplus \, _nK_1(C^*(\Gamma))
\end{equation} 
with the properties in Lemma \ref{zn split}.  Using our assumption that the BCK assembly map is an isomorphism (which implies that the BCK assembly map with any $\Z/n$ coefficients is an isomorphism by Lemma \ref{bck fc lem}), we pull this back to a decomposition
\begin{equation}\label{bock dec kh}
RK_0(B\Gamma;\Z/n)\cong \, ^nRK_0(B\Gamma)\oplus \, _nRK_1(B\Gamma)
\end{equation}
that is compatible with $\Lambda_0$ and with the BCK assembly map $\mu$.

Now, let $P$ be a $K$-datum for  $\underline{K}_0(C^*(\Gamma))$ 
and assume then that we are given $\alpha\in \text{Hom}_{\Lambda_0^{(N)}}(P,\underline{K}_0(\C))$ such that $\alpha_0[1]\geq0$, and so that $\alpha_0(\widetilde{P_0})=0$.  We need to define $\sigma$ and $\pi$ with the properties in the statement.

To do this, for each $n\geq 2$ let $q_n:K_0(C^*(\Gamma;\Z/n))\to _n\!\!K_1(C^*(\Gamma))$ be the projection determined by the choice of splittings in line \eqref{bock dec}.  For $n\geq 2$ with $n|N$, define $Q_n:=q_n(P_n)$.  

As $\alpha_n(\widetilde{P_n})=\{0\}$, each $\alpha_n$ descends to a map $\alpha_n:Q_n\to \Z/n$.  In particular, we get $\alpha_N:Q_n\to \Z/N$.  
On the other hand, 
we have that $Q_N$ identifies with a subgroup of $_NK_1(C^*(\Gamma))$ and therefore of the torsion subgroup $\text{Tor}(K_1(C^*(\Gamma)))$ of $K_1(C^*(\Gamma))$.  Embedding $\Z/N$ into $\Q/\Z$ as multiplies of $1/N$ in the usual way, and using injectivity (i.e.\ divisibility - see for example \cite[Corollary 2.3.2]{Weibel:1995ty}) of the abelian group $\Q/\Z$, we thus see that $\alpha_N$ extends to a map 
$$
\alpha_{tor}:\text{Tor}(K_1(C^*(\Gamma)))\to \Q/\Z.
$$
Let $\Gamma_{ab}$ be the abelianization of $\Gamma$, and define $\sigma_0$ to the composition 
$$
\text{Tor}(\Gamma_{ab})\to \text{Tor}(K_1(C^*(\Gamma))) \xrightarrow{\alpha_{tor}} \Q/\Z
$$
where the first map is induced by the map $\Gamma\to K_1(C^*(\Gamma))$ sending group elements to the canonical unitaries.  Using injectivity again, $\sigma_0$ extends to a map $\sigma_1:\Gamma_{ab}\to \Q/\Z$.  We define $\sigma:\Gamma\to \C$ to be the complex conjugate of the composition 
$$
\Gamma\to \Gamma_{ab}\xrightarrow{\sigma_1}\Q/\Z\to \C
$$ 
where the first map is the canonical quotient, and the last map is $x\mapsto e^{2\pi i x}$.  Define now also $\pi:=(\alpha_0[1]+1)\tau$, where $\tau:\Gamma\to \C$ is the trivial representation. We claim that this $\pi$ and $\sigma$ have the properties in part \eqref{2m02}.

Indeed, $\alpha_0$ agrees with $\alpha[1]\tau_*$ on $P_0\subseteq K_0(C^*(\Gamma))$; moreover, $\sigma_*$ agrees with $\tau_*$ on this summand by Lemma \ref{kes same}.  Similarly for each $n\geq 2$ with $n|N$, $\alpha_n$ agrees with $\alpha[1]\tau_*$ on $\widetilde{P_n}$ by Lemma \ref{kes same} again and compatibility of all these maps with $\Lambda_0$.  Moreover, $\tau_*$ vanishes on the summand $_nK_1(C^*(\Gamma))$ in line \eqref{bock dec} by choice of the splitting in line \eqref{bock dec}.  To complete the proof, it thus suffices to show that for each $n$ with $n\geq 2$ and $n|N$, the map induced by $\alpha_n$ on $Q_n$ agrees with $-\sigma_*$; it moreover suffices to show that it agrees with $\tau_*-\sigma_*$.  

For this, we use commutativity of the diagram 
$$
\xymatrix{ RK_0(B\Gamma;\Z/n) \ar[d]^{\beta_n} \ar[r]^-\mu & K_0(C^*(\Gamma);\Z/n) \ar[rr]^-{\tau_*-\sigma_*}  & & \Z/n \ar[d]  \\
RK_1(B\Gamma) \ar[rrr]^-{\rho_{\sigma}} & & & \R/\Z }
$$
from Theorem \ref{eta inv}.  Using our assumption that the BCK assembly map $\mu$ is an isomorphism (which implies that it is also an isomorphism with all finite coefficients by Lemma \ref{bck fc lem}), and also naturality of $\mu$ with respect to $\beta_n$, it thus suffices to show that we have an equality 
\begin{equation}\label{desid 10}
\alpha_n\circ \mu=\rho_\sigma:\mu^{-1}(Q_n)\to \R/\Z
\end{equation}
for each $n$ (we use here that the map from $RK_0(B\Gamma;\Z/n)$ to the summand $_nRK_1(B\Gamma)$ as in line \eqref{bock dec kh} is $\beta_n$).  

Fix then $n\geq 2$ with $n|N$, and let $z\in \mu^{-1}(Q_n)\subseteq  \,_n RK_1(B\Gamma)$ be an arbitrary class.  Our assumption on torsion classes in $RK_1(B\Gamma)$ implies that $z=[S^1,\C,f]$ for some continuous map $f:S^1\to B\Gamma$.  Proposition \ref{circle eta} thus gives that 
$$
\rho_{\sigma}(z)=q
$$
where $\sigma\circ f_*(1)=e^{-2\pi i q}$.  On the other hand, by definition of $\sigma$,
$$
\sigma(f_*(1))=e^{-2\pi i \alpha_{tor}([f_*(1)])}=e^{-2\pi i \alpha_n[f_*(1)]}
$$
where $[f_*(1)]$ is the class of $f_*(1)$ in $K_1(C^*(\Gamma))$.  Hence to show that the identity in line \eqref{desid 10} holds, it suffices to show that 
$$
[f_*(1)]=\mu[S^1,\C,f].
$$
To see this note that the identification $B\Z=S^1$ and naturality of the assembly map gives us a commutative diagram
$$
\xymatrix{RK_1(S^1) \ar[d]^-{f_*} \ar[r]^-\mu & K_1(C^*(\Z)) \ar[d]^-{f_*} \\
RK_1(B\Gamma) \ar[r]^-\mu & K_1(C^*(\Gamma)) }.
$$
The two groups on the top line are isomorphic to $\Z$, and are generated by $[S^1,\C,\text{id}]$ and $[1]$ respectively.  The left hand vertical map takes $[S^1,\C,\text{id}]$ to $[S^1,\C,f]$ (compare line \eqref{bd func} above), and the right hand vertical map takes $[1]$ to $[f_*(1)]$, so we are done.
\end{proof}

\section{Concrete examples}\label{example sec}

In this section, we use the results of Sections \ref{main sec} and \ref{bc sec} to do explicit computations for some concrete, mainly low-dimensional, examples.

We first start with a fairly general, abstract, theorem.

\begin{theorem}\label{high main}
Let $\Gamma$ be a countable discrete group.   Assume that $\Gamma$ satisfies the following properties:
\begin{enumerate}[(i)]
\item the weak matricial LLP of Definition \ref{llp}
\item property FD from Example \ref{good ex};
\item $C^*(\Gamma)$ satisfies the approximate $K$-homology UCT of Definition \ref{kh auct};
\item the BCK assembly map of line \eqref{max bc} is an isomorphism.
\end{enumerate}

Let $\mathcal{S}$ be the collection of finite symmetric subsets of $\Gamma$.  Let $e\in \{1,2,3,...,\infty\}$ be the exponent of the torsion subgroup of $K_1(C^*(\Gamma))$.  

Then for any $S\in \mathcal{S}$ and $\epsilon>0$ there exist $T\in \mathcal{S}$, $\delta>0$, a finite subset $C$ of $\widetilde{K}_0(C^*(\Gamma))$, and $M\in \N\cap [1,e]$ with the following properties.

Let $\phi:\Gamma\to M_n(\C)$ be a $(T,\delta)$ representation.  Then there is a map $\phi_*$ as in Lemma \ref{to ucp} such that $\phi_*(c)$ is defined for all $c\in C$, and moreover so that if $\phi_*(c)=0$ for all $c\in C$ then there are unitary representations $\theta:\Gamma\to M_k(\C)$ and $\pi:\Gamma\to M_{Mn+k}(\C)$ that factor through a finite quotient of $\Gamma$ and that satisfy 
$$
\|\phi(s)^{\oplus M}\oplus \theta(s)-\pi(s)\|<\epsilon
$$
for all $s\in S$.  
\end{theorem}

\begin{proof}
Let $(S,\epsilon)$ be given, and let $(P,T,\delta)$ be as in the conclusion of Theorem \ref{gp main}, with $N=N(P)$.  With notation as in Proposition \ref{match the 0}, let $C$ be a finite subset of $\widetilde{K}_0(C^*(\Gamma))$ that is large enough to generate $P_0$, and also so that $\rho_n(C)$ generates $\widetilde{P_n}$ for all $n\geq 2$ with $n|N$.  Enlarging $T$ and shrinking $\delta$ is necessary, we may (using Lemma \ref{to ucp}) assume that any $(T,\delta)$-homomorphism gives a well-defined map on the $\Lambda_0^{(N)}$ module generated by $C$ as in Definition \ref{rep phi*}.  Let $M$ be as in Theorem \ref{match the 0}, and note that we may as well assume that $M\leq e$.  

Assume then that $\phi:\Gamma\to M_n(\C)$ is a $(T,\delta)$-representation such that $\phi_*(c)=0$ for all $c\in C$.   Theorem \ref{match the 0} then gives a representation $\sigma:\Gamma\to M_{Mn}(\C)$ such that $\kappa^{S,\epsilon}_P(M\phi_*)=\kappa^{S,\epsilon}_P(\sigma_*)$.  Hence the conclusion of Theorem \ref{gp main} implies that there exists a representation $\theta:\Gamma\to M_{k}(\C)$ with finite image and a unitary $u\in M_{Mn+k}(\C)$ such that 
$$
\|u(\phi(s)^{\oplus M}\oplus \theta(s))u^*-\sigma(s)\oplus \theta(s)\|<\epsilon
$$
for all $s\in S$.  Setting $\pi(g)=u^*(\sigma(g)\oplus \theta(g))u$, we are done.
\end{proof}

\begin{example}
Theorem \ref{high main} applies in particular to torsion-free, residually finite, amenable groups.  For example, let $\Gamma=\Z^n$.  Then $K_1(C^*(\Gamma))$ is torsion-free, so the number $e$ in the theorem is one.  Moreover, we may take $S$ to be a generating set of $\Z^n$, and we may assume that $C$ is a set of generators of $\widetilde{K}_0(C^*(\Gamma))$.  Thus the result says that for any $\epsilon>0$ there is a finite symmetric subset $T$ of $\Z^n$ and $\delta>0$ such that for any $(T,\delta)$ representation $\phi:\Gamma\to M_n(\C)$, if $\phi_*$ vanishes on $C$ then there are unitary representations $\theta:\Gamma\to M_k(\C)$ and $\pi:\Gamma\to M_{n+k}(\C)$ that factor through a finite quotient of $\Gamma$ and that satisfy 
$$
\|\phi(s)\oplus \theta(s)-\pi(s)\|<\epsilon
$$
for all $s\in S$.  This says that $\Z^n$ is $\mathcal{R}_q$-stable conditionally on vanishing of $\phi_*(c)$ for all $c\in C$ in the sense of Definition \ref{ws} and Example \ref{good ex}.

On the other hand, for $n\geq 3$, the result of \cite[Theorem 4.2]{Gong:1998aa} shows that one has to add an auxiliary representation $\theta$ i.e.\  $\Z^n$ is not stable (conditional on vanishing on $C$) for $n\geq 3$.
\end{example}

We now start looking at low-dimensional examples in more detail.  We need some preliminary results that let us reformulate $K$-theoretic conditions in more familiar homological language.  The first result we need comes from Matthey's paper \cite{Matthey:2002aa}.   We warn the reader that Matthey uses ``$K_*$'' for what we call ``$RK_*$'' (see line \eqref{rep k hom} above).

\begin{proposition}[Matthey]\label{bc low}
For any connected CW complex $X$ there are homomorphisms 
$$
\beta_i^X:H_i(X)\to RK_{i\text{ mod }2}(X)
$$
for $i=0,1,2$ that are natural for continuous maps between CW-complexes.  

The map $\beta_0^X$ has the property that $\beta_0^X[pt]_H=[pt]_K$ where $[pt]_H\in H_0(X)$ (respectively, $[pt]_K\in RK_0(X)$) is the class coming from the inclusion of any point.  The map $\beta_1^X$ has the property that if $f:S^1\to X$ describes an element of $\pi_1(X)$ and $[f]\in H_1(B\Gamma)$ is the image of this element under the Hurewicz map $\pi_1(X)\to H_1(X)$, then 
$$
\beta_1^X([f])=[S^1,\C,f]
$$
(here we use notation for $K$-cycles as in Definition \ref{k cycle}).

Moreover, if $X$ is a CW complex of dimension at most three, then the direct sum map 
$$
\beta_0^X\oplus \beta_2^X:H_0(X)\oplus H_2(X)\to RK_0(X)
$$
is an isomorphism, and there is an isomorphism 
$$
\beta_{od}^X:H_1(X)\oplus H_3(X)\to RK_3(X).
$$
that restricts to $\beta_1^X$ on $H_1(X)$.
\end{proposition}

\begin{proof}
The existence and naturality of $\beta_0^X$, plus the fact that $\beta_0^X[pt]_H=[pt]_K$ is contained in \cite[Proposition 2.1 (i)]{Matthey:2002aa}.  Existence and naturality of $\beta_1^X$ and $\beta_2^X$ is contained in \cite[Proposition 3.2]{Matthey:2002aa}.  The explicit description of $\beta_1^X$ in terms of loops is contained in \cite[Proposition 3.4]{Matthey:2002aa}.  Finally, the isomorphism statement follows from \cite[Proposition 2.1 (ii) and Proposition 3.3]{Matthey:2002aa}.
\end{proof}

We next need to recall a result due to Dadarlat \cite[Theorem 1.1]{Dadarlat:2022aa}.  To state it, we recall that $H_2(B\Gamma)$ identifies with the second group homology $H_2(\Gamma)$, essentially by definition.  Choose a (possibly infinite) presentation $\Gamma=\langle S\mid R\rangle$, so there is a short exact sequence 
$$
\{e\}\to \langle \langle R\rangle \rangle \to F_S\to \Gamma\to \{e\}
$$
where $F_S$ is the free group on $S$, and $\langle \langle R\rangle \rangle$ is the normal subgroup of $F_S$ generated by the words in $R$.  Hopf (see \cite[Theorem 5.3]{Brow:1982rt}) showed that 
$$
H_2(\Gamma)\cong\frac{\langle \langle R\rangle \rangle\cap [F,F]}{[F,\langle \langle R\rangle \rangle]},
$$
where $[A,B]$ denotes the subgroup generated by commutators $[a,b]$ with $a\in A$, $b\in B$ and $A$, $B$ subgroups of some ambient group $G$.  Following Dadarlat, an element $c$ of $H_2(B\Gamma)$ can therefore be represented as a product 
$$
\prod_{i=1}^g [a_i,b_i]
$$
where $a_i,b_i\in F_S$, and the image $\prod_{i=1}^g [\overline{a_i},\overline{b_i}]$ in $\Gamma$ is trivial.  Let $\phi:\Gamma\to \mathcal{B}(\C^n)$ be a quasi-representation.  Then as long as $\prod_{i=1}^g [\phi(\overline{a_i}),\phi(\overline{b_i})]$ is within distance one of the identity, the path 
\begin{equation}\label{path for wn}
[0,1]\owns t\mapsto \text{det}\Big((1-t)+t\prod_{i=1}^g [\phi(\overline{a_i}),\phi(\overline{b_i})]\Big)
\end{equation}
passes only through $\C^\times :=\{z\in \C\mid \neq 0\}$, and starts and ends at $1$.  

\begin{definition}\label{dad wn}
Let $\Gamma$ be a countable group, let $c\in H_2(B\Gamma)$ be given by 
$$
\prod_{i=1}^g [a_i,b_i]
$$
as above, and let $\phi:\Gamma\to M_n(\C)_1$ be a quasi-representation such that 
$$
\Bigg\|\prod_{i=1}^g [\phi(\overline{a_i}),\phi(\overline{b_i})]-1 \Bigg\|<1.
$$
We write $w(\phi,c)$ for the winding number of the path in line \eqref{path for wn} above.
\end{definition}

Dadarlat's theorem \cite[Theorem 1.1]{Dadarlat:2022aa} is more general than the statement we give below in that it works without the ucp assumption.  We state it like this just to fit it more cleanly into our framework.

\begin{theorem}[Dadarlat]\label{dad 2d}
Let $\Gamma$ be a countable discrete group, and let $c\in H_2(B\Gamma)$.  Then for any $K$-datum $P$ as in Definition \ref{fin set tot hom} that contains $\mu(\beta_2^{B\Gamma}(c))$ there exists a finite subset $S$ of $\Gamma$ and $\epsilon>0$ such that $(P,S,\epsilon)$ is a $\underline{K}_0$-triple in the sense of Definition \ref{k0 trip}, such that $w(\phi,c)$ makes sense, and with the following property.  

If $\phi:\Gamma\to M_n(\C)$ is a ucp $(S,\epsilon)$-representation in the sense of Definition \ref{quasi rep} and 
$$
\kappa^{(S,\epsilon)}_P:\text{Hom}_{\Lambda_0}(\underline{K}_0^\epsilon(S),\underline{K}_0(\C))\to \text{Hom}_{\Lambda_0^{(N)}}(P,\underline{K}_0(\C))
$$
is as in Lemma \ref{partial kappa} above, then 
$$
\kappa^{(S,\epsilon)}_P(\phi_*)(\mu(\beta_2^{B\Gamma}(c)))=w(c,\phi). \eqno \qed
$$
\end{theorem}

Here is our main general theorem for groups with low-dimensional classifying space.  Recall the class $\mathcal{R}_q$ of representations from Example \ref{good ex}.

\begin{theorem}\label{low main}
Let $\Gamma$ be a finitely generated discrete group.   Assume that $\Gamma$ satisfies the following properties:
\begin{enumerate}[(i)]
\item the weak matricial LLP of Definition \ref{llp}
\item property FD from Example \ref{good ex};
\item $C^*(\Gamma)$ satisfies the approximate $K$-homology UCT of Definition \ref{kh auct};
\item the BCK assembly map of line \eqref{max bc} is an isomorphism.
\item the classifying space $B\Gamma$ is realized as a CW complex of dimension at most three. 
\end{enumerate} 

Let $\mathcal{S}$ be the collection of finite symmetric subsets of $\Gamma$.  Then for any $S\in \mathcal{S}$ and $\epsilon>0$ there exist $T\in \mathcal{S}$, $\delta>0$ and a finite subset $C$ of $H_2(B\Gamma)$ with the following property.

For any $(T,\delta)$ representation $\phi:\Gamma\to M_n(\C)$ the winding number invariants $w(\phi,c)$ make sense for $c\in C$.  Moreover, if $w(\phi,c)=0$ for all $c\in C$, then there are unitary representations $\theta:\Gamma\to M_k(\C)$ and $\pi:\Gamma\to M_{n+k}(\C)$ that factor through a finite quotient of $\Gamma$ such that  
$$
\|\phi(s)\oplus \theta(s)-\pi(s)\|<\epsilon
$$
for all $s\in S$.  
\end{theorem}

\begin{proof}
Let $(S,\epsilon)$ be given, and let $(P,T,\delta)$ be as in the conclusion of Theorem \ref{gp main}.  Let $\widetilde{P_n}$ be as in the conclusion of  Theorem \ref{match the} for $n=0$ or $n\geq 2$ and $n|N$ for the $K$-datum $P$.  Using Proposition \ref{bc low}, the assumption that the BCK assembly map is an isomorphism, and Corollary \ref{red cor}, there is a finite subset $C$ of $H_2(B\Gamma)$ such that $\mu(\beta_2^{B\Gamma}(C))$ is large enough to generate the $\widetilde{P}_0$, and so that its image in each $K_0(C^*(\Gamma);\Z/n)$ generates the groups $\widetilde{P_n}$ for $n\geq 2$.  Expanding $T$ and shrinking $\delta$ if necessary, we may assume that $w(c,\phi)$ is defined for all $c\in C$ whenever $\phi:\Gamma\to M_n(\C)$ is a $(T,\delta)$-representation.

Assume then that $\phi:\Gamma\to M_n(\C)$ is a $(T,\delta)$-representation such that the winding number invariants $w(\phi,c)$ are zero for all $c\in C$.    Theorem \ref{dad 2d} implies that with $\kappa^{(S,\epsilon)}_P$ as in Lemma \ref{partial kappa}, we have that $\kappa^{(S,\epsilon)}_P(\phi_*)$ vanishes on all $\widetilde{P_n}$.  Note that Proposition \ref{bc low} implies that $RK_1(B\Gamma)\cong H_1(B\Gamma)\oplus H_3(B\Gamma)$.  As $B\Gamma$ is at most three-dimensional, $H_3(B\Gamma)$ is torsion free, and so all torsion in $RK_1(B\Gamma)$ comes from $\beta_1^{B\Gamma}(H_1(B\Gamma))$.  Therefore by Proposition \ref{bc low} again, every torsion class in $RK_1(B\Gamma)$ is of the form $[S^1,\C,f]$ for some continuous map $f:S^1\to B\Gamma$.  Hence by Theorem \ref{match the} there is a character $\sigma:\Gamma\to \C$ with torsion image and a representation $\pi_0:\Gamma\to M_{n+1}(\C)$ such that 
$$
\kappa^{(S,\epsilon)}_P(\phi_*)+\sigma_*=\pi_*
$$
as elements of $\text{Hom}_{\Lambda_0^{(N)}}(P,\underline{K_0}(\C))$.  Note that as $\Gamma$ is finitely generated, the image of $\sigma$ is in fact finite.

We may now apply Theorem \ref{gp main} to conclude that there exists a representation $\theta_0:\Gamma\to M_{k_0}(\C)$ with finite image and a unitary $u\in M_{n+k_0+1}(\C)$ such that 
$$
\|u(\phi(s)\oplus \sigma(s)\oplus \theta_0(s))u^*-\pi_0(s)\oplus \theta_0(s)\|<\epsilon
$$
for all $s\in S$.  Setting $k=k_0+1$, $\theta(g)=\sigma(g)\oplus \theta_0(g)$, and $\pi(g)=u^*(\pi_0(g)\oplus \theta_0(g))u$, we are done.  
\end{proof}

\subsection{Free-by-cyclic groups}

In this subsection we discuss \emph{free-by-cyclic} groups, i.e.\ groups of the form $F\rtimes_\phi \Z$ where $F$ is a finitely generated free group and $\Z$ acts on $F_n$ by an automorphism $\phi$.  This class of groups is large and well-studied.  Note that it contains $\Z^2$ and the Klein bottle group $\Z\rtimes \Z$ as special cases, and that these (and the degenerate case $\Gamma=\Z$) are the only amenable examples.

The following result summarizes well-known facts about free-by-cyclic groups.

\begin{proposition}\label{fbc props}
Let $\Gamma=F\rtimes_\phi \Z$ be a free-by-cyclic group where $F$ is a finitely generated free group, and $\phi$ is an automorphism of $F$.  Then $\Gamma$ is UCT, FD, LLP, and the BCK assembly map for $\Gamma$ is an isomorphism.  

Moreover, $\Gamma$ admits a finite CW complex $B\Gamma$ of dimension two for its classifying space, and  $H_2(B\Gamma)$ is a free-abelian group with rank equal to the multiplicity of the eigenvalue one for the map
$$
\phi_*:H_1(F;\C)\to H_1(F;\C)
$$
induced by $\phi$ on homology with complex coefficients.
\end{proposition}

\begin{proof}
The group $\Gamma$ is clearly torsion-free.  It has the Haagerup property by \cite[Example 6.1.6]{Cherix:2001fk}.  As discussed in Remarks \ref{uct rem} and \ref{bc rem} respectively, the Haagerup property implies that $\Gamma$ is UCT and that the BCK assembly map is an isomorphism.  The FD and LLP properties are covered by Remarks \ref{rfd rem} and \ref{llp rem} respectively.  

Let now $BF$ be the classifying space for $F$, which we realize in the usual way as a wedge of $n$ circles where $n$ the rank of $F$.  Let $\phi:BF\to BF$ be a map that induces $\phi$ on fundamental groups (such exists and is unique up to homotopy).  We may assume that $\phi$ is cellular (see \cite[Theorem 4.8]{Hatcher:2002ud}), whence the mapping torus defined by 
$$
M_\phi:=(BF\times [0,1])/((x,0)\sim (1,\phi(x)))
$$
can be given the structure of a finite, two-dimensional CW complex.  It moreover admits a fibration $BF\to M_\phi\to S^1$, and is thus aspherical by the long exact sequence of homotopy groups for a fibration (see for example \cite[Theorem 4.41]{Hatcher:2002ud}).  The fundamental group $\pi_1(M_\phi)$ agrees with $\Gamma$ by the Seifert-van Kampen theorem, and therefore $M_\phi$ is the requited model for $B\Gamma$.  The statement on $H_2$ follows from the long exact sequence 
$$
0\to H_2(M_\phi)\to H_1(BF)\xrightarrow{\text{id}-\phi_*} H_1(BF) \to H_1(M_\phi) \to H_0(BF)\to 0
$$
for the homology of the mapping torus as in \cite[Example 2.48]{Hatcher:2002ud}.  
\end{proof}

The next result follows directly from Proposition \ref{fbc props} and Theorem \ref{low main}.  For the definition of $\mathcal{R}_q$-stable, see Definition \ref{ws} and Example \ref{good ex}.

\begin{theorem}\label{fbc the}
Let $\Gamma=F\rtimes_\phi \Z$ be a free-by cyclic group.  Let $n$ be the rank of $F$, and let $\phi_*:\R^n\to \R^n$ be the map induced by $\phi$ on the first homology group.  Let $m$ be the rank of one as an eigenvalue of $\phi_*$, and let $c_1,...,c_m$ be a corresponding basis for $H_2(B\Gamma)$.  Then $\Gamma$ is $\mathcal{R}_q$-stable, conditional on vanishing of the winding number invariants $w(c_i,\phi)$ for all $i\in \{1,...,m\}$.  

In particular, if $\phi_*$ does not have one as as an eigenvalue, then $\Gamma$ is $\mathcal{R}_q$-stable. \qed
\end{theorem}

We do not know if $\Gamma$ as in Theorem \ref{fbc the} is (conditionally) stable.  This seems an interesting question.  It also seems plausible that there is a more elementary proof of Theorem \ref{fbc the}, but this is not immediately obvious to us.

\subsection{One relator groups}

A finitely generated group $\Gamma$ is called a \emph{one relator group} if it admits a presentation of the form $\Gamma=\langle S \mid r\rangle$ with $S$ finite and where $r$ is a single word.  Important examples include the fundamental group of an orientable surface of genus $g$ 
\begin{equation}\label{os gp}
\Sigma_g:=\Bigg\langle a_1,...,a_g,b_1,...,b_g\mid \prod_{i=1}^g [a_i,b_i]\Bigg\rangle,
\end{equation}
the fundamental group of a non-orientable surface of (non-orientable) genus $g$
\begin{equation}\label{nos gp}
N_g:=\Bigg\langle a_1,...,a_{g}\mid \prod_{i=1}^{g} a_i^2\Bigg\rangle
\end{equation}
and the Baumslag-Solitar groups 
\begin{equation}\label{bs gp}
BS(m,n):=\langle a,b\mid ba^mb^{-1}=a^n\rangle.
\end{equation}
All one-relator groups other than $BS(1,m)$ for some $m$, and (the degenerate case of) cyclic groups are non-amenable: see \cite[page 338]{Ceccherini-Silbertstein:1997aa}.  One-relator groups always satisfy the UCT and the BCK assembly map is always an isomorphism in the torsion-free case as noted in Remarks \ref{uct rem} and \ref{bc rem} respectively.  There are many examples of one-relator groups with the LLP and property FD; whether they all have the LLP is open in general, but there are known examples without property FD: see Remarks \ref{rfd rem} and \ref{llp rem} .

We need to recall some information about homology and classifying spaces of one relator groups: see for example \cite[Section II.4, example 3]{Brow:1982rt} for more information on what follows.  Let $\Gamma=\langle S\mid r\rangle$ be a one-relator group.  Then $\Gamma$ is torsion free if and only if $r$ is not of the form $u^n$ for some word $u$ and $n\geq 2$.  Moreover, if $\Gamma$ is torsion-free then the presentation two-complex (i.e.\ the two-dimensional complex constructed for example in \cite[Corollary 1.28]{Hatcher:2002ud}) is a $B\Gamma$. The homology of $\Gamma$ is (therefore) given by $H_0(B\Gamma)=\Z$, $H_1(B\Gamma)=\Gamma_{ab}$ the abelianization of $\Gamma$, and 
$$
H_2(B\Gamma)=\left\{\begin{array}{ll} \Z, & r\in [\Gamma,\Gamma] \\ 0, & \text{otherwise}\end{array}\right.
$$
(here $[\Gamma,\Gamma]$ is the commutator subgroup of $\Gamma$).  All the higher homology vanishes.  Thus we get the following result from Theorem \ref{low main}.  For the definition of $\mathcal{R}_q$-stable, see Definition \ref{ws} and Example \ref{good ex}.

\begin{theorem}\label{one r}
Let $\Gamma=\langle S\mid r\rangle$ be a one-relator group, and $F_S$ be the free group on the generators of $S$.   Assume moreover that $\Gamma$ is torsion-free, LLP, and FD.  Then the following hold.
\begin{enumerate}[(i)]
\item \label{one r i} If $r$ is not in the commutator subgroup of $F_S$, then $\Gamma$ is $\mathcal{R}_q$-stable.
\item \label{one r ii} If $r$ is in the commutator subgroup of $F_S$, let $w(r,\phi)$ be the associated winding number invariant of a quasi-representation $\phi$ as in Definition \ref{dad wn}.  Then $\Gamma$ is $\mathcal{R}_q$-stable, conditional on vanishing of $w(r,\phi)$.   \qed
\end{enumerate}
\end{theorem}

Let us discuss some specific cases where the groups in question are known to be RFD in more detail.

\begin{example}\label{surf ex}
Perhaps the most interesting case where Theorem \ref{one r} applies is to fundamental groups of surfaces, i.e.\ the groups $\Sigma_g$ and $N_g$ from lines \eqref{os gp} and \eqref{nos gp} above.  There have property FD by \cite[Theorem 2.8 (2)]{Lubotzky:2004xw}, and are LLP by \cite[Remark 4.11]{Fournier-Facio:2026aa}.  Hence Theorem \ref{one r} applies: this is Theorem \ref{intro surf the} from the introduction.  This in some sense gives a complement to the non-stability result of \cite[Theorem 4.8]{Eilers:2018ab}, and the discussion of \cite[Section 4]{Dadarlat:2011kx}.
\end{example} 

\begin{example}
Let us make some comments about the Baumslag-Solitar groups $BS(n,m)$ of line \eqref{bs gp}.  

First, note that these groups are rarely (R)FD: in fact, they are not even residually finite unless $|m|=|n|$, or one of $|m|$ or $|n|$ equals one \cite[Theorem C]{Meskin:1972aa}.  On the other hand, $BS(n,n)$ is isomorphic to a free-by-cyclic group (see \cite[Example 4.23]{Fournier-Facio:2026aa}), and so RFD and LLP as already observed.  On the other hand, the subgroup of $BS(n,-n)=\langle a,b\mid ab^na^{-1}b^n\rangle$ generated by $a^2$ and $b$ has index two, and is isomorphic to $BS(n,n)$.  As both the LLP and property FD pass to finite index supergroups (see \cite[Corollary 4.1]{Fournier-Facio:2026aa} and \cite[Corollary 2.5]{Lubotzky:2004xw} respectively), $BS(n,-n)$ is LLP and has property FD too 

In the case of $BS(n,n)$ the defining relation is a commutator, and we are in the situation of Theorem \ref{one r} part \eqref{one r ii}.  Thus $BS(n,n)$ is $\mathcal{R}_q$-stable, conditional on vanishing of the winding number invariant $w(ab^na^{-1}b^{-n},\cdot)$ of Definition \ref{dad wn}.  On the other hand, in the case of $BS(n,-n)$, the defining relation is not a commutator, and we are in the situation of Theorem \ref{one r} part \eqref{one r i}.  Thus $BS(n,-n)$ is $\mathcal{R}_q$-stable.  These results partly generalize the case of the fundamental groups of the torus and the Klein bottle covered in Theorem \ref{torus and klein}, which are $BS(1,1)$ and $BS(1,-1)$ respectively.  

Note that the case of $BS(n,n)$ was discussed in \cite[Theorem 4.10]{Eilers:2018ab} where it is shown that these groups are not stable due to the non-vanishing of the associated winding number invariant; our result here provides a sort of converse.  Compare also \cite[Question 2 in Section 6]{Eilers:2018ab}.

The groups $BS(1,n)$ are amenable and residually finite, so our results also apply in that case: as the defining relation is not in the commutator subgroup, these groups are $\mathcal{R}_q$-stable.
\end{example}

\begin{remark}\label{o r summary}
In general, many one-relator groups are known to be virtually free-by-cyclic, i.e.\ to have a finite index free-by-cyclic subgroup; this is also true generically in a reasonable sense.  This was studied in detail by Kielak and Linton \cite{Kielak:2024aa}, and known results are summarized in \cite[Sections 4.3 and 6.2]{Fournier-Facio:2026aa}.  As both the LLP and property FD pass to finite index supergroups (see \cite[Corollary 4.1]{Fournier-Facio:2026aa} and \cite[Corollary 2.5]{Lubotzky:2004xw} respectively), torsion-free one-relator groups in this class satisfy the hypotheses of Theorem \ref{one r}.
\end{remark}


\subsection{Three manifold groups}

The following theorem specializes to Theorem \ref{intro 3man the} from the introduction.  For the definition of $\mathcal{R}_q$-stable, see Definition \ref{ws} and Example \ref{good ex}.

\begin{theorem}\label{three man}
Let $\Gamma$ be the fundamental group of a compact, connected, aspherical three-manifold $M$, possibly with boundary.  Write $H_2(M)=F\oplus T$ where $F$ is a finitely generated free abelian group and $T$ is a torsion group.  Let $c_1,...,c_m$ be a basis for $F$.  Then $\Gamma$ is $\mathcal{R}_q$-stable, conditional on vanishing of $w(c_i,\phi)$ for all $i$.  

In particular, if $H_2(M)$ is torsion, then $\Gamma$ is $\mathcal{R}_q$-stable. 
\end{theorem}

\begin{proof}
We first note that $\Gamma$ has the LLP by \cite[Example 4.10]{Fournier-Facio:2026aa} and property FD by \cite[Theorem 6.1]{Fournier-Facio:2026aa}.  It is moreover a-T-menable (see \cite[Remark 4.4 and Example 4.10]{Fournier-Facio:2026aa}) whence UCT, and the BCK assembly map is an isomorphism.  As $M$ is aspherical, we may take $M$ to be a classifying space for $\Gamma$; it admits the structure of a CW complex of dimension three.


The result now follows from Theorem \ref{low main}.
\end{proof}

The following remark gives more detail on the form $H_2(M)$; it was inspired by a comment of Francesco Fournier-Facio.

\begin{remark}
Assume for simplicity that $M$ is a closed, connected, aspherical three manifold with fundamental group $\Gamma$.  Let $r$ be the rank of the abelianization of $\Gamma$, i.e.\ the first Betti number of $M$.  As in \cite[Exercise 24 on page 259]{Hatcher:2002ud}\footnote{In the orientable case, this can be deduced by Poincar\'{e} duality; in the non-orientable case, it can be deduced from \cite[Corollary 2.38]{Hatcher:2002ud}, Poincar\'{e} duality for $\Z/2$-coefficients, and the universal coefficient theorems for homology and cohomology.}, we have $H_2(M)\cong \Z^r$ if $M$ is orientable, and $H_2(M)\cong \Z^{r-1}\oplus \Z/2$ if $M$ is non-orientable.  Hence if $M$ is closed, the number $m$ appearing in Theorem \ref{three man} is the first Betti number of $M$ in the orientable case, and one less than this in the non-orientable case.  In particular, $\Gamma$ is $\mathcal{R}_q$-stable if\footnote{And only if: this follows for example from the main result of \cite{Dadarlat:384aa}.} either $M$ is orientable and has  first Betti number zero, or is non-orientable and has first Betti number one.
\end{remark}

\begin{example}\label{flat ex}
Perhaps the simplest class of groups covered by Theorem \ref{three man} are the fundamental groups of closed flat manifolds.  For example, this class includes $\Z^3$, the fundamental group of the $3$-torus.  In general, there are ten such groups: see for example \cite[Theorems 3.5.5 and 3.5.9]{Wolf:2011aa} for the classification.  

Fundamental groups of closed flat manifolds are amenable and residually finite, so are LLP and FD by classical results (compare Remarks \ref{rfd rem} and \ref{llp rem}).   Theorem \ref{three man} implies that the corresponding groups are $\mathcal{R}_q$-stable, conditional on vanishing of the winding number invariants coming from their two-homology.  In particular, if the two-homology is torsion, then the corresponding group is $\mathcal{R}_q$-stable: this is the case for the fundamental group of the (non-orientable) manifold called $\mathcal{B}_4$ in \cite[Theorem 3.5.9]{Wolf:2011aa} as already pointed out by Dadarlat \cite[Example 2.7]{Dadarlat:384aa}.  It also follows for the (orientable) manifold $M$ called $\mathcal{G}_6$\footnote{Also called the \emph{Hantzche-Wendt manifold}.} in \cite[Theorems 3.5.5]{Wolf:2011aa}: indeed, by \cite[Corollary 3.5.10]{Wolf:2011aa}, we have $H_1(M)=\Z/4\oplus \Z/4$, whence $H_2(M)$ is torsion by the universal coefficient theorem for (co)homology, and Poincar\'{e} duality.

Note, however, that fundamental groups of closed flat three manifolds are not stable, as pointed out in \cite[Theorem 3.13]{Eilers:2018ab}.  In the case of $\Gamma=\Z^3$, this group is not even stable conditional on vanishing of the relevant winding number invariants: this follows from \cite[Theorem 4.2]{Gong:1998aa}.  It seems likely that this phenomenon -- conditional $\mathcal{R}_q$-stability holding, but conditional stability failing -- is an issue for many other groups with three-dimensional classifying space.

We do not claim much originality in this flat manifold case -- the result of Theorem \ref{three man} can be deduced from the same methods as \cite[Theorem 1.5]{Dadarlat:384aa} for amenable groups -- but hope that summarizing the main points above might be useful.
\end{example}

We now give a concrete example of a hyperbolic three-manifold that Theorem \ref{three man} shows to be $\mathcal{R}_d$-stable.  Thanks to Asaf Hadari for explaining some of the ideas in the following example to me.

\begin{example}\label{3 man wk}
Let $M$ be a closed, aspherical three manifold with $H_2(M)$ torsion.  Then Theorem \ref{three man} implies that $\Gamma=\pi_1(M)$ is $\mathcal{R}_q$-stable.  To show that fairly direct computations are possible, here is a concrete construction of a three-manifold fibered over the circle, with non-amenable fundamental group, and such that $H_2(M)$ is torsion. 

Let $S$ be a closed orientable surface of genus $g=2$ as pictured
\begin{center}
\includegraphics[width=3cm]{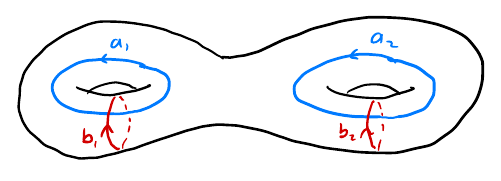}
\end{center}
with the curves $\{a_1,b_1,a_2,b_2\}$ as an ordered basis for $H_1(S)\cong \Z^4$.  Note that the standard symplectic form (compare \cite[Section 6.1.2]{Farb:2012aa}) with respect to this basis is given by 
$$
J=\begin{psmallmatrix} 0 & 1 & 0 & 0 \\ -1 & 0 & 0 & 0 \\ 0 & 0 & 0 & 1 \\ 0 & 0 & -1 & 0 \end{psmallmatrix}
$$ 
Now, let $B=\begin{psmallmatrix} 5 & 3 \\ 3 & 2\end{psmallmatrix}$ and $A=\begin{psmallmatrix} B & 0 \\ 0 & B\end{psmallmatrix}$, which is an element of $M_4(\Z)$ with determinant one such that  $A^TJA=J$, i.e.\ $A$ is an element of the symplectic group $Sp(4,\Z)$.  Using \cite[Theorem 6.4 (and Section 2.1 for the definition)]{Farb:2012aa}, there is an orientation-preserving diffeomorphism $\phi_0:S\to S$ that induces the map $A$ on homology.  Let moreover $\rho:S\to S$ be the left-right reflection around a plane separating the two holes as in the picture above, so that $\rho$ induces the following map on homology
$$
R=\begin{psmallmatrix} 0 & 0 & -1 & 0 \\ 0 & 0 & 0 & 1 \\ -1 & 0 & 0 & 0 \\ 0 & 1 & 0 & 0 \end{psmallmatrix}.
$$
Let $\phi=\rho\circ \phi_0$, and let 
$$
M_\phi:=S\times [0,1] / ((x,0)\sim (\phi(x),1))
$$
be the mapping torus of $\phi$, which is a closed three manifold with (non-amenable) fundamental group $\Gamma=\pi_1(S)\rtimes \Z$, where the action of $\Z$ is via $\phi$.  Note that $M_\phi$ is aspherical by the same argument using the long exact sequence for a fibration that we used in the proof of Proposition \ref{fbc props}.  We claim that $H_2(M_\phi)=\Z/2$, whence $\Gamma$ will be $\mathcal{R}_q$-stable.

Indeed, as in \cite[Example 2.8]{Hatcher:2002ud}, the relevant part of the long exact homology sequence for a mapping torus looks like
$$
\cdots\to H_2(S)\xrightarrow{\text{id}-\phi_*} H_2(S)\to H_2(M_\phi) \to H_1(S) \xrightarrow{\text{id}-\phi_*} H_1(S)\to \cdots.
$$
As $S$ is orientable, $H_2(S)\cong \Z$; as $\phi$ reverses orientation, the map $\phi_*:H_2(S)\to H_2(S)$ is multiplication by $-1$.  On the other hand, the map $\phi_*:H_1(S)\to H_1(S)$ is given by the matrix $RA$; one readily checks that this matrix does not have one as an eigenvalue whence $\text{id}-\phi_*:H_1(S)\to H_1(S)$ is injective.  The result follows.
\end{example}

There are many other examples of three-manifolds with torsion, or even trivial, second homology to which Theorem \ref{three man} applies.  For example, an interesting infinite family of closed hyperbolic three manifolds (which are aspherical as the universal cover is hyperbolic $3$-space) with the same homology as the $3$-sphere are given in \cite{Brock:2015aa}.  Note that the fundamental groups of these manifolds are perfect and non-amenable; Theorem \ref{three man} implies that their fundamental groups are all $\mathcal{R}_q$-stable.

\bibliography{Generalbib}

\begin{thebibliography}{100}

\bibitem{Antonini:2014aa}
P.~Antonini, S.~Azzali, and G.~Skandalis.
\newblock Flat bundles, von {N}eumann algebras, and ${K}$-theory with
  $\mathbb{R}/\mathbb{Z}$-coefficients.
\newblock {\em J. K-theory}, 13:275--303, 2014.

\bibitem{Aparicio:2020aa}
M.~P.~G. Aparicio, P.~Julg, and A.~Valette.
\newblock The {B}aum-{C}onnes conjecture: an extended survey.
\newblock In {\em Advances in Noncommutative Geometry}, pages 127--244.
  Springer, 2020.

\bibitem{Arveson:1969aa}
W.~Arveson.
\newblock Subalgebras of ${C^*}$-algebras.
\newblock {\em Acta Math.}, 123:141--224, 1969.

\bibitem{Arveson:1977aa}
W.~Arveson.
\newblock Notes on extensions of ${C^*}$-algebras.
\newblock {\em Duke Math. J.}, 44(2):329--355, 1977.

\bibitem{Arzhantseva:2014aa}
G.~Arzhantseva.
\newblock Asymoptotic approximations of finitely generated groups.
\newblock In {\em Extended abstracts Fall 2012 - automorphisms of free grouos},
  volume~1 of {\em Trends Math. Res. Persepct. CRM Barc.}, pages 7--15, 2014.

\bibitem{Arzhantseva:2015aa}
G.~Arzhantseva and L.~P\u{a}unescu.
\newblock Almost commuting permutations are near commuting permutations.
\newblock {\em J. Funct. Anal.}, 269:745--757, 2015.

\bibitem{Atiyah:1961uq}
M.~Atiyah.
\newblock Characters and cohomology of finite groups.
\newblock {\em Publ. Math. Inst. Hautes \'{E}tudes Sci.}, 9:23--64, 1961.

\bibitem{Atiyah:1967aa}
M.~Atiyah.
\newblock Algebraic topology and elliptic operators.
\newblock {\em Comm. Pure Appl. Math.}, XX:237--249, 1967.

\bibitem{Atiyah:1976th}
M.~Atiyah.
\newblock Elliptic operators, discrete groups and von {N}eumann algebras.
\newblock {\em Asterisque}, 32-33:43--72, 1976.

\bibitem{Atiyah:1975ab}
M.~Atiyah, V.~K. Patodi, and I.~Singer.
\newblock Spectral asymmetry and {R}iemannian geometry. {I}.
\newblock {\em Math. Proc. Cambridge Philos. Soc.}, 77:43--69, 1975.

\bibitem{Atiyah:1975aa}
M.~Atiyah, V.~K. Patodi, and I.~Singer.
\newblock Spectral asymmetry and {R}iemannian geometry. {II}.
\newblock {\em Math. Proc. Cambridge Philos. Soc.}, 78:405--432, 1975.

\bibitem{Atiyah:1976aa}
M.~Atiyah, V.~K. Patodi, and I.~Singer.
\newblock Spectral asymmetry and {R}iemannian geometry. {III}.
\newblock {\em Math. Proc. Cambridge Philos. Soc.}, 79:71--99, 1976.

\bibitem{Atiyah:1969aa}
M.~Atiyah and G.~Segal.
\newblock Equivariant ${K}$-theory and completion.
\newblock {\em J. Differential Geom.}, 3:1--18, 1969.

\bibitem{Atiyah:1968ul}
M.~Atiyah and I.~Singer.
\newblock The index of elliptic operators {I}.
\newblock {\em Ann. of Math.}, 87(3):484--530, 1968.

\bibitem{Bader:2023aa}
U.~Bader, A.~Lubotzky, R.~Sauer, and S.~Weinberger.
\newblock Stability and instability of lattices in semisimple groups.
\newblock arXiv:2303.08943, 2023.

\bibitem{Baum:1994pr}
P.~Baum, A.~Connes, and N.~Higson.
\newblock Classifying space for proper actions and ${K}$-theory of group
  ${C}^*$-algebras.
\newblock {\em Contemporary Mathematics}, 167:241--291, 1994.

\bibitem{Baum:1980pt}
P.~Baum and R.~G. Douglas.
\newblock ${K}$-homology and index theory.
\newblock In {\em Operator algebras and applications, Part I}, volume~38 of
  {\em Proc. Sympos. Pure Math.}, pages 117--173. American Mathematical
  Society, 1980.

\bibitem{Baum:2009hq}
P.~Baum, N.~Higson, and T.~Schick.
\newblock A geometric description of equivariant ${K}$-homology for proper
  actions.
\newblock In {\em Quanta of maths}, volume~11 of {\em Clay Math. Proc.}, pages
  1--22. American Mathematical Society, Providence, RI, 2010.

\bibitem{Becker:2020aa}
O.~Becker and A.~Lubotzky.
\newblock Group stability and property ({T}).
\newblock {\em J. Funct. Anal.}, 278:108298, 2020.

\bibitem{Beguin:2999aa}
C.~B\'{e}guin, H.~Bettaieb, and A.~Valette.
\newblock ${K}$-theory for ${C^*}$-algebras of one-relator groups.
\newblock {\em K-Theory}, 16:277--298, 1999.

\bibitem{Bekka:1999kx}
B.~Bekka.
\newblock On the full ${C}^*$-algebras of arithmetic groups and the congruence
  subgroup problem.
\newblock {\em Forum Math.}, 11(6):705--715, 1999.

\bibitem{Blackadar:1998yq}
B.~Blackadar.
\newblock {\em K-Theory for Operator Algebras}.
\newblock Cambridge University Press, second edition, 1998.

\bibitem{Blackadar:2006eq}
B.~Blackadar.
\newblock {\em Operator Algebras: Theory of ${C}^*$-Algebras and {V}on
  {N}eumann Algebras}.
\newblock Springer, 2006.

\bibitem{Brock:2015aa}
J.~Brock and N.~Dunfield.
\newblock Injectivity radii of hyperbolic integer homology $3$-spheres.
\newblock {\em Geom. Topol.}, 19:497--523, 2015.

\bibitem{Brow:1982rt}
K.~S. Brown.
\newblock {\em Cohomology of groups}.
\newblock Number~87 in Graduate Texts in Mathematics. Springer, 1982.

\bibitem{Brown:1984rx}
L.~G. Brown.
\newblock The universal coefficient theorem for {E}xt and quasidiagonality.
\newblock In {\em Operator algebras and group representations}, volume~I of
  {\em Monogr. Stud. Math. 17}, pages 60--64. Pitman, 1984.

\bibitem{Brown:1977qa}
L.~G. Brown, R.~G. Douglas, and P.~Fillmore.
\newblock Extensions of ${C^*}$-algebras and ${K}$-homology.
\newblock {\em Ann. of Math.}, 105:265--324, 1977.

\bibitem{Brown:2008qy}
N.~Brown and N.~Ozawa.
\newblock {\em ${C}^*$-Algebras and Finite-Dimensional Approximations},
  volume~88 of {\em Graduate Studies in Mathematics}.
\newblock American Mathematical Society, 2008.

\bibitem{Buss:2014aa}
A.~Buss, S.~Echterhoff, and R.~Willett.
\newblock Exotic crossed produts and the {B}aum-{C}onnes conjecture.
\newblock {\em J. Reine Angew. Math.}, 740:111--159, 2018.

\bibitem{Buss:2018nm}
A.~Buss, S.~Echterhoff, and R.~Willett.
\newblock The maximal injective crossed product.
\newblock {\em Ergodic Theory Dynam. Systems}, 40:2995--3014, 2020.

\bibitem{Buss:2022aa}
A.~Buss, S.~Echterhoff, and R.~Willett.
\newblock Amenability and weak containment for actions of locally compact
  groups on ${C^*}$-algebras.
\newblock To appear, Mem. Amer. Math. Soc., 2022.

\bibitem{Carrion:2020aa}
J.~Carri\'{o}n, J.~Gabe, C.~Schafhauser, A.~Tikuisis, and S.~White.
\newblock Classification of $*$-homomorphisms {I}: the simple nuclear case.
\newblock arXiv:2307.06480, 2020.

\bibitem{Ceccherini-Silbertstein:1997aa}
T.~Ceccherini-Silbertstein and R.~Grigorchuk.
\newblock Amenability and growth of one-relator groups.
\newblock {\em Enseign. Math. (2)}, 43:337--354, 1997.

\bibitem{Cherix:2001fk}
P.-A. Cherix, M.~Cowling, P.~Jolissaint, P.~Julg, and A.~Valette.
\newblock {\em Groups with the {H}aagerup Property ({G}romov's
  a-{T}-menability)}, volume 197 of {\em Progress in mathematics}.
\newblock Birkh\"{a}user, 2001.

\bibitem{Chiffre:2018ds}
M.~D. Chiffre, L.~Glebsky, A.~Lubotzky, and A.~Thom.
\newblock Stability, cohomology vanishing, and non-approximable groups.
\newblock {\em Forum Math. Sigma}, 8:e18, 2020.

\bibitem{Choi:1976aa}
M.-D. Choi and E.~G. Effros.
\newblock The completely positive lifting problem for ${C^*}$-algebras.
\newblock {\em Ann. of Math.}, 104:585--609, 1976.

\bibitem{Connes:1990if}
A.~Connes and N.~Higson.
\newblock D\'{e}formations, morphismes asymptotiques et ${K}$-th\'{e}orie
  bivariante.
\newblock {\em C. R. Acad. Sci. Paris S\'{e}r. I Math.}, 311:101--106, 1990.

\bibitem{Cuntz:1977aa}
J.~Cuntz.
\newblock Simple ${C^*}$-algebras generated by isometries.
\newblock {\em Comm. Math. Phys.}, 57(2):173--185, 1977.

\bibitem{Cuntz:1983jx}
J.~Cuntz.
\newblock ${K}$-theoretic amenability for discrete groups.
\newblock {\em J. Reine Angew. Math.}, 344:180--195, 1983.

\bibitem{Dadarlat:1999ab}
M.~Dadarlat.
\newblock Approximate unitary equivalence and the topology of {E}xt$({A},{B})$.
\newblock In {\em $C^*$-algebras (M\"{u}nster 1999)}, pages 42--60. Springer,
  2000.

\bibitem{Dadarlat:2005aa}
M.~Dadarlat.
\newblock On the topology of the {K}asparov groups and its applications.
\newblock {\em J. Funct. Anal.}, 228(2):394--418, 2005.

\bibitem{Dadarlat:2011kx}
M.~Dadarlat.
\newblock Group quasi-representations and index theory.
\newblock {\em J. Topol. Anal.}, 4(3):297--319, 2012.

\bibitem{Dadarlat:2011uq}
M.~Dadarlat.
\newblock Group quasi-representations and almost flat bundles.
\newblock {\em J. Noncommut. Geom.}, 8(1):163--178, 2014.

\bibitem{Dadarlat:384aa}
M.~Dadarlat.
\newblock Obstructions to matrix stability of discrete groups.
\newblock {\em Adv. Math.}, 384:107722, 2021.

\bibitem{Dadarlat:2022aa}
M.~Dadarlat.
\newblock Quasi-representations of groups and two-homology.
\newblock {\em Comm. Math. Phys.}, 393:267--277, 2022.

\bibitem{Dadarlat:2001aa}
M.~Dadarlat and S.~Eilers.
\newblock Asymptotic unitary equivalence in ${KK}$-theory.
\newblock {\em ${K}$-theory}, 23(4):305--322, 2001.

\bibitem{Dadarlat:2002aa}
M.~Dadarlat and S.~Eilers.
\newblock On the classification of nuclear ${C^*}$-algebras.
\newblock {\em Proc. London Math. Soc.}, 85(3):168--210, 2002.

\bibitem{Dadarlat:1996aa}
M.~Dadarlat and T.~Loring.
\newblock A universal multicoefficient theorem for the {K}asparov groups.
\newblock {\em Duke Math. J.}, 84(2):355--377, 1996.

\bibitem{Dadarlat:2012aa}
M.~Dadarlat and R.~Meyer.
\newblock E-theory for ${C^*}$-algebras over topological spaces.
\newblock {\em J. Funct. Anal.}, 263(1):216--247, 2012.

\bibitem{Dadarlat:2016qc}
M.~Dadarlat, R.~Willett, and J.~Wu.
\newblock Localization ${C^*}$-algebras and ${K}$-theoretic duality.
\newblock {\em Ann. ${K}$-theory}, 3:615--630, 2018.

\bibitem{Salle:2023aa}
M.~de~la Salle.
\newblock Alg\'{e}bres de von {N}eumann, produits tensoriels, corr\'{e}lations
  quantique et calculabilit\'{e}.
\newblock {\em S\'{e}minaire Bourbaki}, 75:1203, 2023.

\bibitem{Dogon:2023aa}
A.~Dogon.
\newblock Flexible {H}ilbert--{S}chmidt stability versus hyperlinearity for
  property ({T}) groups.
\newblock {\em Math. Z.}, 305(58), 2023.

\bibitem{Dold:1962aa}
A.~Dold.
\newblock Relations between ordinary and extra-ordinary homology.
\newblock In {\em Colloquium on algebraic topology}, pages 2--9. Aarhus
  Universitet, 1962.

\bibitem{Effros:1985aa}
E.~G. Effros and U.~Haagerup.
\newblock Lifting problems and local reflexivity for {$C^*$}-algebras.
\newblock {\em Duke Math. J.}, 52(1):103--128, 1985.

\bibitem{Eilers:1999aa}
S.~Eilers and T.~Loring.
\newblock Computing contingencies for stable relations.
\newblock {\em Int. J. Math.}, 10(301-326), 1999.

\bibitem{Eilers:1998aa}
S.~Eilers, T.~Loring, and G.~Pedersen.
\newblock Stability of anticommutation relations: an application of
  noncommutative {CW} complexes.
\newblock {\em J. Reine Angew. Math.}, 499:101--143, 1998.

\bibitem{Eilers:1999ln}
S.~Eilers, T.~Loring, and G.~K. Pedersen.
\newblock Morphisms of extensions of ${C^*}$-algebras: pushing forward the
  busby invariant.
\newblock {\em Adv. Math.}, 147:74--109, 1999.

\bibitem{Eilers:2018ab}
S.~Eilers, T.~Shulman, and A.~S\o{}renson.
\newblock ${C^*}$-stability of discrete groups.
\newblock {\em Adv. Math.}, 373:107324, 2020.

\bibitem{Elliott:1995dq}
G.~Elliott.
\newblock The classification problem for amenable ${C}^*$-algebras.
\newblock In {\em Proceedings of the International Congress of Mathematicians},
  volume 1,2, pages 922--932, 1995.

\bibitem{Elliott:1995aa}
G.~A. Elliott and M.~R\o{}rdam.
\newblock Classification of certain infinite simple ${C^*}$-algebra, {II}.
\newblock {\em Comment. Math. Helvetici}, 70:615--638, 1995.

\bibitem{Enders:2023aa}
D.~Enders and T.~Shulman.
\newblock Alost commuting matrices, cohomology, and dimension.
\newblock {\em Annales Scientifiques de l'\'{E}cole Normale Sup\'{e}rieure},
  56(6):1653--1683, 2023.

\bibitem{Enders:2024aa}
D.~Enders and T.~Shulman.
\newblock On the (local) lifting property.
\newblock arXiv:2403.12224v2, 2024.

\bibitem{Exel:1991aa}
R.~Exel and T.~Loring.
\newblock Invariants of almost commuting unitaries.
\newblock {\em J. Funct. Anal.}, 95:364--376, 1991.

\bibitem{Exel:1992aa}
R.~Exel and T.~Loring.
\newblock Finite-dimensional representations of free product ${C^*}$-algebras.
\newblock {\em Internat. J. Math.}, 03(04):469--476, 1992.

\bibitem{Farb:2012aa}
B.~Farb and D.~Margalit.
\newblock {\em A primer on mapping class groups}.
\newblock Princeton University Press, 2012.

\bibitem{Fell:1962aa}
J.~M.~G. Fell.
\newblock Weak containment and induced representations of groups.
\newblock {\em Canad. J. Math.}, 14:237--268, 1962.

\bibitem{Fournier-Facio:2026aa}
F.~Fournier-Facio and R.~Willett.
\newblock The local lifting property, property {FD} and stability of
  approximate representations.
\newblock arXiv:2603.18456, 2026.

\bibitem{Gabe:2019ws}
J.~Gabe.
\newblock Classification of $\mathcal{O}_\infty$-stable ${C^*}$-algebras.
\newblock {\em Mem. Amer. Math. Soc.}, 293(1461), 2024.

\bibitem{Glebe:2024aa}
F.~Glebe.
\newblock A constructive proof that many groups with non-torsion $2$-cohomology
  are not matricially stable.
\newblock arXiv 2204.10354. To appear, Groups, Geom. Dynam., 2022.

\bibitem{Glebsky:2010aa}
L.~Glebsky.
\newblock Almost commuting matrices with respect to normalized
  {H}ilbert-{S}chmidt norm.
\newblock arXiv:1001.3082, 2010.

\bibitem{Gong:1998aa}
G.~Gong and H.~Lin.
\newblock Almost multiplicative morphisms and almost commuting matrices.
\newblock {\em J. Operator Theory}, 40:217--275, 1998.

\bibitem{Guentner:2014bh}
E.~Guentner, R.~Willett, and G.~Yu.
\newblock Finite dynamical complexity and controlled operator {K}-theory.
\newblock arXiv:1609.02093; to appear, Ast\'{e}risque, 2016.

\bibitem{Haagerup:1979rq}
U.~Haagerup.
\newblock An example of a non nuclear ${C}^*$-algebra which has the metric
  approximation property.
\newblock {\em Invent. Math.}, 50:289--293, 1979.

\bibitem{Hadwin:1981aa}
D.~Hadwin.
\newblock Nonseparable approximate equivalence.
\newblock {\em Trans. Amer. Math. Soc.}, 266(1):203--231, 1981.

\bibitem{Hadwin:2018aa}
D.~Hadwin and T.~Shulman.
\newblock Stability of group relations under small hilbert-schmidt
  perturbations.
\newblock {\em J. Funct. Anal.}, 275:761--2, 2018.

\bibitem{Hatcher:2002ud}
A.~Hatcher.
\newblock {\em Algebraic Topology}.
\newblock Cambridge University Press, 2002.

\bibitem{Higson:1998qd}
N.~Higson.
\newblock The {B}aum-{C}onnes conjecture.
\newblock In {\em Proceedings of the International Congress of Mathematicians},
  volume~{II}, pages 637--646, 1998.

\bibitem{Higson:2004la}
N.~Higson and E.~Guentner.
\newblock Group ${C}^*$-algebras and ${K}$-theory.
\newblock In {\em Noncommutative Geometry}, number 1831 in Springer Lecture
  Notes, pages 137--252. Springer, 2004.

\bibitem{Higson:2001eb}
N.~Higson and G.~Kasparov.
\newblock ${E}$-theory and ${KK}$-theory for groups which act properly and
  isometrically on {H}ilbert space.
\newblock {\em Invent. Math.}, 144:23--74, 2001.

\bibitem{Higson:2000bs}
N.~Higson and J.~Roe.
\newblock {\em Analytic ${K}$-homology}.
\newblock Oxford University Press, 2000.

\bibitem{Higson:2008qb}
N.~Higson and J.~Roe.
\newblock ${K}$-homology, assembly and rigidity theorems for relative
  eta-invariants.
\newblock {\em Pure Appl. Math. Q.}, 6(2):555--601, 2010.

\bibitem{Ioana:2020aa}
A.~Ioana, P.~Spaas, and M.~Wiersma.
\newblock Cohomological obstriuctions to lifting properties for full
  ${C^*}$-algebras of property ({T}) groups.
\newblock {\em Geom. Funct. Anal.}, 30:1402--1438, 2020.

\bibitem{Jakob:1998aa}
M.~Jakob.
\newblock A bordism-type description of homology.
\newblock {\em Manuscripta Math.}, 96:67--80, 1998.

\bibitem{Junge:1995aa}
M.~Junge and G.~Pisier.
\newblock Bilinear forms on exact operator spaces and ${B(H)\otimes B(H)}$.
\newblock {\em Geom. Funct. Anal.}, 5(2):329--363, 1995.

\bibitem{Kasparov:1975ht}
G.~Kasparov.
\newblock Topological invariants of elliptic operators {I}: ${K}$-homology.
\newblock {\em Math. USSR-Izv.}, 9(4):751--792, 1975.

\bibitem{Kasparov:1988dw}
G.~Kasparov.
\newblock Equivariant ${KK}$-theory and the {N}ovikov conjecture.
\newblock {\em Invent. Math.}, 91(1):147--201, 1988.

\bibitem{Kasparov:2016aa}
G.~Kasparov.
\newblock Elliptic and transversally elliptic index theory from the viewpoint
  of ${KK}$-theory.
\newblock {\em J. Noncommut. Geom.}, 10(4):1303--1378, 2016.

\bibitem{Kazhdan:1982aa}
D.~Kazhdan.
\newblock On $\epsilon$-representations.
\newblock {\em Israel J. Math.}, 43(4):315--323, 1982.

\bibitem{Keswani:1999aa}
N.~Keswani.
\newblock Geometric ${K}$-homology and controlled paths.
\newblock {\em New York J. Math.}, 5:53--81, 1999.

\bibitem{Keswani:2000aa}
N.~Keswani.
\newblock Relative eta-invariants and ${C^*}$-algebra ${K}$-theory.
\newblock {\em Topology}, 39:957--983, 2000.

\bibitem{Kielak:2024aa}
D.~Kielak and M.~Linton.
\newblock Virtually free-by-cyclic groups.
\newblock {\em Geom. Funct. Anal.}, 34:1580--1608, 2024.

\bibitem{Kirchberg:1993aa}
E.~Kirchberg.
\newblock On non-semisplit extensions, tensor products, and exactness of group
  ${C^*}$-algebras.
\newblock {\em Invent. Math.}, 112:449--489, 1993.

\bibitem{Kirchberg:1994aa}
E.~Kirchberg.
\newblock Commutants of unitaries in {UHF} algebras and functorial properties
  of exactness.
\newblock {\em J. Reine Angew. Math.}, 452:39--77, 1994.

\bibitem{Kirchberg:1994ez}
E.~Kirchberg.
\newblock Discrete groups with {K}azhdan's property ({T}) and the factorization
  are residually finite.
\newblock {\em Math. Ann.}, 299(1):551--563, 1994.

\bibitem{Kirchberg-ICM}
E.~Kirchberg.
\newblock Exact {${\rm C}^*$}-algebras, tensor products, and the classification
  of purely infinite algebras.
\newblock In {\em Proceedings of the {I}nternational {C}ongress of
  {M}athematicians, {V}ol.\ 1, 2 ({Z}\"urich, 1994)}, pages 943--954.
  Birkh\"auser, Basel, 1995.

\bibitem{Kubota:2019aa}
Y.~Kubota.
\newblock The relative {M}ischenko-{F}omenko higher index and almost flat
  bundles {II}: almost flat index pairing.
\newblock {\em J. Noncommut. Geom.}, 14(2):353--382, 2022.

\bibitem{Lafforgue:2002zl}
V.~Lafforgue.
\newblock Banach ${KK}$-theory and the {B}aum-{C}onnes conjecture.
\newblock In {\em Proceedings of the International Congress of Mathematicians},
  volume III, pages 796--812, 2002.

\bibitem{Lafforgue:2002qm}
V.~Lafforgue.
\newblock K-th\'{e}orie bivariante pour les alg\`{e}bres de {B}anach et
  conjecture de {B}aum-{C}onnes.
\newblock {\em Invent. Math.}, 149(1):1--95, 2002.

\bibitem{Lance:1973aa}
E.~C. Lance.
\newblock On nuclear ${C^*}$-algebras.
\newblock {\em J. Funct. Anal.}, 12:157--176, 1973.

\bibitem{Jr.:1989uq}
H.~B. Lawson and M.-L. Michelsohn.
\newblock {\em Spin Geometry}.
\newblock Princeton University Press, 1989.

\bibitem{Lazarovich:2019aa}
N.~Lazarovich, A.~Levit, and Y.~Minsky.
\newblock Surface groups are flexibly stable.
\newblock arXiv:1901.07182. To appear, J. Eur. Math. Soc., 2019.

\bibitem{Li:2012aa}
Q.~Li and J.~Shen.
\newblock A note on unital full amalgamated free products of {RFD}
  ${C^*}$-algebras.
\newblock {\em Illinois J. Math}, 56:647--659, 2012.

\bibitem{Lin:2002aa}
H.~Lin.
\newblock Stable approximate unitary equivalence of homomorphisms.
\newblock {\em J. Operator Theory}, 47(2):343--378, 2002.

\bibitem{Lin:2005aa}
H.~Lin.
\newblock An approximate universal coefficient theorem.
\newblock {\em Trans. Amer. Math. Soc.}, 357(8):3375--3405, 2005.

\bibitem{Lin:1995aa}
H.~Lin and N.~C. Phillips.
\newblock Almost multiplicative morphisms and the {C}untz algebra ${O_2}$.
\newblock {\em Internat. J. Math.}, 6(4):625--643, 1995.

\bibitem{Loring:1985ud}
T.~Loring.
\newblock {\em The torus and noncommutative topology}.
\newblock PhD thesis, University of California, Berkeley, 1985.

\bibitem{Loring:1997aa}
T.~Loring.
\newblock {\em Lifting solutions to perturbing problems in ${C^*}$-algebras}.
\newblock American Mathematical Society, 1997.

\bibitem{Lott:1999aa}
J.~Lott.
\newblock Delocalized ${L}^2$-invariants.
\newblock {\em J. Funct. Anal.}, 169:1--31, 1999.

\bibitem{Lubotzky:2004xw}
A.~Lubotzky and Y.~Shalom.
\newblock Finite representations in the unitary dual and {R}amanujan groups.
\newblock In {\em Discrete geometric analysis}, number 347 in Contemporary
  Mathematics, pages 173--189. American Mathematical Society, 2004.

\bibitem{Luck:2002lk}
W.~L\"{u}ck.
\newblock The relationship between the {B}aum-{C}onnes conjecture and the trace
  conjecture.
\newblock {\em Invent. Math.}, 149:123--152, 2002.

\bibitem{Matthey:2002aa}
M.~Matthey.
\newblock Mapping the homology of a group to the ${K}$-theory of its
  ${C^*}$-algebra.
\newblock {\em Illinois J. Math}, 46(3):953--977, 2002.

\bibitem{Matthey:2008aa}
M.~Matthey, H.~Oyono-Oyono, and W.~Pitsch.
\newblock Homotopy invariance of higher signatures and $3$-manifold groups.
\newblock {\em Bull. Soc. Math. France}, 136:1--25, 2008.

\bibitem{Meskin:1972aa}
S.~Meskin.
\newblock Nonresidually finite one-relator groups.
\newblock {\em Trans. Amer. Math. Soc.}, 164:105--114, 1972.

\bibitem{Milnor:1956aa}
J.~Milnor.
\newblock Construction of universal bundles, {II}.
\newblock {\em Ann. of Math.}, 63(3):430--436, 1956.

\bibitem{Naimark:1943aa}
M.~A. Naimark.
\newblock Positive definite operator functions on a commutative group.
\newblock {\em Bull. (Izv.) Acad. Sci. URSS (Ser. Math)}, 7:237--244, 1943.

\bibitem{Oyono-Oyono:2011fk}
H.~Oyono-Oyono and G.~Yu.
\newblock On quantitative operator ${K}$-theory.
\newblock {\em Ann. Inst. Fourier (Grenoble)}, 65(2):605--674, 2015.

\bibitem{Ozawa:2000th}
N.~Ozawa.
\newblock Amenable actions and exactness for discrete groups.
\newblock {\em C. R. Acad. Sci. Paris S{\'e}r. I Math.}, 330:691--695, 2000.

\bibitem{Ozawa:2001aa}
N.~Ozawa.
\newblock On the lifting property for universal ${C^*}$-algebras of operator
  spaces.
\newblock {\em J. Operator Theory}, 46:579--591, 2001.

\bibitem{Ozawa:2004ab}
N.~Ozawa.
\newblock About the {QWEP} conjecture.
\newblock {\em Internat. J. Math.}, 15:501--530, 2004.

\bibitem{Ozawa:2004aa}
N.~Ozawa.
\newblock There is no separable universal {$II_1$}-factor.
\newblock {\em Proc. Amer. Math. Soc.}, 132(2):487--490, 2004.

\bibitem{Paulsen:2003ib}
V.~Paulsen.
\newblock {\em Completely Bounded Maps and Operator Algebras}.
\newblock Cambridge University Press, 2003.

\bibitem{Phillips-documenta}
N.~C. Phillips.
\newblock A classification theorem for nuclear purely infinite simple
  {$C^*$}-algebras.
\newblock {\em Doc. Math.}, 5:49--114 (electronic), 2000.

\bibitem{Pisier:2020aa}
G.~Pisier.
\newblock {\em Tensor products of ${C^*}$-algebras and operator spaces: the
  {C}onnes-{K}irchberg problem}.
\newblock Cambridge University Press, 2020.

\bibitem{Ramras:2012fk}
D.~A. Ramras, R.~Willett, and G.~Yu.
\newblock A finite dimensional approach to the strong {N}ovikov conjecture.
\newblock {\em Algebr. Geom. Topol.}, 13:2283--1316, 2013.

\bibitem{Raven:2004aa}
J.~Raven.
\newblock {\em An equivariant bivariant {C}hern character}.
\newblock PhD thesis, The Pennsylvania State University, 2004.

\bibitem{Roe:1998ad}
J.~Roe.
\newblock {\em Elliptic Operators, topology and asymptotic methods}.
\newblock Chapman and Hall, second edition, 1998.

\bibitem{Rordam:1995aa}
M.~R\o{}rdam.
\newblock Classification of certain infinite simple ${C^*}$-algebras.
\newblock {\em J. Funct. Anal.}, 131:415--458, 1995.

\bibitem{Rordam:2002cs}
M.~R\o{}rdam.
\newblock {\em Classification of Nuclear ${C^*}$-algebras}.
\newblock Springer, 2002.

\bibitem{Rordam:2000mz}
M.~R\o{}rdam, F.~Larsen, and N.~Laustsen.
\newblock {\em An Introduction to $K$-Theory for $C^*$-Algebras}.
\newblock Cambridge University Press, 2000.

\bibitem{Rosenberg:1987bh}
J.~Rosenberg and C.~Schochet.
\newblock The {K}\"{u}nneth theorem and the universal coefficient theorem for
  {K}asparov's generalized ${K}$-functor.
\newblock {\em Duke Math. J.}, 55(2):431--474, 1987.

\bibitem{Schochet:1984ab}
C.~Schochet.
\newblock Topological methods for ${C^*}$-algebras {IV}: mod $p$ homology.
\newblock {\em Pacific J. Math.}, 114(2):447--468, 1984.

\bibitem{Schochet:2002aa}
C.~Schochet.
\newblock The fine structure of the {K}asparov groups {II}: topologizing the
  {UCT}.
\newblock {\em J. Funct. Anal.}, 194:263--287, 2002.

\bibitem{Segal:1970aa}
G.~Segal.
\newblock Fredholm complexes.
\newblock {\em Quart. J. Math.}, 21(385-402), 1970.

\bibitem{Shulman:2022aa}
T.~Shulman.
\newblock Central amalgamation of groups and the {RFD} property.
\newblock {\em Adv. Math.}, 394:108131, 2022.

\bibitem{Skandalis:1988rr}
G.~Skandalis.
\newblock Une notion de nucl\'{e}arit\'{e} en ${K}$-th\'{e}orie (d'apr\`{e}s
  {J}. {C}untz).
\newblock {\em K-Theory}, 1(6):549--573, 1988.

\bibitem{Skandalis:1991aa}
G.~Skandalis.
\newblock Le bifoncteur de {K}asparov n'est pas exact.
\newblock {\em C. R. Acad. Sci. Paris}, 313:939--941, 1991.

\bibitem{Thom:2010aa}
A.~Thom.
\newblock Examples of hyperlinear groups without the factorization property.
\newblock {\em Groups, Geom. Dynam.}, 4:195--208, 2010.

\bibitem{Thom:2018aa}
A.~Thom.
\newblock Finitary approximations of groups and their applications.
\newblock In {\em Proceedings of the {I}nternational {C}ongress of
  {M}athematicians ({R}io de {J}aneiro, 2018)}, volume III, pages 1779--1799,
  2018.

\bibitem{Tu:1999aa}
J.-L. Tu.
\newblock The {B}aum-{C}onnes conjecture and discrete group actions on trees.
\newblock {\em K-theory}, 17:303--318, 1999.

\bibitem{Tu:1999bq}
J.-L. Tu.
\newblock La conjecture de {B}aum-{C}onnes pour les feuilletages moyennables.
\newblock {\em ${K}$-theory}, 17:215--264, 1999.

\bibitem{Voiculescu:1983km}
D.-V. Voiculescu.
\newblock Asymptotically commuting finite rank unitary operators without
  commuting approximants.
\newblock {\em Acta Sci. Math. (Szeged)}, 45:429--431, 1983.

\bibitem{Wang:2023ab}
H.~Wang, C.~Zhang, and D.~Zhou.
\newblock Localization ${C^*}$-algebras and index pairing.
\newblock {\em J. Homtopy Relat. Struct.}, 18:1--22, 2023.

\bibitem{Weibel:1995ty}
C.~Weibel.
\newblock {\em An Introduction to Homological Algebra}, volume~38 of {\em
  Cambridge studies in advanced mathematics}.
\newblock Cambridge University Press, 1995.

\bibitem{White:2022aa}
S.~White.
\newblock Abstract classification theorems for amenable ${C^*}$-algebras.
\newblock In {\em Proceedings of the International Congress of Mathematicians
  2022}, volume~4 of {\em EMS Press}, pages 3314--3338, 2023.

\bibitem{Willett:2020ab}
R.~Willett.
\newblock Bott periodicity and almost commuting matrices.
\newblock {\em Contemp. Math.}, 749:379--388, 2020.

\bibitem{Willett:2020aa}
R.~Willett and G.~Yu.
\newblock Controlled ${KK}$-theory and a {M}ilnor exact sequence.
\newblock arXiv:2011.10906. To appear, Doc. Math., 2020.

\bibitem{Willett:2010ay}
R.~Willett and G.~Yu.
\newblock {\em Higher Index Theory}.
\newblock Cambridge University Press, 2020.

\bibitem{Willett:2021te}
R.~Willett and G.~Yu.
\newblock The {U}niversal {C}oefficient {T}heorem for ${C^*}$-algebras with
  finite complexity.
\newblock {\em Mem. Eur. Math. Soc.}, 8, 2024.

\bibitem{Wolf:2011aa}
J.~Wolf.
\newblock {\em Spaces of Constant Curvature}.
\newblock American Mathematical Society, 6th edition, 2011.

\bibitem{Yu:1998wj}
G.~Yu.
\newblock The {N}ovikov conjecture for groups with finite asymptotic dimension.
\newblock {\em Ann. of Math.}, 147(2):325--355, 1998.

\bibitem{Yu:200ve}
G.~Yu.
\newblock The coarse {B}aum-{C}onnes conjecture for spaces which admit a
  uniform embedding into {H}ilbert space.
\newblock {\em Invent. Math.}, 139(1):201--240, 2000.

\end{thebibliography}

\end{document}